\documentclass[11pt]{article}
 \usepackage[title]{appendix}
\usepackage[letterpaper, total={6in, 9in}]{geometry}
\usepackage{amstext,amsthm,amsmath,amsfonts,amssymb,amscd,latexsym,txfonts,stmaryrd,enumerate}
\usepackage[all,cmtip]{xy}
\usepackage{appendix}
 \xymatrixcolsep{8mm} 
 \usepackage{tikz}

\newcommand{\Ann}{\operatorname{Ann}}

\newcommand{\rank}{\operatorname{rank}}
\newcommand{\soc}{\operatorname{soc}}

\newcommand{\C}{{\mathbb C}}

\newcommand{\Z}{{\mathbb Z}}

\newcommand{\Q}{{\mathbb Q}}

\newcommand{\F}{{\mathbb F}}
\newcommand{\fm}{{\mathfrak m}}

\usepackage[colorlinks=true, pdfstartview=FitV, linkcolor=blue, citecolor=blue, urlcolor=blue]{hyperref}

\renewcommand{\P}{{\mathbb P}}

\theoremstyle{theorem}
\newtheorem{theorem}{Theorem}[section]
\newtheorem{proposition}[theorem]{Proposition}
\newtheorem{lemma}[theorem]{Lemma}
\newtheorem{corollary}[theorem]{Corollary}
\newtheorem{problem}[theorem]{Problem}

\newtheorem*{problem*}{Problem}
\newtheorem*{question*}{Question}
\newtheorem*{claim*}{Claim}

\newtheorem*{conjecture*}{Conjecture}
\newtheorem{conjecture}[theorem]{Conjecture}
\newtheorem*{fact*}{Fact}

\theoremstyle{definition}
\newtheorem{definition}[theorem]{Definition}

\newtheorem{construction}[theorem]{Construction}
\newtheorem{example}[theorem]{Example}
\newtheorem*{ack}{Acknowledgment}

\newtheorem{remark}[theorem]{Remark}
\newtheorem*{remark*}{Remark}
\newtheorem*{remarks*}{Remarks}

\def\ns{\footnotesize \it}
\def\cha{\mathrm{char}\,}

\newcommand\blfootnote[1]{%
  \begingroup
  \renewcommand\thefootnote{}\footnote{#1}%
  \addtocounter{footnote}{-1}%
  \endgroup
}

	\title{Cohomological Blow Ups of Graded Artinian Gorenstein Algebras Along Surjective Maps }
	\author{
Anthony Iarrobino\\[.05in]
{\ns Department of Mathematics, Northeastern University, Boston, MA 02115,
USA.
}\\{\ns a.iarrobino@northeastern.edu}\\[.2in] Pedro Macias Marques\\[.05in]
{\ns Departamento de Matem\'{a}tica, Escola de Ci\^{e}ncias e Tecnologia, Centro de Investiga\c{c}\~{a}o}\\[-.05in]
{\ns  em Matem\'{a}tica e Aplica\c{c}\~{o}es, Instituto de Investiga\c{c}\~{a}o e Forma\c{c}\~{a}o Avan\c{c}ada,}\\[-.05in]
{\ns Universidade de \'{E}vora, Rua Rom\~{a}o Ramalho, 59, P--7000--671 \'{E}vora, Portugal}\\
{\ns pmm@uevora.pt}
\\[.08in] Chris McDaniel\\[.05in]
{\ns Department of Mathematics, Endicott College, 376 Hale St
Beverly, MA 01915, USA.}\\ {\ns cmcdanie@endicott.edu}
\\[0.2in]
Alexandra Seceleanu\\[0.05in]
{\ns Department of Mathematics, University of Nebraska-Lincoln,}\\
{\ns 203 Avery Hall,  Lincoln NE 68588, USA.} \\ {\ns aseceleanu@unl.edu}\\
[0.2 in] Junzo Watanabe\\[-0.1in]
 \\{\ns Department of Mathematics, Tokai University,}\\
 {\ns Hiratsuka, Kanagawa 259-1292 Japan}\\
{\ns watanabe.junzo@tokai-u.jp}
}

\date{}

\begin{document}
		\maketitle
	\begin{abstract}	 We introduce the cohomological blow up of a graded Artinian Gorenstein (AG) algebra along a surjective map, which we term BUG (Blow Up Gorenstein) for short. This is intended to translate to an algebraic context the cohomology ring of a blow up of a projective manifold along a projective submanifold. We show, among other things, that a BUG is a connected sum, that it is the general fiber in a flat family of algebras, and that it preserves the strong Lefschetz property.  We also show that standard graded compressed algebras are rarely BUGs, and  we classify those BUGs that are complete intersections.  We have included many examples throughout this manuscript.
	\end{abstract}
	

	\tableofcontents
	
\section{Introduction}
Given two graded Artinian Gorenstein (AG) algebras $A$ and $T$ over a field $\F$, of socle degrees $d>k$, respectively, and a surjective algebra map between them $\pi\colon A\rightarrow T$, we construct a new graded AG algebra $\hat{A}$ of socle degree $d$ called the cohomological blow up of of $A$ along $\pi$; we shall sometimes refer to $\hat{A}$ as BUG, short for Blow Up Gorenstein.  As the name suggests, our construction is based on the blow up operation in complex geometry, and particularly its effect on the (singular) cohomology rings of the spaces involved.  The purpose of the present paper is to extend this blow up operation on cohomology rings to the more general class of graded AG algebras, to study its interactions with other familiar algebraic constructs, and to draw parallels to corresponding geometric ones.
Our cohomological blow ups are very different from the much-studied blow-up algebras (Rees algebras and related rings), which correspond to the coordinate ring (not the cohomology ring) of the blow-up variety \cite{Vas}.

Generally speaking, graded Artinian Gorenstein algebras are algebraic analogues of cohomology algebras (in even degrees) of smooth, compact, even-dimensional manifolds, e.g. complex manifolds or symplectic manifolds.  For compact complex (or symplectic) manifolds $\pi\colon Y\hookrightarrow X$ of (complex) dimensions $k<d$, respectively, the blow-up of $X$ along $\pi$ is another compact complex $d$-dimensional manifold $\tilde{X}$ obtained from $X$ by removing $Y$ and gluing in its place a codimension one submanifold $\tilde{Y}$ called the exceptional divisor, which can be realized as the projectivization of the normal bundle $\mathcal{N}_{Y/X}$.  The cohomology algebras (over $\Q$) of these spaces $Y$, $X$, $\tilde{Y}$ and $\tilde{X}$ satisfy the following three algebraic properties.  For notational convenience, let $A=H^{2\bullet}(X)$, $T=H^{2\bullet}(Y)$, $\tilde{A}=H^{2\bullet}(\tilde{X})$, and $\tilde{T}=H^{2\bullet}(\tilde{Y})$, graded so that $A_i=H^{2i}(X)$, and so on.

First if $\beta\colon \tilde{X}\rightarrow X$ is the blow down map, $\beta_0\colon \tilde{Y}\rightarrow Y$ its restriction to the exceptional divisor, and $\hat{\pi}\colon \tilde{Y}\hookrightarrow \tilde{X}$ the natural inclusion map, then the obvious commutative diagram of spaces induces a commutative diagram of cohomology algebras:
\begin{align}
\label{eq:cMaps}
\xymatrix{X & \tilde{X}\ar[l]_-{\beta}\\
Y\ar[u]^-{\pi} & \tilde{Y}\ar[u]_-{\tilde{\pi}}\ar[l]^-{\beta_0}\\} & \hspace{3cm} \Rightarrow &  \xymatrix{A\ar[r]^-{\beta^*}\ar[d]_-{\pi^*} & \tilde{A}\ar[d]^-{\hat{\pi}^*}\\
	T\ar[r]_-{\beta_0^*} & \tilde{T}\\}
\end{align}

Second, as a projective space bundle over $Y$, the cohomology algebra of the exceptional divisor $\tilde{Y}$ is a free extension over the cohomology of $Y$ generated by the Euler class $\xi=e\left(\mathcal{N}_{\tilde{Y}/\tilde{X}}\right)\in H^2(\tilde{Y})$ of the normal bundle $\mathcal{N}_{\tilde{Y}/\tilde{X}}$, in symbols
\begin{equation}
\label{eq:HYtild}
\tilde{T}\cong \frac{T[\xi]}{\left(\xi^n+t_\xi^{n-1}+\cdots+t_{n-1}\xi+t_n\right)}
\end{equation} 
where $t_i=(-1)^i\cdot c_i(Y)\in T_i$ are the Chern classes of the normal bundle $\mathcal{N}_{Y/X}$, the top one of which is the Euler class $c_n(Y)=e(\mathcal{N}_{Y/X})\in T_n$, ($n=d-k=$ codimension of $Y$ in $X$). 

Third, the induced map $\beta^*$ is injective and fits into a short exact sequence of $A$-modules: 
\begin{equation}
\label{eq:SES}
\xymatrix{0\ar[r] & A\ar[r]^-{\beta^*} & \tilde{A}\ar[r]^-{\tilde{\pi}^*} & \tilde{T}/\beta_0^*\left(T\right)\ar[r] & 0.}
\end{equation}

Equations \eqref{eq:HYtild} and \eqref{eq:SES} imply that the Hilbert function of the cohomology of the blow up is 
\begin{equation}
\label{eq:Hilbert}
H(\tilde{A})=H(A)+H(T)[1]+\cdots+H(T)[n-1].
\end{equation}
In fact, one can show that Equations \eqref{eq:cMaps}, \eqref{eq:HYtild}, and \eqref{eq:SES} uniquely determine the algebra structure of $\tilde{A}$.  Moreover, if the restriction map $\pi^*\colon A\rightarrow T$ is surjective, one can show that the cohomology algebra of the blow up is given by 
\begin{equation}
\label{eq:cohomologybu}
\tilde{A}=\frac{A[\xi]}{\left(\xi\cdot K, \  \xi^n+a_1\xi^{n-1}+\cdots+a_{n-1}+a_n\right)}
\end{equation} 
where $K\subset A$ is the kernel of $\pi^*$, $\pi^*(a_i)=t_i$, and $a_n=(-1)^n\cdot \tau(Y)\in A_n$ where $\tau(Y)$ is the Thom class of the normal bundle $\mathcal{N}_{Y/X}$.  
See \cite{Gitler,GH,Ha,Huybrechts,Keel,McDuff} for further details on geometric and topological aspects of blow ups, specifically \cite[Proposition 6.4]{GH}, \cite[Proposition 2.4]{McDuff}, and \cite[Theorem 3.11]{Gitler}.  See also \cite{MS} for further details on characteristic classes.

The novelty of this paper is to show that Equations \eqref{eq:cMaps}, \eqref{eq:HYtild}, \eqref{eq:SES}, and \eqref{eq:cohomologybu} can be extended to arbitrary AG algebras  over any field $\F$ to define a new construction on these rings called the cohomological blow up.  Specifically, given any surjective degree preserving map of graded AG algebras $\pi\colon A\rightarrow T$, we define (Definition \ref{def:blowup}) a \emph{cohomological blow up of $A$ along $\pi$} as in Equation \eqref{eq:cohomologybu}, leaving $a_1,\ldots,a_{n-1}$ as free parameters and setting $a_n$ to be the algebraic analogue of the Thom class.

The algebraic analogue of the Thom class for a map of AG algebras $\pi\colon A\rightarrow T$ is defined as a certain annihilator of the kernel of $\pi$, determined by choices of socle generators, called orientations, of $A$ and $T$ (Definition \ref{def:Thom}).  Assuming that $\pi$ is surjective, and $A=R/I$ where $I\subset R$ is a homogeneous ideal in a graded polynomial ring, an alternative characterization of the Thom class is an element $\tau\in R$ for which $T= R/(I\colon\tau)$.  In terms of Macaulay duality, if $R$ is acting on its dual divided power algebra $Q$ by contraction and if $F\in Q$ and $G\in Q$ are dual generators of $A=R/\Ann(F)$ and $T/\Ann(G)$, then $\tau\in R$ is a contraction operator satisfying $\tau\circ F=G$.  Each of these interpretations of the Thom class leads to distinct, but equivalent descriptions of the cohomological blow up.  We frame these in terms of three general constructions on AG algebras, starting from the given data consisting of two oriented AG algebras $A$ and $T$ of socle degrees $d$ and $k$ with $n=d-k\geq 2$ and a surjective degree preserving algebra map between them $\pi\colon A\rightarrow T$.

Our first construction (Construction \ref{con:hat}) fixes an indeterminate $\xi$, then chooses any homogeneous monic polynomial $f_A(\xi)\in A[\xi]$ (we use the subscript to indicate the coefficient algebra) and uses Equation \eqref{eq:cohomologybu} to construct a new algebra
$$f_A(\xi)=\xi^n+a_1\xi^{n-1}+\cdots+a_n \ \ \ \ \Rightarrow \ \ \ \ \hat{A}=\frac{A[\xi]}{\left(\xi\cdot K, \hat{f}(\xi)\right)}.$$
where $K$ is the kernel of $\pi$, and $a_i\in A_i$ are any homogenous elements.  We show that $\hat{A}$ is an AG algebra if and only if $a_n=\lambda\cdot\tau$ is a non-zero scalar multiple of the Thom class $\tau\in A_n$ of $\pi$ (Theorem \ref{thm:hatAGor}).  In this case we replace the ``hat'' with a ``tilde'' and call the resulting algebra $\tilde{A}$ the \emph{cohomological blow up} of $A$ along $\pi$ with parameters $(\pi(a_1),\ldots,\pi(a_{n-1}),\lambda)$, and we call $A$ the \emph{cohomological blow down} of $\tilde{A}$ (Definition \ref{def:blowup}).  We show that $\tilde{A}$ is essentially characterized by analogues of Equations \eqref{eq:cMaps}, \eqref{eq:HYtild} and \eqref{eq:SES} above (Theorem \ref{lem:buchar}).  We then use a standard result in commutative algebra to show that $\tilde{A}$ is isomorphic to the general fiber in a flat family of algebras (Theorem \ref{thm:flatfam}).  Using this flat family, we further show that if both $A$ and $T$ have the strong Lefschetz property (SLP) and $\F$ is infinite of characteristic either zero or sufficiently large, then $\tilde{A}$ must also have SLP (Theorem \ref{thm:SLPblowup}).  This result has parallels in complex geometry; if $Y\subset X$ are projective manifolds, then the blow up $\tilde{Y}\subset\tilde{X}$ is also projective, and projective manifolds always satisfy the hard Lefschetz theorem, see e.g. \cite{GH,Ha,Huybrechts}.  We remark that the essential result of Theorem \ref{thm:SLPblowup}, that blow ups preserve SLP, seems to be well known in various other contexts and guises, e.g. \cite{ADH,BL,Cav,Karu,McMullen,Timorin}; on the other hand, to our knowledge, the generality of Theorem \ref{thm:SLPblowup} is new, and we hope it will be a welcome reference for this useful result.  We further show that over an infinite field, if $A$ and $T$ both have the weak Lefschetz property (WLP) and the difference between their socle degrees is two, then $\tilde{A}$ also has WLP (Theorem \ref{prop:codim2}).  We give examples that show general cohomological blow ups may not preserve WLP (Example \ref{8.7ex}), and blow downs may not preserve SLP or WLP in general (Example \ref{ex:failingSLP}).  In a sequel \cite{MMSW} we plan to investigate the behavior of the related Hodge-Riemann bilinear relations (HRR) in relation with cohomological blow ups and blow downs. 

Our second construction (Construction \ref{con:hatMD}) starts with Macaulay dual generators $F\in Q_d$ and $G\in Q_k$ for $A=R/\Ann(F)$ and $T=R/\Ann(G)$, and a contraction operator $\tau\in R_n$ such that $\tau\circ F=G$.  Then one again fixes an indeterminate $\xi$ and chooses a monic homogeneous polynomial $f_R(\xi)\in R[\xi]$, and from this constructs a new dual generator $\hat{F}\in Q[\Xi]_d$ and its associated AG algebra 
$$f_R(\xi)=\xi^n+r_1\xi^{n-1}+\cdots+r_n \ \ \ \ \Rightarrow \ \ \ \ \hat{A}_{MD}=\frac{R[\xi]}{\Ann(\hat{F})}$$
where the $r_i\in R_i$ are homogeneous of degree $i$, $1\leq i\leq n$.  Applying the projection map $R[\xi]\rightarrow A[\xi]$ to $f_R(\xi)\mapsto f_A(\xi)=\overline{f_R(\xi)}$, one can also construct $\hat{A}$ from above, but in general the algebras $\hat{A}$ and $\hat{A}_{MD}$ will be non-isomorphic.  However we show that they are isomorphic exactly when they are equal to some cohomological blow up of $A$ along $\pi$, and we give necessary and sufficient conditions for this to occur (Theorem \ref{thm:BUMD1}); in this case we shall replace the ``hat'' with ``tilde'', drop the subscript, and simply write $\tilde{A}$.  We use this result to show that every cohomological blow up along a surjective map can be realized as a connected sum in the sense of \cite{IMS} (Theorem \ref{thm:BUCS}).  We show that one can also obtain the cohomological blow down as a connected sum (Theorem \ref{thm:bdcs}).  This latter result implies that every AG algebra $A$ can be realized as a connected sum (over some $\tilde{T}\neq\F$) by blowing up then blowing down.

Our third construction (Construction \ref{con:ideal}) starts with presentations $A=R/I$ and $T=R/(I:\tau)$.  One then fixes an indeterminate $\xi$ and chooses a monic homogeneous polynomial $f_R(\xi)\in R[\xi]$ and constructs a new ideal $\hat{I}\subset R[\xi]$ by
$$f_R(\xi)=\xi^n+r_1\xi^{n-1}+\cdots+r_n \ \ \ \ \Rightarrow \ \ \ \ \hat{I}=I+\xi\cdot (I:\tau)+(f_R(\xi))$$
where the $r_i\in R_i$ are homogeneous of degree $i$, $1\leq i\leq n$.  It is clear from the definition that $R[\xi]/\hat{I}\cong \hat{A}$ from Construction \ref{con:hat} above, and hence we deduce that $\hat{I}$ is a Gorenstein ideal (i.e. irreducible and primary to the maximal ideal of $R[\xi]$) if and only if $a_n=\lambda\cdot\tau$, a non-zero scalar multiple of the Thom class of $\pi$.  In this case we replace the ``hat'' with ``tilde'' and call $\tilde{I}$ a \emph{cohomological blow up ideal} because the quotient algebra $R[\xi]/\tilde{I}$ is isomorphic to the cohomological blow up algebra $\tilde{A}$ defined above.  We compute a minimal generating set of $\tilde{I}$ in terms of minimal generating sets of $I$ and $(I:\tau)$ (Theorem \ref{thm:mingens}).  Using this result in conjunction with the notion of exact pairs of zero divisors borrowed from \cite{HS, KSV}, we classify BUGs which are complete intersections (Theorem \ref{thm:CI}).  We also use cohomological blow ups to prove a special case of Watanabe's Bold Conjecture (WBC) \cite{WatanabeTalk}, namely that any standard graded Artinian complete intersection cut out by products of linear and/or quadratic forms can be embedded as a subalgebra of another standard graded Artinian complete intersection cut out by quadratic forms and of the same socle degree (Theorem \ref{thm:WBC}); see \cite{McDBC} for a different proof of a related result, and see \cite{McDSmith} for a proof of other cases of WBC.

It is interesting to note that our cohomological blow up along a surjective map preserves the standard grading, i.e. if $A$ and $T$ are standard graded, then the cohomological blow up $\tilde{A}$ is also standard graded (Remark \ref{rem:stdg}).  However we show by examples that neither the cohomological blow down, nor, what one might call, the cohomological blow up along a non-surjective map, necessarily preserve standard grading at all (Example \ref{ex:nonst}, Remark \ref{rem:nostdgrad}).  We say that a standard graded AG algebra $\tilde{A}$ has a \emph{standard BUG structure} if it is a cohomological blow up of a standard graded AG algebra $A$ along some surjective map $\pi$.  Using Equation \eqref{eq:Hilbert} above, we deduce that some Hilbert functions cannot occur with standard BUG structures, and we call such Hilbert functions \emph{inaccessible}.  Specifically we show that the compressed Hilbert functions of \cite{Iarrobino} with embedding dimension at least three and socle degree at least six are always inaccessible (Theorem \ref{thm:compressed}).  We plan for a sequel \cite{IMMSW} with further results in this direction, including a complete description of cohomological blowup algebras of codimension two - where all AG Hilbert functions are accessible - and a study of inaccessible AG Hilbert functions in codimension at least three.

This paper is organized as follows.  In Section \ref{sec:Thom} we give a description of the algebraic analogue of the Thom class of a map of graded AG algebras.  In Section \ref{sec:CBU} we give Construction \ref{con:hat} and necessary and sufficient conditions for it to yield a Gorenstein algebra.  Then we define the cohomological blow up algebra, and give a description of its Hilbert function.  In Section \ref{sec:MD} we discuss Macaulay duality, we give Construction \ref{con:hatMD}, and we give necessary and sufficient conditions for it to yield the cohomological blow up algebra.  In Section \ref{sec:connsum}, we show that every cohomological blow-up algebra is a connected sum.  We also describe the blow-down in terms of connected sums, and show that every AG algebra has a connected sum decomposition, obtained by blowing up then blowing down.  In Section~\ref{sec:Ideal}, we give Construction \ref{con:ideal}, introduce the cohomological blow ups ideal, and compute its minimal generating set.  In Section \ref{compsec} we define inaccessible Hilbert functions, and show that almost all compressed AG algebras are inaccessible.  In Section \ref{Lefsec} we show that cohomological blowing up preserves SLP. In Section \ref{geomsec} we discuss further the connection of our work with geometry and other areas, give some geometrically motivated examples, and suggest problems for future work.  In this paper, we have made a special effort to include a wide array of examples, many of which were found as counter-examples to conjectures or questions that arose in our discussions and preparations of this manuscript, and we hope the reader will find them useful.  In the appendix we have included a list of all examples and brief descriptions of them for the readers convenience.

\section{AG Algebras, Orientations, and Thom Classes}
\label{sec:Thom}
Let $\F$ be any field.  A \emph{graded AG algebra} means a commutative $\Z_{\geq 0}$-graded connected Artinian Gorenstein $\F$-algebra; in particular if $A$ is a graded AG algebra of socle degree $d$, then $A=\bigoplus_{i=0}^dA_i$ with $A_0=\F$ and $\operatorname{socle}(A)=A_d\cong \F$.  Since most algebraic objects in this paper are graded, we may sometimes drop the adjective and refer simply to an AG algebra.  Unless explicitly stated otherwise (as Remark \ref{rem:lambda}) we will not restrict the ground field $\F$ except in our study of Lefschetz properties (Section~\ref{Lefsec}). We say that $A$ is \emph{standard graded} if it is generated as an algebra by its linear forms, i.e. $\F[A_1]=A$.  Although many of our examples are standard graded, we will not assume this, except in our discussion of compressed algebras (Section \ref{compsec}).  An \emph{orientation} on the AG algebra $A$ is a choice of linear isomorphism $\int_A\colon A_d\overset{\cong}{\rightarrow} \F$ or, equivalently, a choice of socle generator, which we shall denote by $0\neq a_{soc}\in A_d$.  An \emph{oriented AG algebra} is a pair $\left(A,\int_A\right)$ consisting of a graded Artinian Gorenstein algebra $A$ and an orientation $\int_A$; when the orientation is understood we shall drop the $\int_A$ and just speak of an oriented AG algebra $A$, and its distinguished socle generator $a_{soc}$.  

Suppose that $A$ and $T$ are two oriented AG algebras of socle degrees $d$ and $k$, respectively, with $d>k$, and suppose that $\pi\colon A\rightarrow T$ is a degree-preserving algebra map between them; we may occasionally drop the adjectives and simply refer to the map $\pi$.  Since $A$ is Gorenstein, multiplication defines a perfect pairing for all $0\leq i\leq d$:
$$\xymatrixrowsep{.1pc}\xymatrix{A_i\times A_{d-i}\ar[r] & \F\\
(a,a')\ar@{=}[r] & \int_Aaa'}.$$  
In particular, for any $\phi\in \operatorname{Hom}_\F(A_i,\F)$ there exists a unique $a\in A_{d-i}$ such that $\phi(b)=\int_Aa\cdot b$ for every $b\in A_i$.  Pulling back the orientation on $T$ by $\pi$ defines a homomorphism ${\int_T}\circ {\pi}\colon A_k\rightarrow\F$ and hence, as shown in \cite[Lemma 2.1]{IMS}, there exists a unique element $\tau\in A_{d-k}$ for which 
\begin{equation}
\label{eq:Thom}
\int_T\pi(a)=\int_A\tau\cdot a, \ \ \ \forall \ a\in A.
\end{equation} 

\begin{definition}[Thom class, Euler class]
	\label{def:Thom}
	The unique element $\tau\in A_{d-k}$ defined by Equation \eqref{eq:Thom} is called the \emph{Thom class of $\pi$}.
	The image of the Thom class $\pi(\tau)\in T_{d-k}$ is called the \emph{Euler class of $\pi$}.\footnote{Topologically, if $A=H^{2\bullet}(X)$ and $T=H^{2\bullet}(Y)$ are the cohomology rings of a manifold $X$ and a submanifold $\pi\colon Y\hookrightarrow X$ and $\pi^*\colon H^{2\bullet}(X)\rightarrow H^{2\bullet}(Y)$ is the induced restriction map on cohomology algebras, then the Thom class of $\pi^*$ is exactly the Thom class of the normal bundle of $Y$ in $X$, also known as the \emph{Poincar\'e dual class of $Y\subset X$}, and its image under the restriction map is the Euler class of the normal bundle, e.g. \cite{MS}.} 
\end{definition}

Note that the Thom class is zero if and only if the socle of $T$ does not belong to the image of $\pi$.  We call $\pi\colon A\rightarrow T$ a \emph{restriction map} if its Thom class is non-zero.  In particular, if $\pi$ is surjective, then it is a restriction map, but not every restriction map is surjective.  In complex geometry, if $X$ is a K\"ahler manifold and $\pi\colon Y\subset X$ is a codimension $n$ K\"ahler submanifold, then the induced map on cohomology $\pi^*\colon H^{2\bullet}(X)\rightarrow H^{2\bullet}(Y)$ is a restriction map \cite[Exercise 3.3.9]{Huybrechts}, though it need not be surjective, e.g. Example \ref{ex:Segre1}.   Here is a  useful characterization of the Thom class of a restriction map.  

\begin{lemma}
	\label{lem:Thomdual}
	Let $\pi\colon A\rightarrow T$ be a restriction map between two oriented AG algebras of socle degrees $d>k$, respectively, and let $K=\ker(\pi)\subset A$ be its kernel.  Let $a\in A$ be any homogeneous element of degree $n=d-k$.  Then $a\cdot K=0$ if and only if $a=\lambda\cdot \tau$ is a multiple of the Thom class of $\pi$. 
\end{lemma} 
\begin{proof}
	Let $\int_A\colon A_d\rightarrow\F$ and $\int_T\colon T_k\rightarrow \F$ be orientations on $A$ and $T$, respectively.  Assume first that $a\in A_n$ and $a\cdot K=0$.  Consider the short exact sequence of vector spaces $0\rightarrow K_k\rightarrow A_k\rightarrow T_k\rightarrow 0$.  Since $T$ is Gorenstein of socle degree $k$, $T_k$ is one-dimensional, and hence $K_k\subset A_k$ is a codimension-one subspace.  Therefore the set of homomorphisms $\phi\in \operatorname{Hom}_{\F}(A_k,\F)$ which vanish on $K_k$ is one-dimensional.  Since $\phi_1(u)=\int_A\tau\cdot u=\int_T\pi(u)=0$ for every $u\in K$, and also $\phi_2(u)=\int_Aa\cdot u=0$ for every $u\in K$, we must have $\phi_1=\lambda\cdot\phi_2$, which implies that $\tau=\lambda\cdot a$.  The converse is clear. 
	\end{proof}
If $\pi\colon A\rightarrow T$ is surjective, which will be a standing assumption throughout this paper, then the Thom class has some nice alternative descriptions. First, let $R=\F[x_1,\ldots,x_r]$ be a graded polynomial ring with homogeneous maximal ideal $\mathfrak{m}=(x_1,\ldots,x_r)\subset R$, and let $I\subset R$ be a homogeneous $\mathfrak{m}$-primary ideal.  Then $R/I$ is graded Artinian and it is Gorenstein if and only if $I$ is irreducible, meaning that it cannot be written as an intersection of two strictly larger ideals \cite[Lemma I.1.3]{MeyerSmith}.  Suppose that $J\subset R$ is another homogeneous $\mathfrak{m}$-primary irreducible ideal with $I\subset J$, let $T=R/J$ and let $\pi\colon A\rightarrow T$ be the natural projection map.  The following fact is well known; see \cite[Lemma 4]{Watanabe} and \cite[Theorem I.2.1]{MeyerSmith}. 
\begin{lemma}
	\label{lem:transition}
	With the above assumptions,  let $\bar{\tau}\in A=R/I$ be the Thom class of $\pi$ and let $\tau\in R$ any homogeneous lift.  Then we have 
	$$J=(I\colon\tau) \ \ \text{and} \ \ (I:J)=(\tau)+I.$$
\end{lemma}    
In \cite{MeyerSmith}, the authors refer to the the homogeneous lift of the Thom class $\tau\in R$ in Lemma \ref{lem:transition} as a transition element for $I\subseteq J$.

Next let $Q=\F[X_1,\ldots,X_r]$ be a divided power algebra on which $R$ acts by contraction, e.g. \cite[Appendix A]{IK} or \cite[Appendix A.2.4]{Eisenbud}: 
$$x_i\circ X_1^{a_1}\cdots X_i^{a_i}\cdots X_r^{a_r}=\begin{cases} X_1^{a_1}\cdots X_i^{a_i-1}\cdots X_r^{a_r} & \text{if} \ a_i>0\\ 0 & \text{if} \ a_i=0\\ \end{cases}.$$
Then it is well known that a homogeneous ideal $I\subset R$ is $\mathfrak{m}$-primary irreducible of socle degree $d$ (meaning the socle degree of $A=R/I$) if and only if there exists a homogeneous form $F\in Q_d$ for which $I=\Ann(F)=\{r\in R \ | \ r\circ F=0\}$.  In this case, $F$ is called a Macaulay dual generator of $A$.  Suppose that $G\in Q_k$ ($k<d$) is another homogeneous form for which $I=\Ann(F)\subset J=\Ann(G)$, let $T=R/J$ and let $\pi\colon A=R/\Ann(F)\rightarrow R/\Ann(G)$ be the natural projection map.  Note that $A$ and $T$ have natural orientations coming from $F$ and $G$:  $\int_Aa = (a\circ F)(0)$ and $\int_Tt=(t\circ G)(0)$.  
\begin{lemma}
\label{lem:MDThom}
Under the assumptions of Lemma \ref{lem:transition}, if $\tau\in R$ is any homogeneous lift of the Thom class of $\pi$, then   
$$\tau\circ F=G.$$
\end{lemma}

\begin{proof}
Indeed the Thom class condition \eqref{eq:Thom} translates to the condition  
\begin{align}
\notag (a\circ G) (0)  &= (\tau\cdot a\circ F) (0)=(a\circ (\tau\circ F))(0), \ \ \forall \ \ a\in R.\\
\intertext{or, equivalently}
\label{eq:MD}
0 & = a\circ\left(G-\tau\circ F\right)(0), \ \ \forall \ \ a\in R.
\end{align}
but because the pairing $R_i\times Q_i\rightarrow \F$, $(a,H)=(a\circ H)(0)$ is non-degenerate, we see that Condition~\eqref{eq:MD} is equivalent to the claimed condition that $G=\tau\circ F$.
\end{proof}

\begin{remark}
	\label{rem:orientations}
	If $\tau$ is the Thom class of a map of AG algebras $\pi\colon A\rightarrow T$, and if $\lambda\in\F^\times$ is a non-zero constant, we can get a new Thom class  $\tau'=\lambda\cdot\tau$ by either scaling the distinguished socle generator of $A$ by $\lambda$, i.e. $a_{soc}\mapsto \lambda\cdot a_{soc}=a'_{soc}$, or by scaling the distinguished socle generator of $T$ by $\lambda^{-1}$, i.e. $t_{soc}\mapsto \lambda^{-1}\cdot t_{soc}=t'_{soc}$.
\end{remark}

In our proofs we make frequent use of the following result: it is valid for an arbitrary Artinian Gorenstein algebra $A$ (not necessarily graded) and an arbitrary Artinian algebra $B$.

\begin{lemma}
\label{lem:soclefits}
Let $\varphi:A\to B$ be a homomorphism of Artinian algebras so that $A$ is Gorenstein and $\varphi$ restricted to the socle of $A$ is injective. Then $\varphi$ is injective. 

In particular, if $\varphi:A\to B$ is a surjective homomorphism of graded AG rings of the same socle degree, then $\varphi$ is an isomorphism.
\end{lemma}
\begin{proof}
Let $\varphi:A\to B$ be a ring homomorphism with $A$ AG. Let $a$ be an element of the kernel of $\varphi$. If $a\neq 0$ then there exists $0\neq a'\in A$ such that  $0\neq aa'\in\soc(A)$. However $\varphi(aa')=\varphi(a)\varphi(a')=0$ since $\varphi(a)=0$ by assumption. This contradicts the restriction of $\varphi$ to $\soc(A)$ being injective. Thus it must be the case that $a=0$ and consequently $\varphi$ is injective. 

Now assume $\varphi:A\to B$ is a surjective homomorphism of graded AG rings of the same socle degree. Then $\varphi:\soc(A)\to \soc(B)$ is a vector space isomorphism. Applying the first assertion, $\varphi$ is injective, hence bijective.
\end{proof}

\section{Cohomological Blow Ups}
\label{sec:CBU}
Let $A$ and $T$ be oriented AG algebras of socle degrees $d$ and $k$ respectively with $d>k$ and $\pi\colon A\rightarrow T$  a surjective degree preserving algebra map with Thom class  $\tau\in A_{d-k}$.
\begin{construction}
	\label{con:hat}
	Set $n=d-k$ and let $\xi$ be an indeterminate of degree one.  Choose any homogeneous elements $a_i\in A$ for $1\leq i\leq n$, and define the monic homogeneous polynomial of degree $n$, $f_A(\xi)\in A[\xi]$, by
	\begin{equation}
	\label{eq:fA}
	f_A(\xi)=\xi^n+a_{1}\xi^{n-1}+\cdots+a_{n-1}\xi+a_n.
	\end{equation}
	Then let $K\subset A$ be the kernel of $\pi$, and construct the algebra $\hat{A}$ by 
	\begin{equation}
	\label{eq:hatA}
	\hat{A}=\frac{A[\xi]}{(\xi\cdot K,f_A(\xi))}.
	\end{equation}
	We further set $t_i=\pi(a_i)$ for $1\leq i\leq n$, define the monic homogeneous polynomial $f_T(\xi)\in T[\xi]$ 
	\begin{equation}
	\label{eq:fT}
	f_T(\xi)=\xi^n+t_1\xi^{n-1}+\cdots+t_n,
	\end{equation} 
	and construct the algebra $\tilde{T}$ by 
	\begin{equation}
	\label{eq:hatT}
	\tilde{T}=\frac{T[\xi]}{(f_T(\xi))}.
	\end{equation}
	\end{construction}
Before giving the formal definition of the cohomological blow up (Definition \ref{def:blowup}), we shall discuss some of the fundamental properties of the algebras $\hat{A}$ and $\tilde{T}$ from Construction \ref{con:hat}.  First, note that since $\xi\cdot K=0$ in $\hat{A}$, the algebra structure of $\hat{A}$ depends only on the images $t_1,\ldots,t_{n-1}$, as well as $a_n$. There is a natural degree preserving algebra map $\beta\colon A\rightarrow \hat{A}$ induced from the natural inclusions $A\hookrightarrow A[\xi]$ which makes $\hat{A}$ into and $A$-algebra.  Also, the surjective tensor product map $\pi\otimes 1\colon A[\xi]\rightarrow T[\xi]$ passes to a surjective degree preserving map on quotients $\hat{\pi}\colon \hat{A}\rightarrow \tilde{T}$.  Moreover these maps $\pi:A\to T$, $\beta\colon A\rightarrow \hat{A}$ and $\hat{\pi}:\hat{A}\to \tilde{T}$ together with the natural inclusion $\beta_0\colon T\rightarrow \tilde{T}$ fit together in a commutative diagram
\begin{equation}
\label{eq:cd}
\xymatrix{A\ar[r]^-{\beta}\ar[d]_-{\pi} & \hat{A}\ar[d]^-{\hat{\pi}}\\ T\ar[r]_-{\beta_0} & \tilde{T}.}
\end{equation}    

The following elementary observation will be useful and we record it as a lemma.  We offer our own proof here for completeness, but it can also be deduced from \cite[Lemma 1.8]{IMM} and \cite[Lemma 2.1]{SmithStong}.  Recall that for graded Artinian algebras $A$, $B$, and $C$, we say that $C$ is a \emph{free extension of $A$ with fiber $B$} if there are maps $\iota\colon A\rightarrow C$ making $C$ into a free $A$-module and $\pi\colon C\rightarrow B$ with kernel $\ker(\pi)=\mathfrak{m}_A\cdot C$ where $\mathfrak{m}_A$ is the maximal ideal of $A$. 
\begin{lemma}
	\label{lem:freeExt}
	The algebra $\tilde{T}$ from Construction \ref{con:hat} in Equation \eqref{eq:hatT} is a free extension of $T$ with fiber $F\cong\F[\xi]/(\xi^n)$.  In particular, since $T$ and $F$ are Gorenstein, $\tilde{T}$ is also Gorenstein\footnote{A more general result established in a 2008 unpublished note ``Coexact Sequences of Poincar\'e Duality Algebras'', by L.~Smith and R.E.~Stong states that if $\tilde{T}$, $T$ and $F$ are graded Artinian algebras in which $\tilde{T}$ is a free extension of $T$ with fiber $F$, and any two are Gorenstein, then so is the third.}.
\end{lemma}
\begin{proof}
	We first claim that the map $\beta_0\colon T\rightarrow \tilde{T}$ makes $\tilde{T}$ into a free $T$-module with basis $\{1,\xi,\ldots,\xi^{n-1}\}$.  Clearly this set generates $\tilde{T}$ as a $T$-module, since powers of $\xi$ generate $T[\xi]$ as a $T$-module, and all powers greater than $n-1$ can be eliminated with the relation $\xi^n\equiv -(t_1\xi^{n-1}+\cdots+t_n)$.  A $T$-dependence relation in $\tilde{T}$ lifts to a $T$-relation in $T[\xi]$ of the form 
	$c_0\cdot 1+c_1\xi+\cdots+c_{n-1}\xi^{n-1}=t\cdot \left(\xi^n+t_1\xi^{n-1}+\cdots+t_n\right)$ for some $c_i,r\in T$.
	Comparing $\xi$-coefficients, we conclude that $t=0$, and hence $c_i=0$ for all $i$ since $\{1,\xi,\ldots,\xi^{n-1}\}$ are $T$-linearly independent in $T[\xi]$.  Finally note that the natural projection map $\phi\colon T\rightarrow \F$ extends to a projection map $\tilde{\phi}\colon \tilde{T}\rightarrow \F[\xi]/(\xi^n)$ with kernel $\ker(\tilde{\phi})=\mathfrak{m}_T\cdot \tilde{T}$.  Therefore $\tilde{T}$ is a free extension of $T$ with fiber $F=\F[\xi]/(\xi^n)$ as claimed.
	
	For the last statement, it suffices to show that the socle of $\tilde{T}$ satisfies $\operatorname{soc}(\tilde{T})=\operatorname{soc}(T)\cdot\xi^{n-1}$.  One containment is obvious, and for the other, assume that $\hat{t}\in\operatorname{soc}(\tilde{T})$ is a homogeneous socle element.  From our arguments above, we may decompose it as $\hat{t}=s_0+s_1\xi+\cdots+s_{n-1}\xi^{n-1}$ for some unique $s_i\in T$.  Since $\hat{t}$ is in the socle, we must have $t\cdot \hat{t}=0$ for any positive degree element $t\in T_+$, which implies that $s_i\in\operatorname{soc}(T)$ for all $i$, by linear independence of $\{1,\ldots,\xi^{n-1}\}$.  On the other hand, since $\hat{t}$ is homogeneous, and $\deg(s_i)=\deg(\hat{t})-i$ it follows that there is only one non-zero $s_i$, and since $\xi\cdot \hat{t}=0$ it follows that $i=n-1$.  Therefore we have shown that $\hat{t}=s_{n-1}\cdot\xi^{n-1}$ where $s_{n-1}\in\operatorname{soc}(T)$, which implies that $\operatorname{soc}(\tilde{T})=\operatorname{soc}(T)\cdot\xi^{n-1}$, as desired.	
\end{proof}

Lemma \ref{lem:FreeExt} implies that $\tilde{T}$ is always Gorenstein.  On the other hand, the algebra $\hat{A}$ may not always be Gorenstein, as the following example shows.

\begin{example}
\label{ex:notGor}
Let 
$$A=\frac{\F[x,y]}{(x^3,y^3)} \ \overset{\pi}{\rightarrow} \ T=\frac{\F[x,y]}{(x^2,y)}$$
where $\pi(x)=x$ and $\pi(y)=0$ (here $d=4$, $k=1$, and $n=d-k=3$).  Note $K=\ker(\pi)=(x^2,y)$.  Orient $A$ and $T$ with socle generators $a_{soc}=x^2y^2$ and $t_{soc}=x$; then the Thom class of $\pi$ is $\tau=xy^2\in A_3$.  Set $f_T(\xi)=\xi^3+x\xi^2\in T[\xi]$, so that $t_{1}=x$ and $t_{2}=t_{3}=0$, and let $\tilde{T}$ be the associated free extension:
$$\tilde{T}=\frac{T[\xi]}{(f_T(\xi))}=\frac{\F[x,y,\xi]}{(x^2,y,\xi^3+x\xi^2)}.$$
Below are different algebras $\hat{A}$ for different choices of $\pi$-lifts $f_A(\xi)=\xi^3+a_1\xi^2+a_2\xi+a_3$ of $f_T(\xi)$, according to Construction \ref{con:hat}:
\begin{enumerate}
\item $a_1=x$, $a_2=0$, and $a_3=(x^2y+xy^2)$.  Then $f_A(\xi)\in A[\xi]$ is $f_A(\xi)=\xi^3+x\xi^2+(x^2y+xy^2)$, and  
\[
\hat{A}=\frac{A[\xi]}{(\xi \cdot K,f_A(\xi))}=\frac{\F[x,y,\xi]}{(x^3,y^3,x^2\xi,y\xi,\xi^3+x\xi^2+x^2y+xy^2)}
\]
Then an $\F$-basis for $\hat{A}$ is 
$$\left\{1,x,y,\xi,x^2,xy,y^2, x\xi, \xi^2, x^2y, xy^2,x\xi^2\right\}$$
from which it follows that the Hilbert function for $\hat{A}$ is 
$$H(\hat{A})=(1,3,5,3)$$
and hence $\hat{A}$ is not Gorenstein.  Note that in this case, the socle generator $a_{soc}=x^2y^2$ is actually in the ideal $(\xi \cdot K,\hat{f}(\xi))$, hence $\beta(a_{soc})=0$ and thus $\beta$ is not injective.

\item $a_1=x$, $a_2=0$, and $a_3=xy^2=\tau$. Then $f_A(\xi)=\xi^3+x\xi^2+xy^2\in A[\xi]$ and
$$\hat{A}=\frac{A[\xi]}{(\xi \cdot K,f_A(\xi))}=\frac{\F[x,y,\xi]}{(x^3,y^3,x^2\xi,y\xi,\xi^3+x\xi^2+xy^2)}$$ 
has basis 
$$\left\{1,x,y,\xi,x^2,xy,y^2,x\xi,\xi^2,x^2y,xy^2,x\xi^2,x^2y^2\right\}$$
and Hilbert function 
$$H(\hat{A})=(1,3,5,3,1).$$
Here the socle of $\hat{A}$ is generated by $\hat{a}_{soc}=a_{soc}=x^2y^2$, hence $\hat{A}$ is Gorenstein.  

\item $a_1=x+y$ and $a_2=a_3=0$. Here we chose a different $\pi$-lift of $t_1=x$, but the reader will see it does not affect $\hat{A}$; the important choice is  $a_3=0$.
Then $f_A(\xi)=\xi^3+x\xi^2\in A[\xi]$ and 
$$\hat{A}=\frac{\F[x,y,\xi]}{(x^3,y^3,x^2\xi,y\xi,\xi^3+(x+y)\xi^2)}=\frac{\F[x,y,\xi]}{(x^3,y^3,x^2\xi,y\xi,\xi^3+x\xi^2)}$$
with basis 
$$\left\{1,x,y,\xi,x^2,xy,y^2,x\xi,\xi^2,x^2y,xy^2,x\xi^2,x^2y^2\right\}$$
and Hilbert function 
$$H(\hat{A})=(1,3,5,3,1)$$
However, note that in this case the socle of $\hat{A}$ is the two-dimensional $\F$-vector space generated by $x\xi^2$ and $x^2y^2$, therefore $\hat{A}$ is not Gorenstein.  In contrast to case (1), in this case the image of the socle of $A$ in $\hat{A}$ via $\beta$ is nonzero and in fact the map $\beta\colon A\rightarrow\hat{A}$ is injective.
\end{enumerate} 
\end{example}

The example suggests a strong dependence of the Gorenstein property on the choice of $a_n$ from Equation \eqref{eq:fA}.  Before we give our main result in this direction, we need some lemmas.  The following lemma computes the Hilbert function of $\hat{A}$ in terms of those of $A$ and $T$.
\begin{lemma}
	\label{lem:decomp}
	For each $\hat{a}\in \hat{A}$, there exist unique elements $b_0\in A$ and elements $s_1,\ldots,s_{n-1}\in T$ for which 
	\[\hat{a}=\beta(b_0)+\beta(b_1)\xi^1+\cdots+\beta(b_{n-1})\xi^{n-1}, \ \ \text{for some} \ \ b_i\in A, \ \ \text{where} \ \ \pi(b_i)=s_i.\]
	In particular the Hilbert function of $\hat{A}$ satisfies 
	\begin{equation}
	\label{eq:HhatA}
	H(\hat{A})=H(\beta(A))+H(T)[1]+\cdots+H(T)[n-1].
	\end{equation}
\end{lemma}
\begin{proof}
	Existence of such a decomposition is easy:  every element $\hat{a}\in\hat{A}$ has a representative in the polynomial ring $A[\xi]$, and high powers of $\xi$ can be reduced via the relation
	\[
	f_A(\xi)=\xi^n+a_1\xi^{n-1}+\cdots+a_n\equiv 0.
	\]
	As for uniqueness, suppose that there are some other elements $c_0\in A$ and other $s_1',\ldots,s_{n-1}'\in T$ with $\pi$-lifts $c_1,\ldots,c_{n-1}\in A$ such that 
	\[
	\hat{a}=\beta(c_0)+\beta(c_1)\xi+\cdots + \beta(c_{n-1})\xi^{n-1}.
	\]
	Comparing their decompositions in $\tilde{T}$ via the projection $\hat{\pi}$ we find that 
	\[
	0=\beta_0(\pi(b_0-c_0))+\beta_0\left(\pi(b_1-c_1)\right)\xi+\cdots+\beta_0\left(\pi(b_{n-1}-c_{n-1})\right)\xi^{n-1}.
	\]
	It follows from Lemma \ref{lem:FreeExt} that $\tilde{T}$ is a free $T$-module with basis $\{1,\xi,\ldots,\xi^{n-1}\}$ via the injective map $\beta_0\colon T\rightarrow \tilde{T}$, and hence we conclude that $s_i=\pi(b_i)=\pi(c_i)=s_i'$ for all $i=1,\ldots,n-1$, and that $\pi(a_0)=\pi(c_0)$, hence $b_0-c_0\in K$.  But we also have the relation in $\hat{A}$ 
	\[
	0=\beta\left(b_0-c_0\right)+\beta\left(b_1-c_1\right)\xi+\cdots+\beta\left(b_{n-1}-c_{n-1}\right)\xi^{n-1}.
	\]
	Since $\beta\left(b_i-c_i\right)\xi^i\in \left(\xi \cdot K\right) $, the identity displayed above simplifies to  
	\[
	0=\beta(b_0-c_0)
	\]
	and hence $\beta(b_0)=\beta(c_0)$ in $\hat{A}$. Thus the decomposition is unique in the desired sense. Equation \eqref{eq:HhatA} follows immediately from this decomposition. 
\end{proof}

\begin{lemma}
	\label{lem:Kcap}
	With $A$, $T$, $\pi$, $K$, and $f_A(\xi)=\xi^n+a_1\xi^{n-1}+\cdots+a_n\in A[\xi]$ as in Construction \ref{con:hat}, we have in $A[\xi]$
	$$\left(\xi\cdot K,f_A(\xi)\right)\cap \beta(A)=a_n\cdot K.$$
	
\end{lemma}
\begin{proof}
	Let $b\in \left(\xi\cdot K,f_A(\xi)\right)\cap \beta(A)$.  Then there is a polynomial $g(\xi)\in A[\xi]$ for which $b-g(\xi)f_A(\xi)\in \left(\xi\cdot K\right)$ in $A[\xi]$.  Writing $g(\xi)=g_m\xi^m+\cdots+g_1\xi+g_0$ for $g_i\in A$, since $f_A(\xi)$ is monic we must therefore have in $A[\xi]$
	\begin{equation}
	\label{eq:Kcap}
	b-g_m\xi^{m+n}+\left(\text{lower order terms}\right)\in\left(\xi \cdot K\right).
	\end{equation}
	Since $m+n\geq n\geq 1$ we can compare coeffients on the left hand side and right hand side of \eqref{eq:Kcap} to deduce that $g_m\in K$.  If $m\geq 1$, one we can combine the term $g_m\xi^m \cdot \hat{f}(\xi)$ with the other $\left(\xi \cdot K\right)$ terms, and lower the $\xi$-degree of $g(\xi)$.  Repeating this procedure, we may assume that the $\xi$-degree is $m=0$, that $g_m=g_0\in K$, and hence that $b-g_0 \cdot f_A(\xi)\in \left(\xi \cdot K\right)$.  Since $g_0\in K$, it follows that $g_0\left(f_A(\xi)-a_n\right)\in \left(\xi\cdot K\right)$ as well, and therefore that $b-g_0\cdot a_n\in\left(\xi\cdot K\right)$ which implies that $b-g_0\cdot a_n=0$, and hence that $b\in a_n
	\cdot K$, as desired.  The reverse containment is obvious.
\end{proof}

\begin{theorem}
	\label{thm:hatAGor}
	Let $A$, $T$, $\pi$, $\tau$, $f_A(\xi)$, $\hat{A}$, and $\beta$
	be as above.  Then
	\begin{enumerate}
		\item the algebra map $\beta\colon A\rightarrow \hat{A}$ is injective if and only if the constant coefficient of $f_A(\xi)$ satisfies $a_n=\lambda\cdot \tau$ for some $\lambda\in\F$ (possibly $\lambda=0$), and 
	
		\item the algebra $\hat{A}$ is Gorenstein if and only if $a_n=\lambda\cdot\tau$ for some non-zero $\lambda\in\F^\times$.
	\end{enumerate}
\end{theorem}

\begin{proof}
	For (1.), observe that $\beta\colon A\rightarrow \hat{A}$ is injective if and only if $\left(\xi\cdot K,f_A(\xi)\right)\cap \beta(A)=0$, which by Lemma \ref{lem:Kcap} is equivalent to $a_n\cdot K=0$, which is in turn, equivalent to $a_n=\lambda\cdot\tau$ for some $\lambda\in\F$ by Lemma \ref{lem:Thomdual}.
	
	For (2.), assume first that $\hat{A}$ is Gorenstein.  Then $\operatorname{soc}(\hat{A})=\hat{A}_e$ where $e$ is the largest integer for which $\hat{A}_e\neq 0$.  It follows from Lemma \ref{lem:decomp} that $e\leq d$ and that $e<d$ if and only if $\beta(A_d)=0$.  But if $e<d$, then by the surjectivity of the algebra map $\hat{\pi}\colon \hat{A}\rightarrow\tilde{T}$ we must have $\hat{\pi}(\hat{A}_{d-1})=\tilde{T}_{d-1}\neq 0$, and hence $0\neq \hat{A}_{d-1}=\operatorname{soc}(\hat{A})$ and $e=d-1$. By Lemma \ref{lem:soclefits},  $\hat{\pi}$ must be an isomorphism, i.e. 
	$$\hat{\pi}\colon\hat{A}=\frac{A[\xi]}{(\xi\cdot K,f_A(\xi))}\overset{\cong}{\rightarrow}\frac{A[\xi]}{(K,f_A(\xi))}=\tilde{T}.$$
	In particular, we see that $K\subseteq (\xi\cdot K,f_A(\xi))$ in $A[\xi]$, which by Lemma \ref{lem:Kcap} implies that $K\subseteq a_n.K$, which is impossible for degree reasons (since we are assuming $d>k$).  Therefore we must have $\operatorname{soc}(\hat{A})=\hat{A}_d\neq 0$, which by Lemma \ref{lem:decomp} must be the image of $A_d$, and hence $\beta(a_{soc})\neq 0$.  By (1) this implies that $a_n=\lambda\cdot \tau$ for some $\lambda\in\F$.  Next we claim that $\lambda\neq 0$.  Indeed assume that $\lambda=0$ so that $a_n=0$.  Let $b_0\in A$ be any $\pi$-lift of $t_{soc}\in T_k$.  Then note that $\xi^{n-1}\beta(b_0)\in \operatorname{soc}(\hat{A})$.  Indeed in $A[\xi]$ we have  $\xi\cdot \left(\xi^{n-1}b_0\right)-b_0 f_A(\xi)\in\left(\xi\cdot K\right)$, hence $\xi\cdot \left(\xi^{n-1}b_0\right)\in (\xi\cdot K,f_A(\xi))$, but also for any $b\in A$ of positive degree $b_0\cdot b\in K$ and hence $b\cdot \left(\xi^{n-1}\cdot b_0\right)\in\left(\xi\cdot K\right)$ as well.  We further claim that $\left\{\xi^{n-1}\cdot \beta(b_0), \beta(a_{soc})\right\}$ are linearly independent in $\hat{A}$.  Indeed, $\hat{\pi}(\xi^{n-1}\cdot \beta(b_0))=\xi^{n-1}t_{soc}$ is a socle generator of  $\tilde{T}$ (hence is non-zero), whereas $\hat{\pi}(\beta(a_{soc}))=0$.  Since $\beta(a_{soc})\neq 0$, this shows that the socle of $\hat{A}$ has dimension at least two, contradicting our assumption that $\hat{A}$ is Gorenstein.
	
	Conversely assume that $a_n=\lambda\cdot \tau$ for some non-zero $\lambda\in\F^\times$.  Then by (1.)  $\beta\colon A\rightarrow \hat{A}$ is injective, and hence $\beta(a_{soc})\in\operatorname{soc}(\hat{A})$.  We want to show that $\beta(a_{soc})$ generates the socle.  To that end, suppose that $\hat{a}\in \hat{A}$ is any other socle element, and as in Lemma \ref{lem:decomp} write
	\begin{equation}
	\label{eq:hata}
	\hat{a}=\beta(b_0)+\beta(b_1)\xi+\cdots+\beta(b_{n-1})\xi^{n-1}.
	\end{equation}
	Note that since $\hat{\pi}$ is surjective, it must map socle elements of $\hat{A}$ to socle elements of $\tilde{T}$.  Thus applying $\hat{\pi}$ to \eqref{eq:hata} we find that there is a constant $c$ (possibly $c=0$) such that 
	$$\hat{\pi}(\hat{a})=\pi(b_0)+\pi(b_1)\xi+\cdots+\pi(b_{n-1})\xi^{n-1}=c\cdot t_{soc}\cdot \xi^{n-1}.$$
	Since $\tilde{T}$ is a free $T$-module with basis $\{1,\xi,\ldots,\xi^{n-1}\}$, it follows that $\pi(b_i)=0$ in $T$ for $0\leq i\leq n-2$ and $\pi(b_{n-1})=c\cdot t_{soc}$.  This implies that $b_0,\ldots,b_{n-2}\in K$, and therefore that $\beta(b_i)\xi^i=0$ in $\hat{A}$ for $i=1,\ldots,n-2$.  Thus $\hat{a}$ reduces to  
	$$\hat{a}=\beta(b_0)+\beta(b_{n-1})\xi^{n-1}.$$
	Since $\hat{a}\in\operatorname{soc}(\hat{A})$ we have $\xi\cdot\hat{a}=0$ in $\hat{A}$ and therefore
	\begin{align*}
	0= \xi\cdot \hat{a}= & \xi\cdot\beta(b_0)+\beta(b_{n-1})\xi^n\\
	& (\text{since} \ b_0\in K, \ \text{and} \ \xi^n\equiv -\left(a_1\xi^{n-1}+\cdots+a_n\right))\\
	\equiv  & -\beta(b_{n-1})\cdot \left(a_1\xi^{n-1}+\cdots+a_{n-1}\xi+a_n\right)\\
	& (\text{since} \ \pi(b_{n-1})=c\cdot t_{soc}, \ \text{and} \ a_i\in A_i \ \text{for} \ 1\leq i\leq n-1)\\
	\equiv & \beta(b_{n-1}) \cdot a_n\\
	& (\text{since} \ b_{n-1}=c\cdot t_{soc} \ \text{and} \ a_n=\lambda\cdot \tau)\\
	\equiv  & \lambda\cdot c\cdot \beta(a_{soc}).  
	\end{align*}
	Since $\lambda\neq 0$ and $\beta(a_{soc})\neq 0$ we must conclude that $c=0$ and hence $b_{n-1}\in K$ as well, and therefore that $\hat{a}=\beta(b_0)$.  But since $\beta(b_0)=\hat{a}\in\operatorname{soc}(\hat{A})$ and $\beta$ is injective, it  follows that $b_0\in\operatorname{soc}(A)$, as desired.  Therefore $\operatorname{soc}(\hat{A})\subseteq \beta(\operatorname{soc}(A))$, from which it follows that $\hat{A}$ must be Gorenstein.	
\end{proof}

\begin{corollary}
	\label{cor:splitseq}
	If $a_n=\lambda\cdot \tau$ with $\lambda\in \F$ then 
	\[
	H(\hat{A})=H(A)+H(T)[1]+\cdots+H(T)[n-1]
	\]
	 and there is a split exact sequence of $A$-modules 
	\[
	\xymatrix{0\ar[r] & A \ar[r]^-{\beta} &  \hat{A} \ar[r]^-{\hat{\pi}} &  \tilde{T}/\beta_0(T)\ar[r] &  0.}	\]
\end{corollary}
\begin{proof}
	Injectivity for $\beta$ follows from Theorem \ref{thm:hatAGor} part (1) and surjectivity of $\hat{\pi}$ is by definition. That the displayed sequence is a complex follows from the commutative diagram \eqref{eq:cd}. Exactness of the sequence above viewed as a sequence of vector spaces follows from Lemma \ref{lem:decomp} and the identity
	\[
	H(\hat{A})=H(\beta(A))+H(T)[1]+\cdots+H(T)[n-1]=H(A)+H(\tilde{T}/\beta_0(T)).
	\]
	Finally, since $A$ is a Gorenstein ring, $A$ is injective as an $A$-module, thus the above sequence splits.
\end{proof}

\begin{remark}
\label{rem:fakeG}
It follows from Theorem \ref{thm:hatAGor} and Corollary \ref{cor:splitseq}, that if $a_n=0$ then $\beta\colon A\rightarrow \hat{A}$ is injective and 
		$$H(\hat{A})=H(A)+H(T)[1]+\cdots+H(T)[n-1],$$
but $\hat{A}$ is not Gorenstein.  We call such algebras \emph{boundary Gorenstein algebras}, indicating that they are in the closure of the Gorenstein locus of the Hilbert scheme; the algebra $\hat{A}$ in Example \ref{ex:notGor}(3.) is a boundary Gorenstein algebra of this type.  More precisely, one can show that if $\F$ is algebraically closed then the algebra 
$$\tilde{A}[\lambda]=\frac{A[\xi,\lambda]}{(\xi\cdot K, \xi^n+a_1\xi^{n-1}+\cdots+\lambda\cdot\tau)}$$
is flat as a module over $\F[\lambda]$, where the fibers $\tilde{A}[c]/(\lambda-c)\cdot \tilde{A}[\lambda]$ are Gorenstein if $c\neq 0$ and not Gorenstein, but boundary Gorenstein, if $c=0$.  We shall give another flat family in which the Gorenstein algebra $\hat{A}$ is a general fiber in Section \ref{Lefsec}. 
\end{remark}

\begin{definition}[Preferred Orientations]
\label{def:orientations}
By Lemma \ref{lem:FreeExt}, $\tilde{T}$ from Construction~\ref{con:hat} is always Gorenstein, and we define its preferred orientation as the one corresponding to the socle generator, $\hat{t}_{soc}=\xi^{n-1}\cdot t_{soc}$.
If $\hat{A}$ from Construction \ref{con:hat} is Gorenstein, then we define its preferred orientation as the one corresponding to the socle generator $\hat{a}_{soc}=\beta(a_{soc})$; hence the preferred orientation on $\hat{A}$ is the one inherited from $A$ via $\beta$.
\end{definition}

\begin{definition}[Cohomological Blow-Up, Exceptional Divisor, Cohomological Blow Down]
	\label{def:blowup}
	 Given oriented AG algebras $A$ and $T$ of socle degrees $d>k$, respectively, and surjective degree-preserving algebra map $\pi\colon A\rightarrow T$ with Thom class $\tau\in A_n$ where $n=d-k$, and given a homogeneous monic polynomial $f_A(\xi)=\xi^n+a_1\xi^{n-1}+\cdots+a_n\in A[\xi]$ of degree $n$ with homogeneous elements $a_i\in A_i$ for $1\leq i\leq n$ and with $a_n=\lambda\cdot \tau$ for some non-zero constant $\lambda$, and setting $t_i=\pi(a_i)$ for $1\leq i\leq n-1$, we call the corresponding oriented AG algebra from Construction \ref{con:hat}, Equation \eqref{eq:hatA} the \emph{cohomological blow up of $A$ along $\pi$ with parameters $(t_1,\ldots,t_{n-1},\lambda)$}, or BUG for short, and write
\[
	\tilde{A}=(\hat{A}=)\frac{A[\xi]}{(\xi \cdot K,\underbrace{\xi^n+a_1\xi^{n-1}+\cdots+\lambda\cdot\tau}_{f_A(\xi)})}.
\]
	 with its preferred orientation $\tilde{a}_{soc}=a_{soc}$.
	 The oriented AG algebra from Construction \ref{con:hat}, Equation \eqref{eq:hatT} 
	 \[
	 \tilde{T}=(\tilde{T}=)\frac{T[\xi]}{(\underbrace{\xi^n+t_1\xi^{n-1}+\cdots+\lambda\cdot \pi(\tau))}_{f_T(\xi)}}
	 \]
	 with its preferred orientation $\tilde{t}_{soc}=\xi^{n-1}\cdot t_{soc}$ is called the \emph{exceptional divisor of $T$ with parameters $(t_1,\ldots,t_{n-1},\lambda)$}.  
	 In this case we refer to $A$ as the \emph{cohomological blow down of $\hat{A}$ along $\hat{\pi}$}.
\end{definition}

\begin{remark}
	\label{rem:lambda}
One can force $\lambda=1$ in Definition \ref{def:blowup} by scaling orientations on either $A$ or $T$.  Specifically, given $\pi\colon A\rightarrow T$ with Thom class $\tau$, take new distinguished socle generator either $a_{soc}'=\lambda\cdot a_{soc}$ or $t_{soc}'=\lambda^{-1}\cdot t_{soc}$ so that the same map with one of these scaled orientations $\pi'\colon A'\rightarrow T'$ will have Thom class $\tau'=\lambda\cdot \tau$, and hence $a_n=\lambda\cdot\tau=\tau'$. 

If $\lambda\in\F$ has an $n^{th}$-root, say $\mu$ (e.g. if $\F$ is algebraically closed), then one can also force $\lambda=1$ by rescaling the parameters in the cohomological blow.  Specifically if $\tilde{A}$ is the cohomological blow up of $A$ along $\pi\colon A\rightarrow T$ with parameters $(t_1,\ldots,t_{n-1},\lambda)$, then $\tilde{A}$ is isomorphic to the the cohomological blow up of $A$ along $\pi$ with parameters $(t'_1,\ldots,t'_{n-1},1)$ where $t'_i=\mu^i\cdot t_i$ via the map $\xi\mapsto\mu\cdot \xi$.  
\end{remark}

The following example shows that the hypothesis on $\F$ in Remark \ref{rem:lambda} is necessary.

\begin{example}\label{ex:lambda}
	Let $A=\Q[x]/(x^3)$ and $T=\Q=\F$ with $\pi\colon A\rightarrow T$ the natural projection having $\ker(\pi)=(x)$ and Thom class $\tau=x^2$.  Taking $f_A(\xi)=\xi^2+\lambda x^2$ with $\lambda\in \Q^\times$,  the cohomological blow-up algebra of $A$ along $\pi$ with parameters $(0,\lambda)$ is 
	\[
	\tilde{A}(\lambda)=\frac{\Q[x,\xi]}{(x\xi, \xi^2+\lambda x^2)}.
	\]
We claim that any pair of integers $p,q$ such that $p$ is a prime that does not divide $q$ yield non isomorphic algebras $\tilde{A}(p)$ and $\tilde{A}(q)$. Indeed, assume that there is a $\Q$-algebra isomorphism  $\psi:\tilde{A}(q)\to \tilde{A}(p)$.  Then for some $a,b,c,d\in\Q$ we must have $\psi(x)=ax+b\xi$ and $\psi(\xi)=cx+d\xi$. Moreover we may assume that $a,b,c,d $ are integers such that $\operatorname{gcd}(a,b,c,d)=1$ because $\psi$ is an isomorphism if and only if $\alpha\cdot\psi$ is an isomorphism for any $\alpha\in\Q^\times$. We have
\[
\psi((x\xi, \xi^2+q x^2))=(acx^2+(ad+bc)x\xi+bd\xi^2, (c^2+qa^2)x^2+2(cd+qab)x\xi+(d^2+qb^2)\xi^2).
\]
In order for $\psi$ to be a $\Q$-algebra isomorphism the ideal above must be equal to $(x\xi, \xi^2+px^2)$ and equating the two ideals in $\Q[x,\xi]/(x\xi)$ yields the following equations 
	\begin{align}
	\label{eq:1}
	ac &= pbd\\
	\label{eq:2}
	c^2+qa^2 &= p(d^2+qb^2) .
	\end{align}
It follows from \eqref{eq:1} that $p$ divides $a$ or $c$, and hence from  \eqref{eq:2} that $p$ divides both $a$ and $c$.  Returning to  \eqref{eq:1}, we now deduce that $p$ divides $bd$ and hence $p$ divides $b$ or $d$.  Similarly we must also have $p\mid (d^2+qb^2)$ which implies that $p$ must divide $b$ and $d$.  Therefore we must have $p$ divides $a$, $b$, $c$, and $d$, contradicting our assumption that $\operatorname{gcd}(a,b,c,d)=1$.  

In particular, we have shown that every prime $p$ gives cohomological blow-up of $A$ along $\pi$ with parameters $(0,p)$ which is not isomorphic to the cohomological blow up algebra $\hat{A}(1)$ with parameters $(0,1)$.  In fact this shows that there are infinitely many distinct isomorphism classes of cohomological blow-up algebras $\tilde{A}(p)$ of $A$ along $\pi$ with parameters $(0,p)$, one for each prime $p$.

On the other hand, we shall see in Theorem \ref{thm:flatfam} that the algebras $\tilde{A}(\lambda^2)$ and $\tilde{A}(1)$ are always isomorphic for any rational number $\lambda$. 
\end{example}

The next lemma gives the Thom class of the restriction map from the cohomological blow up to its exceptional divisor.
\begin{lemma}
	\label{lem:Thomhat}
	With notations as in Definition \ref{def:blowup}, if $\tilde{A}$  is the cohomological blow up of $A$ along $\pi\colon A\rightarrow T$ with parameters $(t_1,\ldots,t_{n-1},\lambda)$, and $\tilde{T}$ is its exceptional divisor then the Thom class of the projection map $\tilde{\pi}\colon\tilde{A}\rightarrow\tilde{T}$ is  
	$$\tilde{\tau}=-\lambda^{-1}\xi.$$
\end{lemma}
\begin{proof}
	The socle degree of $\tilde{T}$ is $d-1$, and it follows from Lemma \ref{lem:decomp} that the graded component $\tilde{A}_{d-1}$ has a $\F$-vector space decomposition as
	$$\F\cdot a_T\xi^{n-1}\oplus A_{d-1}, \ \ \text{where} \ \ \pi(a_T)=t_{soc}\in T_k.$$
	We may assume that $k<d-1$, since if $k=d-1$ we clearly have $\tilde{A}=A$ and $\tilde{T}=T$.  Then for each $a\in A_{d-1}$ we have $a\in K$, hence $\xi\cdot a\equiv 0$ in $\tilde{A}$.  Also we have 
	\begin{align*}
	\xi\cdot a_T\xi^{n-1} &\equiv  a_T\xi^n\\
	&\equiv  a_T\left(-\left(a_1\xi^{n-1}+\cdots+a_n\right)\right)\\
	&\quad \text{since} \ a_i\in A_+ \ \Rightarrow \ a_T\cdot a_i\in K\\
	&\equiv a_T\cdot (-a_n) \equiv a_T\cdot (-\lambda\cdot \tau)=-\lambda\cdot a_{soc}=-\lambda\cdot \tilde{a}_{soc},
	\end{align*}
from which it follows that $\xi=-\lambda\cdot\tilde{\tau}$, as desired. 
\end{proof}

The following theorem gives a useful characterization of the blow-up algebra in terms of some universal properties, analogous to Equations \eqref{eq:cMaps}, \eqref{eq:HYtild}, and \eqref{eq:SES} from the Introduction. 

\begin{theorem}
	\label{lem:buchar}
	Suppose that we are given oriented AG algebras $\left(A,\int_A\right)$, $\left(T,\int_T\right)$, $\left(\tilde{A},\int_{\tilde{A}}\right)$, and $\left(\tilde{T},\int_{\tilde{T}}\right)$ with socle degrees $d$, $k$, $d$, and $d-1$, respectively with $d>k$, and surjective degree preserving algebra maps $\pi\colon A\rightarrow T$ and $\tilde{\pi}\colon \tilde{A}\rightarrow \tilde{T}$.  Then $\tilde{A}$ is a cohomological blow up of $A$ along $\pi$ for some parameters $(t_1,\ldots,t_{n-1},\lambda)$, and $\tilde{T}$ is its exceptional divisor if and only if the following conditions are satisfied:
	\begin{enumerate}
		\item There are degree preserving algebra maps $\beta\colon A\rightarrow \hat{A}$ and $\beta_0\colon T\rightarrow \tilde{T}$ making the following diagram commute
		$$\xymatrix{A\ar[r]^-{\beta}\ar[d]_-{\pi} & \hat{A}\ar[d]^-{\tilde{\pi}}\\ T\ar[r]_-{\beta_0} & \tilde{T}}$$
		
		\item The Euler class $\epsilon=\tilde{\pi}(\tilde{\tau})\in \tilde{T}_1$  generates $\tilde{T}$ as a $T$-algebra (via $\beta_0$) with a single homogeneous relation in degree $n=d-k$:
		$$\epsilon^n+\beta_0(t'_1)\epsilon^{n-1}+\cdots+\beta_0(t'_n)\equiv 0$$
		for some homogeneous elements $t'_i\in T_i$ for $1\leq i\leq n$.
		
		\item There is a short exact sequence of $A$-modules
		$$\xymatrix{0\ar[r] & A\ar[r]^-{\beta} & \tilde{A}\ar[r]^-{\tilde{\pi}} & \tilde{T}/\beta_0(T)\ar[r] & 0}$$
	\end{enumerate}
\end{theorem}

\begin{proof}
	Assume that $\left(\tilde{A},\int_{\tilde{A}}\right)$ is the cohomological blow up of $\left(A,\int_A\right)$ along $\pi$ and that $\left(\tilde{T},\int_{\tilde{T}}\right)$ is the exceptional divisor, with parameters $(t_1,\ldots, t_{n-1}, \lambda)$ from Definition \ref{def:blowup}, and their preferred orientations from Definition \ref{def:orientations}.  Then our discussion following Construction \ref{con:hat} shows condition (1.) is satisfied.  Also Lemma \ref{lem:Thomhat} shows that the Thom class of $\tilde{\pi}$ is $\tilde{\tau}=-\lambda^{-1}\xi$, and its Euler class is $\epsilon =\tilde{\pi}(-\lambda^{-1}\xi)$.  From the presentation of $\tilde{T}$ in \eqref{eq:hatT} we have  $f_T(\xi)=\xi^n+t_1\xi^{n-1}+\cdots+\lambda\cdot \pi(\tau)\equiv 0$ in $\tilde{T}$.  Therefore setting $t_i'=(-1)^i\lambda^{-i} t_i$ for $1\leq i\leq n-1$ yields $(-1)^n\lambda^n\left(\epsilon^{n}+t_1'\epsilon^{n-1}+\cdots+t_{n-1}'\epsilon+\lambda^{1-n}\cdot\tau\right)\equiv 0$ in $\tilde{T}$, which is (2.). Finally, Corollary \ref{cor:splitseq} implies condition (3.).
	
	Conversely, assume that conditions (1.), (2.), and (3.) hold.  Define an algebra map
	\begin{equation}
	\label{eq:map}
	\xymatrixrowsep{.1pc}\xymatrix{\phi\colon A[\xi]\ar[r] & \tilde{A}\\
	a\ar@{|->}[r] & \beta(a)\\
	\xi\ar@{|->}[r] & \hat{\tau}}
	\end{equation}
	where $\beta\colon A\rightarrow \tilde{A}$ is the map given by (1.), and $\tilde{\tau}$ is the Thom class of $\tilde{\pi}\colon \tilde{A}\rightarrow\tilde{T}$.  Then conditions (2.) and (3.) guarantee that $\phi$ is surjective.  Indeed by (2.), the quotient $\tilde{T}/\beta_0(T)$ is generated as an $A$ module by non-zero powers of the Euler class $\epsilon=\tilde{\pi}(\tilde{\tau})$, and hence by (3.) $\tilde{A}$ is generated as an $A$ module by $1$ and non-zero powers of the Thom class $\tilde{\tau}$.    Furthermore, note that the ideal generated by $\xi\cdot u$ for $u\in K=\ker(\pi)$ is contained in $\ker(\phi)$.  Indeed, $\phi(u\cdot\xi)=\beta(u)\cdot \tilde{\tau}$ and for any $\tilde{a}\in\tilde{A}$ we have 
	\begin{align*}
	\int_{\tilde{A}}\tilde{\tau}\cdot \beta(u)\cdot\tilde{a}a&= \int_{\tilde{T}}\tilde{\pi}(\beta(u)\cdot\tilde{a})\\
	&= \int_{\tilde{T}}\beta_0(\pi(u))\cdot\tilde{\pi}(\tilde{a})= 0
	\end{align*} 
	which implies that $\beta(u)\tilde{\tau}=\phi(u\cdot\xi)=0$ since $\tilde{A}$ is Gorenstein. 
	
	Consider the relation on the Euler class from (2.):
	$$\epsilon^n+\beta_0(t_1')\epsilon^{n-1}+\cdots+\beta_0(t_{n-1}')\epsilon+\beta_0(t_n')\equiv 0.$$
	For each $1\leq i\leq n-1$, let $a_i$ be any $\pi$-lift of $t_i'$ (which exists since $\pi$ is surjective), and set $g_A(\xi)=\xi^n+a_1\xi^{n-1}+\cdots+a_{n-1}\xi\in A[\xi]$.  Note that $\tilde{\pi}\left(\phi\left(g_A(\xi)\right)\right)=-\beta_0(t_n')\in\beta_0(T)$, and hence according to condition (3.), there exists $a_n\in A$ such that $\beta(a_n)=\phi(g_A(\xi))$ in $\tilde{A}$.  Then setting $f_A(\xi)=g_A(\xi)+a_n$, we see that $f_A(\xi)\in \ker(\phi)$ as well.  Thus we have shown the containment of ideals $(\xi\cdot K,f_A(\xi))\subseteq \ker(\phi)$, and in particular, the map $\phi$ induces a surjective map on the quotient
	$$\bar{\phi}\colon \tilde{A}'=\frac{A[\xi]}{(\xi\cdot K,f_A(\xi))}\rightarrow \tilde{A}.$$ 
	Since $\tilde{A}'$ follows Construction \ref{con:hat}, Lemma \ref{lem:decomp} implies its Hilbert function is as in \eqref{eq:HhatA}:
	\begin{equation}
	\label{eq:A'}
	H(\tilde{A}')=H(\beta'(A))+H(T)[1]+\cdots+H(T)[n-1].
	\end{equation}
	where $\beta'\colon A\rightarrow \tilde{A}'$ is the natural map described after Construction \ref{con:hat}.  By condition (3.), we know that the Hilbert function of $\hat{A}$ is 
	\begin{equation}
	\label{eq:A}
	H(\tilde{A})=H(A)+H(T)[1]+\cdots+H(T)[n-1].
	\end{equation}
	Since $H(\beta(A))_i\leq H(A)_i$ for all $i$, but also $H(\tilde{A}')_i\geq H(A)_i$ for all $i$ by surjectivity of $\bar{\phi}$, we deduce that the Hilbert functions \eqref{eq:A'} and \eqref{eq:A} must be equal, and therefore that $\bar{\phi}\colon \tilde{A}'\rightarrow \tilde{A}$ must be an isomorphism.  Finally since $\tilde{A}$ is Gorenstein, it follows that $\tilde{A}'$ is Gorenstein, and hence must be a cohomological blow up, and the result follows.
\end{proof}

\begin{remark}
	\label{rem:nonsurj}
	As discussed in the Introduction, the conditions of Theorem \ref{lem:buchar} are satisfied by cohomology algebras.  More precisely if $Y\subset X$ are compact complex manifolds of dimension $k<d$, with cohomology algebras $A=H^{2\bullet}(X)$ and $T=H^{2\bullet}(Y)$ and $\pi\colon A\rightarrow T$ the induced restriction map, surjective or not, then the cohomology algebras of the blown up manifolds $\tilde{T}=H^{2\bullet}(\tilde{Y})$ and $\hat{A}=H^{2\bullet}(\tilde{X})$ with (possibly non-surjective) restriction map $\hat{\pi}\colon\hat{A}\rightarrow \tilde{T}$ satisfies the conditions (1.), (2.) and (3.) of Theorem \ref{lem:buchar}. Therefore Theorem \ref{lem:buchar} seems to offer a way to define the cohomological blow up of an oriented AG algebra $A$ along any, possibly non-surjective, map $\pi\colon A\rightarrow T$.  On the other hand, without the surjectivity assumption on $\pi$, one must sacrifice, among other things, the nice presentation given by Construction \ref{con:hat}; see Example \ref{ex:Segre1} and Remark \ref{rem:nostdgrad} in Section \ref{geomsec}.   
\end{remark}

\section{Macaulay Dual Generators}
\label{sec:MD}
Our reference for this section is \cite[Appendix A]{IK}, but see also \cite[Appendix A.2.4]{Eisenbud}.  Let $R=\F[x_1,\ldots,x_r]$ be a polynomial ring, and let $Q=\F[X_1,\ldots,X_r]$ be the dual divided power algebra.  Let $F\in Q_d$ and $G\in Q_k$ be homogeneous forms of degrees $d>k$, respectively, and suppose that $\tau\in R_{d-k}$ is a polynomial for which $G=\tau\circ F$.  Then if $A=R/\Ann(F)$ and $T=R/\Ann(G)$ are the corresponding oriented AG algebras, the identity map on $R$ induces a surjective map on the quotients 
$$\pi\colon A=\frac{R}{\Ann(F)}\rightarrow \frac{R}{\Ann(G)}=T$$
for which the Thom class is $\tau$ as shown in Lemma \ref{lem:MDThom}.

Next, we give another construction, similar to Construction \ref{con:hat}, that will lead to yet another characterization of the cohomological blow up (Theorem \ref{thm:BUMD1}).  First, we need some notation.  Throughout this section we use the notation $\overline{r}$ for the coset of $r\in R$ in the quotient algebra $T$.

\begin{definition}[$G$-dual polynomial]
Given $G$ and $T$ as above, let $\xi$ be an indeterminate, let $n$ be any positive integer and let $f_R(\xi)=\xi^n+a_1\xi^{n-1}+\cdots+a_n\in R[\xi]$ be a homogeneous monic polynomial with $a_i\in R_i$ for $1\leq i\leq n$.  Evaluation at $\xi=1$ gives a non-homogeneous element in $f_R(1)\in R$ and projection to the local ring $T$ gives a (non-homogeneous) unit $\mu_f=\overline{1}+\overline{a_1}+\cdots+\overline{a_n}\in T$.  Let $\mu_h=\overline{1}+\overline{u_1}+\cdots+\overline{u_k}\in T$ be its $T$-inverse, i.e. $\mu_f\cdot \mu_h=\overline{1}$ in $T$, with homogeneous components $\overline{u_i}\in T_i$.  A homogeneous monic polynomial $h_R(\xi)=\xi^k+u_1\xi^{k-1}+\cdots+u_k$ with homogeneous coefficients $u_i\in R_i$ that project to $\overline{u_i}\in T_i$ for every $1\leq i\leq n$ is called a \emph{$G$-dual polynomial for $f_R(\xi)$}.  
\end{definition}

Since $T$ is graded and $\mu_f\cdot\mu_h=\overline{1}$ in $T$, it follows that the polynomials $f_R(\xi)$ and its $G$-dual polynomial $h_R(\xi)$ satisfy  

$$f_R(\xi)\cdot h_R(\xi)\equiv \xi^{n+k} \mod{\Ann_R(G)\cdot R[\xi]}.$$
   
\begin{construction}
	\label{con:hatMD}
	Set $n=d-k$, let $\xi$ be an indeterminate, and let $\Xi$ be its dual divided power.  Choose homogeneous elements $a_i\in R_i$ for $1\leq i\leq n$ and define a homogeneous monic polynomial $f_R(\xi)\in R[\xi]$ by 
	\begin{equation}
	\label{eq:f1}
	f_R(\xi)=\xi^n+a_1\xi^{n-1}+\cdots+a_n.
	\end{equation}
	Let $h_R(\xi)\in R[\xi]$ be any $G$-dual polynomial of $f_R(\xi)$: 
	\begin{equation}
	\label{eq:h}
	h_R(\xi)=\xi^k+u_1\xi^{k-1}+\cdots+u_k.
	\end{equation}
	Construct the $(d-1)$-form $\tilde{G}\in Q[\Xi]$ by 
	\begin{equation}
	\label{eq:Ghat}
	\tilde{G}=h_R(\xi)\circ\left(\Xi^{d-1}\cdot G\right)=\Xi^{n-1}G+\Xi^n(u_1\circ G)+\cdots+\Xi^{d-1}(u_k\circ G),
	\end{equation}  
	and construct the oriented AG algebra
	\begin{equation}
	\label{eq:hatTMD}
	\tilde{T}_{MD}=\frac{R[\xi]}{\Ann(\tilde{G})}.
	\end{equation}
	Next fix a parameter $\lambda\in\F^\times$ and construct the $d$-form $\hat{F}\in Q[\Xi]$ 
	\begin{equation}
	\label{eq:FAhatMD}
	\hat{F}= F- \lambda\cdot \Xi\tilde{G}= F-\lambda\left(\Xi^nG+ \Xi^{n+1}(u_1\circ G)+\cdots+ \Xi^{n+k}(u_k\circ G)\right)
	\end{equation}
	and construct the oriented AG algebra 
	\begin{equation}
	\label{eq:hatAMD}
	\hat{A}_{MD}=\frac{R[\xi]}{\Ann(\hat{F})}.
	\end{equation}
\end{construction} 
In Construction \ref{con:hatMD} one can easily check that we have 

$$\Ann_{R[\xi]}(\hat{F})\cap R=\Ann_R(F), \ \ \text{and} \ \ \Ann_{R[\xi]}(\tilde{G})\cap R=\Ann_R(G)
$$
 
which implies that the inclusion map $R\hookrightarrow R[\xi]$ induces \emph{injective} maps 

$$\beta\colon A=\frac{R}{\Ann_R(F)}\rightarrow \frac{R[\xi]}{\Ann_{R[\xi]}(\hat{F})}=\hat{A}_{MD} \quad \text{ and } \quad
\beta_0\colon T=\frac{R}{\Ann_R(G)}\rightarrow \frac{R[\xi]}{\Ann_{R[\xi]}(\tilde{G})}=\tilde{T}_{MD},
$$ 
Next we observe that $\xi\circ \tilde{F}=\lambda\cdot \tilde{G}$, and hence the identity map $R[\xi]\rightarrow R[\xi]$ passes to a surjective map on the quotient algebras 

$$\hat{\pi}\colon \hat{A}_{MD}=\frac{R[\xi]}{\Ann_{R[\xi]}(\hat{F})}\rightarrow \frac{R[\xi]}{\Ann_{R[\xi]}\left(\lambda^{-1}\xi\circ\hat{F}=-\tilde{G}\right)}=\tilde{T}_{MD}
$$ 
with Thom class $\tilde{\tau}=-\lambda^{-1}\xi$ (compare with Lemma \ref{lem:Thomhat}).  Since the identity and inclusion maps on $R$ and $R[\xi]$ form a commutative square, it follows that the maps on quotients do too, resulting in a commutative diagram
\begin{equation}
\label{eq:CD}
\xymatrix{A\ar[r]^-{\beta}\ar[d]_-{\pi} & \hat{A}_{MD}\ar[d]^-{\hat{\pi}}\\
T\ar[r]_-{\beta_0} & \tilde{T}_{MD};\\}
\end{equation}
compare with condition (1.) of Theorem \ref{lem:buchar}.  The following result is related to condition (2.) of Theorem \ref{lem:buchar}, namely that $\hat{T}_{MD}$ is a free extension of $T$; compare with Lemma \ref{lem:freeExt}.  

\begin{lemma}
	\label{lem:FreeExt}
	Let $G\in Q_k$ be any homogeneous form, let $f_R(\xi)=\xi^n+a_1\xi^{n-1}+\cdots+a_n$ any monic homogeneous polynomial with coefficients $a_i\in R_i$, and let $h_R(\xi)=\xi^k+u_1\xi^{k-1}+\cdots+u_k\in R[\xi]$ be a $G$-dual of $f_R(\xi)$, as in Construction \ref{con:hatMD}.  Then we have the following equality of ideals: 
	$$\Ann\left(h_R(\xi)\circ\left(\Xi^{d-1}\cdot G\right)\right)=\Ann_R(G)\cdot R[\xi]+(f_R(\xi)).$$
	In particular, the oriented AG algebra $\tilde{T}_{MD}$ from Construction \ref{con:hatMD} Equation \eqref{eq:hatTMD} is a free extension over $T$, and in fact is equal to the algebra $\tilde{T}$ from Construction \ref{con:hat} Equation \eqref{eq:hatT}, i.e. 
	$$\tilde{T}_{MD}=\frac{R[\xi]}{\Ann(\tilde{G}=h_R(\xi)\circ \left(\Xi^{d-1}\cdot G\right))}= \frac{R[\xi]}{\Ann_R(G)\cdot R[\xi]+(f_R(\xi))}=\frac{T[\xi]}{(f_T(\xi)=\overline{f_R(\xi)})}=\tilde{T}.$$
\end{lemma}
\begin{proof}
	The containment 
	$$\Ann_R(G)\cdot R[\xi]+(f_R(\xi))\subseteq \Ann_{R[\xi]}\left(\tilde{G}=h_R(\xi)\circ\left(\Xi^{d-1}\cdot G\right)\right)$$
	follows from the relation 
	$$f_R(\xi)\cdot h_R(\xi)\equiv \xi^d \ \text{mod} \ \Ann_R(G)\cdot R[\xi].$$
	Therefore the identity map $R[\xi]\rightarrow R[\xi]$ passes to a surjective map of algebras
	$$\phi\colon \tilde{T}=\frac{T[\xi]}{(\overline{f(\xi)})}=\frac{R[\xi]}{\Ann_R(G)\cdot R[\xi]+(f_R(\xi))}\rightarrow \frac{R[\xi]}{\Ann_{R[\xi]}(\tilde{G})}=\tilde{T}_{MD}.$$
	Let $t_{soc}\in R_k$ be any homogeneous polynomial which projects onto the socle generator of $T$, so that the distinguished socle generator of $\tilde{T}$ is $\overline{t_{soc}\cdot \xi^{n-1}}$.  Since 
	$$t_{soc}\cdot \xi^{n-1}\cdot h_R(\xi)\circ\left(\Xi^{d-1}\cdot G\right)=1$$
	it follows from Lemma \ref{lem:soclefits} that $\phi$ must also be injective and hence an isomorphism, and the result follows.
\end{proof}

\begin{remark}
	\label{rem:SSThom}
	In \cite{SmithStong}, the authors refer to the monic polynomial $f_R(\xi)$ as the \emph{homogenizing polynomial} for $T$ and to the $G$-dual $h_R(\xi)$ as the \emph{dual homogenizing polynomial}.  We refer also to \cite{BJMR} for more on de-homoginization . 	We also remark monic polynomial $f_R(\xi)$ and its $G$-dual polynomial $h_R(\xi)$ also satisfy
	\begin{equation}
	\label{eq:hatfhalt}
	f_R(\xi)\cdot h_R(\xi)\circ\left(\Xi^{d}\cdot G\right)=f_R(\xi)\circ\left(\Xi\cdot\hat{G}\right)=G;
	\end{equation}
	in particular, $f_R(\xi)$ is the Thom class of the projection map 
	$$\pi_B\colon B:=\frac{R[\xi]}{\Ann_{R[\xi]}\left(\Xi\cdot\tilde{G}\right)}\rightarrow \frac{R[\xi]}{\Ann_{R[\xi]}(G)}=T.$$ 
	The algebra $B$ defined above will show up again when we discuss connected sums.
\end{remark}

Lest the reader think that the AG algebra 
$$\tilde{T}=\frac{R[\xi]}{\Ann_{R[\xi]}\left(h_R(\xi)\circ\left(\Xi^{d-1}\cdot G\right)\right)}$$
is always a free extension over $T=R/\Ann_R(G)$, the following example shows otherwise.
\begin{example}
	\label{ex:NotPBI}
	Let $R=\F[x,y,z]$, $Q=\F[X,Y,Z]$, and let $G=XYZ$ with $k=\deg(G)=3$.  Suppose that we choose $n=2$, and $h_R(\xi)=\xi^3+(xy+xz+yz)\xi+(xyz)$ and define $\tilde{G}$ as in Construction \ref{con:hatMD} Equation \eqref{eq:Ghat}: 
	\begin{align*}
	\tilde{G}= & \Xi XYZ+\Xi^3(X+Y+Z)+\Xi^4 =h(\xi)\circ\left(\Xi^4G\right).
	\end{align*}
Note that there is no monic polynomial $f_R(\xi)\in \Ann_{R[\xi]}(\tilde{G})$ of degree $n=2$, and in fact we have 
$$\tilde{T}=\frac{R[\xi]}{\Ann(\tilde{G})}=\frac{\F[x,y,z,\xi]}{(x^2,y^2,z^2, \xi xyz-\xi^3(x+y+z), \xi^3(x+y+z)-\xi^4, \xi^5)}$$
which is not a free extension over $T=R/\Ann(G)$.

Note, however, that the corresponding unit $\overline{h(1)}\in T$ is $\mu_h=\overline{1}+\overline{(xy+xz+yz)}+\overline{(xyz)}$ and its $T$-inverse is $\left(\mu_h\right)^{-1}=\overline{1}-\overline{(xy+xz+yz)}-\overline{(xyz)}$, indicating that the ``correct choice'' is $f_R(\xi)=\xi^3-(xy+xz+yz)\xi-xyz$ of degree $n=3$ (not $n=2$). Indeed we see that if we define the homogeneous form of degree $5$
$$\tilde{G}'=\Xi^2XYZ+\Xi^4(X+Y+Z)+\Xi^5=h_R(\xi)\circ\left(\Xi^5G\right)$$
then 
$$\tilde{T}'=\frac{R[\xi]}{\Ann_{R[\xi]}(\tilde{G}')}=\frac{R[\xi]}{\Ann_{R}(G)\cdot R[\xi]+(f_R(\xi))}=\frac{T[\xi]}{\left(f_T(\xi)=\xi^3-\overline{(xy+xz+yz)}\xi-\overline{(xyz)}\right)}$$
which is a free extension over $T=R/\Ann(G)$ with fiber $F=\F[\xi]/(\xi^3)$.
The moral of the story here is that given a $k$-form $G\in Q_k$ and $T=R/\Ann(G)$, then in order to construct the dual generator of a free extension over $T$ by the formula $h_R(\xi)\circ\left(\Xi^{n+k-1}\cdot G\right)\in Q_{d-1}$, we can either choose the integer $n$ and the monic polynomial $f_R(\xi)$ of degree $n$ and then take $h_R(\xi)$ as its $G$-dual, or we can choose the monic polynomial $h_R(\xi)$ of degree $k$, then take its $G$-dual $f_R(\xi)$ and take $n=\deg(f_R(\xi))$, but we cannot necessarily choose the integer $n$ and the polynomial $h_R(\xi)$ simultaneously.  Of course the $\tilde{G}$ of Construction \ref{con:hatMD} Equation \eqref{eq:Ghat} follows the former procedure and, by Lemma \ref{lem:FreeExt}, is always the dual generator of a free extension of $T$.   
\end{example}

The following result gives necessary and sufficient conditions for $\hat{A}_{MD}$ to be a cohomological blow-up algebra of $A$ along $\pi\colon A\rightarrow T$ in the sense of Definition \ref{def:blowup}.  
\begin{theorem}
	\label{thm:BUMD1}
	 Let $f_R(\xi)$, $h_R(\xi)$, $\tilde{G}$, $\tilde{T}_{MD}$, $\lambda$, $\hat{F}$ and $\hat{A}_{MD}$ be as in Construction \ref{con:hatMD}.  Then the following statements are equivalent:
	\begin{enumerate}
		\item The algebra $\hat{A}_{MD}$ is isomorphic to a cohomological blow up of $A$ along $\pi$ with some parameters $(t_1,\ldots,t_{n-1},\lambda)$. 		
		\item The Hilbert function of $\hat{A}_{MD}$ satisfies 
		$$H(\hat{A}_{MD})=H(A)+ H(T)[1]+\cdots+ H(T)[n-1].$$
		
		\item There exists an element $r\in R$ for which $f_R(\xi)-r\in \Ann_{R[\xi]}(\tilde{F})$.
		
		\item The constant coefficient $r_n$ of $f_R(\xi)$ satisfies
		$$\left(r_n-r\right)\circ F=\lambda\cdot G$$
		for some $r\in\Ann_R(G)$.
	\end{enumerate}
\end{theorem}
\begin{proof}
	(1.) $\Rightarrow$ (2.).  Assume that $\hat{A}_{MD}$ is isomorphic to a cohomological blow up of $\pi\colon A\rightarrow T$. Then Corollary \ref{cor:splitseq} gives (2.).
	
	(2.) $\Rightarrow$ (3.).  Assume that $H(\hat{A}_{MD})=H(A)+ H(T)[1]+\cdots+ H(T)[n-1]$.  Then the sequence of maps 
	$$\xymatrix{0\ar[r] & A\ar[r]^-\beta & \hat{A}_{MD}\ar[r]^-{\hat{\pi}} & \tilde{T}/\beta_0(T)\ar[r] & 0}$$
	is exact by the argument in the proof of Corollary \ref{cor:splitseq}.  Let $f_A(\xi)$ be the equivalence class of $f_R(\xi)\in R[\xi]$ in $\hat{A}_{MD}$.  Note that by Lemma \ref{lem:FreeExt}, $f_A(\xi)$ must be in the kernel of $\hat{\pi}$, and hence by exactness, there exists $\bar{r}\in A=R/\Ann(F)$ for which $\beta(\bar{r})=f_A(\xi)$.  Then if $r\in R$ be any homogeneous lift of $\bar{r}$, we have $f_R(\xi)-r\in \Ann(\hat{F})$ which is (3.).
	
	(3.) $\Rightarrow$ (4.). 
		 Assume that there exists $r\in R$ for which $f_R(\xi)-r\in\Ann(\hat{F})$.  Then we have 
	\begin{align}
	\label{eq:fxir}
	\left(f_R(\xi)-r\right)\circ \hat{F}= & \left(f_R(\xi)-r\right)\circ \left(F-\lambda\cdot h_R(\xi)\circ\left(\Xi^{n+k}\cdot G\right)\right)\\
	\nonumber= & f_R(\xi)\circ F-r\circ F-\lambda\cdot \left(f_R(\xi)\cdot h_R(\xi)\right)\circ (\Xi^{n+k}\cdot G)+\lambda\cdot \Xi\cdot r\circ\tilde{G}\\
	\nonumber= & r_n\circ F-r\circ F-\lambda\cdot G+\lambda \cdot \Xi\cdot r\circ\tilde{G}\\
	\nonumber= & (r_n-r)\circ F-\lambda\cdot G+\lambda\cdot \Xi\cdot r\circ \tilde{G}=  0,
	\end{align}  
	which implies, by comparing $\Xi$-coefficients, that 
	$$(r_n-r)\circ F-\lambda\cdot G=0.$$
	Moreover since $f_R(\xi)-r\in\Ann(\hat{F})\subseteq \Ann(\tilde{G}=\xi\circ \hat{F})$, and $f_R(\xi)\in\Ann(\tilde{G})$, we see that $r\in\Ann(\tilde{G})$ too, which implies that $r\in \Ann_{R[\xi]}(\tilde{G})\cap R=\Ann_R(G)$, and (4.) follows. 

	(4.) $\Rightarrow$ (1.).  Assume that there exists $r\in \Ann(G)$ such that the constant coefficient of $f_R(\xi)-r\in R[\xi]$ satisfies $(r_n-r)\circ F=\lambda\cdot G$.  Then we claim that $f_R(\xi)-r\in \Ann(\hat{F})$.  Indeed as in \eqref{eq:fxir} we have 
	$$\left(f_R(\xi)-r\right)\circ \hat{F}= (r_n-r)\circ F - \lambda\cdot G-\lambda\cdot \Xi\cdot r\circ\tilde{G}
	=-\Xi\left(\lambda r\circ\tilde{G}\right)=0$$
	since $r\circ\tilde{G}=h_R(\xi)\circ\left(\Xi^d\cdot (r\circ G)\right)=0$.  
	
	Next observe that for any $r'\in\Ann_R(G)=\Ann_{R[\xi]}(\tilde{G})\cap R$ we have $\xi\cdot r'\in \Ann_{R[\xi]}(\hat{F})$.  In particular then, the kernel of the surjective map 
	$$\Phi\colon R[\xi]\rightarrow \frac{R[\xi]}{\Ann_{R[\xi]}(\hat{F})}=\hat{A}_{MD}$$
	contains $\Ann_R(G)\cdot\xi R[\xi]$, and $f_R(\xi)-r$, not to mention the ideal $\Ann_R(F)\cdot R[\xi]$.  Set $I\subset R[\xi]$ to be the sum of these components, i.e. $I=\Ann_R(F)\cdot R[\xi]+\Ann_R(G)\cdot\xi R[\xi]+(f_R(\xi)-r)$.  Then clearly we have $R[\xi]/I\cong A[\xi]/(\xi\cdot K,f_A(\xi)=\overline{f_R(\xi)-r})$ where $K=\ker(\pi)=\Ann_R(G)\cdot R[\xi]/\Ann_R(F)\cdot R[\xi]$.  Moreover if $t_i\in T_i$ is the equivalence class of $r_i\in R_i$, then since the equivalence class of $r_n-r\in R_n$ in $A$ is $\lambda\cdot \tau\in A_n$, it follows that $R[\xi]/I$ is the cohomological blow up of $A$ along $\pi$ with  parameters $(t_1,\ldots,t_{n-1},\lambda)$.  Therefore we have a surjective map of Gorenstein algebras
	$$\bar{\Phi}\colon \hat{A}=R[\xi]/I=\frac{A[\xi]}{(\xi\cdot K,f_A(\xi))}\rightarrow \hat{A}_{MD}.$$
	Since $\hat{A}$ and $\hat{A}_{MD}$ have the same socle degree, $\bar{\Phi}$ must be an isomorphism by Lemma \ref{lem:soclefits}, which is (1.).
\end{proof}

The preceding Theorem \ref{thm:BUMD1} implies that the algebras $\hat{A}$ from Construction \ref{con:hat} Equation \eqref{eq:hatA} and $\hat{A}_{MD}$ from Construction \ref{con:hatMD} Equation \eqref{eq:hatAMD} are equal precisely when they are cohomological blow ups; in this case we shall replace ``hat'' with ``tilde'' on all symbols, and write
$$\tilde{A}=\hat{A}_{MD}=\frac{R[\xi]}{\Ann\left(\tilde{F}=F-\Xi\cdot\left(\tilde{G}=h_R(\xi)\circ\left(\Xi^{d-1}\cdot G\right)\right)\right)}=\frac{A[\xi]}{\left(\xi\cdot K, \overline{f_R(\xi)}\right)}=\hat{A}.$$ 
The following gives an example in which $\hat{A}\neq \hat{A}_{MD}$.

\begin{example}
\label{ex:chris}
Let $F=X^2Y^2$ and $G=XY$ so that
$$\pi\colon A=\frac{\F[x,y]}{\Ann(X^2Y^2)}\rightarrow \frac{\F[x,y]}{\Ann(XY)}=T, \ \ \pi(x)=x, \ \pi(y)=y$$ 
is the natural projection map of corresponding AG algebras with Thom class $\tau=xy$; here $d=4$, $k=2$, and $n=2$.  Taking $f_R=\xi^2$ we see that its $G$-dual is itself, i.e. $h_R(\xi)=\xi^2$, and hence setting $\tilde{G}=h_R(\xi)\circ \left(\Xi^{d-1}G\right)=\xi^2\circ\left(\Xi^3\cdot XY\right)=\Xi XY$, we get 
$$\tilde{T}=\frac{\F[x,y,\xi]}{\Ann(\Xi XY)}=\frac{\F[x,y,\xi]}{(x^2,y^2,\xi^2)}.$$
Then taking $\lambda=1$, Construction \ref{con:hatMD} yields $\hat{F}=F-\Xi\tilde{G}=X^2Y^2-\Xi^2XY$, and hence
\[
\hat{A}_{MD}=\frac{\F[x,y,\xi]}{\Ann(X^2Y^2-\Xi^2 XY)}=
\frac{\F[x,y,\xi]}{(x^3,y^3,x^2\xi, y^2\xi,x(\xi^2+xy),y(\xi^2+xy),\xi^3)}.
\]
Note in this case, there is no $r\in\Ann(G=XY)$ for which $f_R(\xi)-r=\xi^2-r\in\Ann(\hat{F})$.  We can also compute the Hilbert function $H(\hat{A}_{MD})=(1,3, 6, 3, 1)$ whereas $H(A)+H(T)[1]=(1, 3, 5, 3, 1)$, and hence $H(\hat{A}_{MD})\neq  H(A)+H(T)[1]$, and therefore $\hat{A}_{MD}$ is not a cohomological blow up of $A$ along $\pi$, by Theorem \ref{thm:BUMD1} (or by Theorem \ref{lem:buchar}).  Moreover taking $f_A(\xi)=\overline{f_R(\xi)}=\xi^2$, Construction \ref{con:hat} yields 
$$\hat{A}=\frac{A[\xi]}{(\xi\cdot K, f_A(\xi)=\xi^2)}=\frac{\F[x,y,\xi]}{(x^3,y^3,\xi x^2,\xi y^2,\xi^2)}$$
which is not even Gorenstein; in particular, $\hat{A}\neq \hat{A}_{MD}$.

Alternatively, a ``correct choice'' is $f_R(\xi)=\xi^2-xy$ for which a $G$-dual polynomial is $h_R(\xi)=\xi^2+xy$ yielding $\tilde{G}=(\xi^2+xy)\circ\left(\Xi^3XY\right)=\Xi XY+\Xi^3$ and 
$$\tilde{T}=\frac{\F[x,y,\xi]}{\Ann(\Xi XY+\Xi^3)}=\frac{\F[x,y,\xi]}{(x^2,y^2,\xi^2-xy)}.$$
In this case, taking $\lambda=-1$ (Theorem \ref{thm:BUMD1} requires it!), Construction \ref{con:hatMD} yields $\tilde{F}=X^2Y^2+\Xi\left(\Xi XY+\Xi^3\right)$ with 
$$\tilde{A}=(\hat{A}_{MD}=)\frac{\F[x,y,\xi]}{\Ann(\Xi^2XY+\Xi^4+X^2Y^2)}=\frac{\F[x,y,\xi]}{(x^3,y^3,\xi x^2,\xi y^2,\xi^2-xy)}.$$
Here we can verify that the conditions of Theorem \ref{thm:BUMD1} are satisfied, for example the Hilbert function is 
$H(\tilde{A})=(1,3,5,3,1)=H(A)+H(T)[1]$.  Hence in this case, $\tilde{A}(=\hat{A}_{MD})$ is the cohomological blow up of $A$ along $\pi$ with parameters $(t_1,\lambda)=(0,-1)$.  In this case, taking $f_A(\xi)=\overline{f_R(\xi)}=\xi^2-xy$, then Construction \ref{con:hat} yields the same algebra: 
$$\tilde{A}=(\hat{A}=)\frac{A[\xi]}{(\xi\cdot K,f_A(\xi))}=\frac{\F[x,y,\xi]}{(x^3,y^3,\xi x^2,\xi y^2, \xi^2-xy)}.$$
\end{example}

\section{Relation With Connected Sums}
\label{sec:connsum}
\subsection{The Cohomological Blow-up Algebra as a Connected Sum}
In this section we relate cohomologiocal blow-up algebras to a different algebraic construction termed connected sum, an algebraic analogue of the better known topological construction by the same name. In the topogical construction, a connected sum is  obtained by gluing two $2d$-dimensional manifolds $M_1$ and $M_2$ along diffeomorphic tubular neighborhoods of a common $2k$-dimensional submanifold $N$. In complex geometry it is well known that the connected sum of an $n$-dimensional complex manifold $M$ with a projective space $\P^n$ is diffeomorphic to the blow up of $M$ at a point \cite[p.101]{Huybrechts}. Our contribution here is to recognize that {\em any} cohomological blow-up algebra along a surjective map is a connected sum in the algebraic sense; see Theorem \ref{thm:BUCS}.  We now recall the algebraic connected sum construction as defined in \cite{IMS}, based on the original construction defined in \cite{AAM}.

\begin{definition}[Fibered product, connected sum]
\label{def:connsum}
  Given oriented AG algebras $\left(A,\int_A\right)$, $\left(B,\int_B\right)$,  another AG algebra $\left(T,\int_T\right)$ and algebra maps $\pi_A\colon A\rightarrow T$ and $\pi_B\colon B\rightarrow T$ one forms the \emph{fibered product algebra} as the sub algebra $A\times_TB\subseteq A\times B$ of the direct product algebra given by
$$A\times_TB=\left\{(a,b)\in A\times B \ | \ \pi_A(a)=\pi_B(b)\right\}.$$

If $A, B$ have the same socle degree $d$, $\pi_A$ and $\pi_B$ have Thom classes $\tau_A$ and $\tau_B$ respectively, and the Euler classes $\pi_B(\tau_B)=\pi_A(\tau_A)$ are equal, then the \emph{total Thom class} $(\tau_A,\tau_B)$ is in the fibered product algebra $A\times_TB$ , and we define the \emph{connected sum algebra} as the quotient of the fibered product by the principal generated by the total Thom class, i.e.
$$A\#_TB=\frac{A\times_TB}{\left\langle (\tau_A,\tau_B)\right\rangle}.$$
\end{definition}

If the projection maps $\pi_A$ and $\pi_B$ are both surjective, then the connected sum algebra defined above is an AG algebra of the same socle degree $d$ as $A$ and $B$, see e.g.~\cite[Lemma 3.8]{IMS}.

In terms of Macaulay duality, the following result in \cite[Theorem 4.6]{IMS} gives a useful criterion to recognize an AG algebra as a connected sum.

\begin{proposition}
	\label{thm:CS}
	Let $R=\F[x_1,\ldots, x_r]$ and let $F,H\in R_d$ be two linearly independent homogeneous forms of degree $d$. Suppose that there exists $\sigma\in R_{d-k}$ (for some $k<d$) satisfying
	\begin{enumerate}
		\item $\sigma\circ F=\sigma\circ H\neq 0$, and 
		\item $\Ann(\sigma\circ F=\sigma\circ H)=\Ann(F)+\Ann(H)$.
	\end{enumerate} 
In this case, set
$$A=\frac{Q}{\Ann{F}}, \ B=\frac{Q}{\Ann(H)}, \ T=\frac{Q}{\Ann(\sigma\circ F=\sigma\circ H)},$$
and let $\pi_A\colon A\rightarrow T$ and $\pi_B\colon B\rightarrow T$ be the natural projection maps.  Then the Thom classes of $\pi_A$ and $\pi_B$ are given by $\tau_A=\sigma+\Ann(F)$ and $\tau_B=\sigma+\Ann(H)$, and we have algebra isomorphisms
$$A\times_TB\cong \frac{Q}{\Ann(F)\cap\Ann(H)}, \ \ A\#_TB\cong \frac{Q}{\Ann(F-H)}.$$
And, conversely, every connected sum $A\#_T B$ of graded AG algebras of the same socle degree over graded AG algebra $T$ arises in this way.
\end{proposition}   

We utilize this result to realize a cohomological blow-up algebra defined as in Construction \ref{con:hatMD} as a connected sum.

\begin{example}
	\label{ex:blowup1}
	Let $F=X^2Y^2$ and $G=Y$ and $\pi$ the projection between their corresponding AG algebras
	$$\pi\colon A=\frac{\F[x,y]}{\Ann(X^2Y^2)}\rightarrow \frac{\F[x,y]}{\Ann(Y)}=T,$$
	with Thom class $\tau=x^2y$.  Here $H(A)=(1,2,3,2,1)$, $ H(T)=(1,1)$, $d=4$, $k=1$, and $n=3$.  Take $h_R(\xi)=\xi+y$ with corresponding unit $\mu_h=1+y\in T$, and inverse $\left(\mu_h\right)^{-1}=1-y$, so that $f_R(\xi)=\xi^3-\xi^2y$.  Then 
	$$\tilde{G}=\left(\xi+y\right)\circ\left(\Xi^3Y\right)=\Xi^2 Y+\Xi^3.$$  
	Choosing $\lambda=-2$, we have 
	$$\tilde{F}=X^2Y^2-2\Xi\left(\Xi^2 Y+\Xi^3\right)=X^2Y^2-2\Xi^3Y-2\Xi^4.$$
	Setting $r=-2x^2y\in\Ann_R(G)$, we have $\hat{f}_R(\xi)=f_R(\xi)-r=\xi^3-\xi^2y+2x^2y\in\Ann_{R[\xi]}(\tilde{F})$, and hence 
	$$\tilde{A}=\frac{\F[x,y,\xi]}{\Ann(\tilde{F})}=\frac{\F[x,y,\xi]}{(x^3,y^3,\xi x, \xi y^2, \xi^3-\xi^2y+2x^2y)}$$
	is according to Theorem \ref{thm:BUMD1} the cohomological blow up of $A$ along $\pi$ with parameters $(\overline{-y}, 0, ,0, 2)$, with Hilbert function $H(\tilde{A})=H(A)+H(T)[1]+H(T)[2]=(1,3,5,3,1)$.
	
	Setting $H=2\Xi^3Y+2\Xi^4$, we have $\tilde{F}=F-H$.  Moreover the element $\tau=\hat{f}_R(\xi)$ satisfies $\tau\circ H=\tau\circ F=Y$ and $\Ann_{R[\xi]}(Y)=\Ann_{R[\xi]}(F)+\Ann_{R[\xi]}(H)$.  Hence it follows from Proposition \ref{thm:CS} that we can also realize $\tilde{A}$ as a connected sum of $A=R[\xi]/\Ann(F)$ and $B=R[\xi]/\Ann(H)$ over $T=R/\Ann(Y)$, i.e.	$\tilde{A}\cong A\#_TB.$
\end{example}
	
In the following we show that Example \ref{ex:blowup1} is an instance of a general phenomenon: any cohomological blow-up algebra is a connected sum.  For the remainder of this section we work under the following set up.  Fix a $d$-form $F\in Q_d$ and a $k$-form $G\in Q_k$ with $d>k$, and let $A=R/\Ann(F)$ and $T=R/\Ann(G)$ be the associated oriented AG algebras.  Assume there exists homogeneous polynomial $\tau\in R_{d-k}$ for which $\tau\circ F=G$, so that the natural projection map $\pi_A\colon A\rightarrow T$ has Thom class $\tau_A=\overline{\tau}$.  Let $n=d-k$, let $\xi$ be an indeterminate, $\Xi$ a divided power variable dual to $\xi$, and fix a monic homogeneous polynomial $f_R(\xi)=\xi^n+r_1\xi^{n-1}+\cdots+r_n\in R[\xi]$ with coefficients $r_i\in R_i$ such that $r_n=\lambda\cdot\tau$ for some nonzero constant $\lambda\in\F^\times$.  Let $h_R(\xi)=\xi^k+u_1\xi^{k-1}+\cdots+u_k\in R[\xi]$ be a $G$-dual polynomial of $f_R(\xi)$. Define the $(d-1)$-form $\tilde{G}$ and its associated oriented AG algebra $\tilde{T}$ by
\[ \tilde{G}=h_R(\xi)\circ\left(\Xi^{d-1}\cdot G\right), \qquad \tilde{T}=\frac{R[\xi]}{\Ann_{R[\xi]}(\tilde{G})}.\]
Define the $d$-form $\tilde{F}$ and its associated oriented AG algebra $\tilde{A}$ 
 \[ \tilde{F}=F-\lambda\cdot \Xi\cdot \tilde{G}, \qquad \tilde{A}=\hat{A}_{MD}=\frac{R[\xi]}{\Ann_{R[\xi]}(\tilde{F})}\]
 as in construction \ref{con:hatMD}; note that Theorem \ref{thm:BUMD1} guarantees that $\hat{A}_{MD}=\tilde{A}$ here.
Define the new $d$-form $H$ and its associated AG algebra  $B$
$$ H=\lambda\cdot \Xi\cdot \tilde{G}\in Q[\Xi]_d,  \qquad B=\frac{R[\xi]}{\Ann_{R[\xi]}\left(H\right)}$$
and recall that 
$$H=\lambda\cdot \Xi\cdot \tilde{G}=\lambda\cdot h_R(\xi)\circ\left(\Xi^d\cdot G\right).$$
It follows that $h_R(\xi)$ is also a $G$-dual polynomial for $\xi\cdot f_R(\xi)$ of degree $n+1$, and hence by Lemma \ref{lem:FreeExt} we have 
$$B=\frac{R[\xi]}{\Ann_{R[\xi]}\left(H\right)}=\frac{R[\xi]}{\Ann_{R[\xi]}(\Ann_R(G)\cdot R[\xi]+(\xi\cdot f_R(\xi))}=\frac{T[\xi]}{\left(\overline{\xi\cdot f_R(\xi)}\right)}.$$
Thus the distinguished socle generator of $B$ as $b_{soc}=\lambda^{-1}\cdot \xi^n\cdot t_{soc}$.  The evaluation $\xi=0$ passes to a map on quotients $\pi_B\colon B\rightarrow T$, and from the identity \eqref{eq:hatfhalt} it follows that its Thom class is 
$$\tau_B=\lambda^{-1}\cdot f_R(\xi).$$
Note that the Euler class of $\pi_B$ is $\pi_B(\lambda^{-1}\cdot f_R(\xi))=\tau$ which is equal to the Euler class of $\pi_A$, and hence it makes sense to form the connected sum $A\#_TB$.

\begin{theorem}
	\label{thm:BUCS}
	The connected sum of $A$ and $B$ over $T$ is equal to the cohomological blow up of $A$ along $\pi$ with parameters $(\overline{a_1},\ldots,\overline{a_{n-1}},\lambda)$, i.e.
	$$\tilde{A}=A\#_TB.$$ 
\end{theorem}
\begin{proof}
	With our setup above, Theorem \ref{thm:BUMD1} implies that $\tilde{A}$ is the cohomological blow up of $A$ along $\pi$ with parameters $(\overline{a_1},\ldots,\overline{a_{n-1}},\lambda)$, i.e.
	$$\hat{A}_{MD}=\tilde{A}=\frac{R[\xi]}{\Ann_{R[\xi]}\left(\tilde{F}=F- H\right)}.$$
	Setting $\sigma=f_R(\xi)$  in Proposition \ref{thm:CS} it then suffice to check that conditions (1.) and (2.) hold.  Condition (1.) holds since we have 
	$$f_R(\xi)\circ F=(\xi^{n}+r_1\xi^{n-1}+\cdots+\lambda\cdot\tau)\circ F=\lambda\cdot G=f_R(\xi)\circ \left(H\right).$$
	It remains to see that 
	\begin{equation}
	\label{eq:remains}
	\Ann_{R[\xi]}(G)=\Ann_{R[\xi]}(H)+\Ann_{R[\xi]}(F).
	\end{equation}
	By Lemma \ref{lem:FreeExt} we have 
	$$\Ann_{R[\xi]}(H)=\Ann_R(G)\cdot R[\xi]+(\xi\cdot f_R(\xi)).$$
	Also, since $F$ is independent of $\xi$ we have 
	$$\Ann_{R[\xi]}(F)=\Ann_R(F)\cdot R[\xi]+(\xi).$$ 
	Since $\Ann_R(F)\subset\Ann_R(G)$ and 
	$$\Ann_{R[\xi]}(G)=\Ann_R(G)\cdot R[\xi]+(\xi),$$ 
	\eqref{eq:remains} follows, and the desired conclusion follows by Proposition \ref{thm:CS}.	
\end{proof}

\begin{remark}
	\label{rem:stdg}
	In general, fibered products and connected sums of standard graded AG algebras need not be standard graded, even in the simplest cases. Examples illustrating this appear for the fibered product in \cite[Example 4.5]{IMS}, and for the connected sum in \cite[Proposition 5.22]{IMS}. Theorem~\ref{thm:BUCS} distinguishes cohomological blow-up algebras of surjective maps as a class of connected sums which preserves the standard grading.  
However there are examples from geometry where $A$ and $T$ are standard graded, but the restriction map $\pi\colon A\rightarrow T$ is not surjective, and the cohomological blow up has a non-standard grading; see Remark \ref{rem:nostdgrad}. 
\end{remark}

\subsection{The Blow Down as a Connected Sum}
Continuing with our set up above, let $\tilde{\pi}_A\colon\tilde{A}\rightarrow \tilde{T}$ be the projection map with Thom class $\tilde{\tau}_A=-\lambda^{-1}\xi$. Set $H=\lambda\cdot \Xi\cdot \tilde{G}$ and consider the surjective map of AG algebras 
$$\tilde{\pi}_B: \tilde B: =\frac{R[\xi]}{\Ann(-H=-\lambda\cdot \Xi\cdot \tilde{G})} \to \frac{R[\xi]}{\Ann(\tilde{G})}= \tilde{T}.$$
The algebra $\tilde{B}$ is the AG algebra $B$ from the previous subsection, but with orientation reversed, i.e. $\tilde{b}_{soc}=-b_{soc}$.  Thus the Thom class of the map $\tilde{\pi}_B$ is $\tilde{\tau}_B=-\lambda^{-1}\xi$ which is equal to the Thom class of $\tilde{\pi}_A$. 
Then it makes sense to form the connected sum $\tilde{A}\#_{\tilde{T}}\tilde{B}$.

\begin{theorem}
	\label{thm:bdcs}
	The connected sum of $\tilde{A}$ and $\tilde{B}$ over $\tilde{T}$ is equal to $A$, i.e.
	$$A=\tilde{A}\#_{\tilde{T}}\tilde{B}.$$
\end{theorem}  
\begin{proof}
	Write 
	$$A=\frac{R[\xi]}{\Ann_{R[\xi]}\left(F=\tilde{F}-(-H)\right)}=\frac{R[\xi]}{\Ann_{R}(F)\cdot R[\xi]+(\xi)}.$$
	Setting $\sigma=\xi$ in Proposition \ref{thm:CS}, we will show that conditions (1.) and (2.) hold.  Condition (1.) holds because  
	$$\xi\circ \tilde{F}=-\lambda\cdot \tilde{G}=\xi\circ (- H).$$  
	For Condition (2.) note that it follows from Theorem \ref{thm:BUMD1} and Construction \ref{con:hat} that
	$$\tilde{A}\cong \frac{A[\xi]}{\left(\xi\cdot K,\overline{f_R(\xi)}\right)}=\frac{R[\xi]}{\Ann_R(F)\cdot R[\xi]+\xi\cdot \Ann_R(G)\cdot R[\xi]+(f_R(\xi))}$$
	and hence that  
	$$\Ann_{R[\xi]}(\tilde{F})=\Ann_R(F)\cdot R[\xi]+\xi\cdot \Ann_R(G)\cdot R[\xi]+(f_R(\xi)).$$ 
	Also from Lemma \ref{lem:FreeExt} we have
	$$\Ann_{R[\xi]}(H)=\Ann_R(G)\cdot R[\xi]+(\xi\cdot f_R(\xi)),$$ 
	and since $\Ann_R(F)\subseteq\Ann_R(G=\tau\circ F)$, their sum satisfies
	$$\Ann_{R[\xi]}(\tilde{F})+\Ann_{R[\xi]}(H)=\Ann_R(G)\cdot R[\xi]+(f_R(\xi))=\Ann_{R[\xi]}(\tilde{G})$$
	which is condition (2).  Hence the result follows from Proposition \ref{thm:CS}.
\end{proof}

One interesting consequence of Theorem \ref{thm:bdcs} is that every AG algebra has a nontrivial connected sum decomposition over some algebra $\tilde{T}$.  This stands in direct contrast to connected sums over the ground field $T=\F$, where $\#_\F$-indecomposable AG algebras exist; see \cite[Theorem 8.3]{AAM} and also \cite[Proposition 3.1]{SmithStong2}.  On the other hand, Theorem \ref{thm:bdcs} shows that over general $T$, there may be no $\#_T$-indecomposable AG algebras.

\begin{example}
	\label{ex:13631}
	Set $F=Z^2XY-X^2Y^2$ so that the corresponding AG algebra is
	$$A=\frac{\F[x,y,z]}{\Ann(Z^2 XY-X^2Y^2)}=
	\frac{\F[x,y,z]}{(x^3,y^3,x^2z, y^2z,x(z^2+xy),y(z^2+xy),z^3)}$$
	with Hilbert function $H(A)=(1,3,6,3,1)$.  Set $G=1$ so that $T=\F$ and the Thom class for the projection $\pi\colon A\rightarrow \F$ is just the socle generator $\tau=z^2xy$.  Then we can set $f_R(\xi)=\xi^4-z^2xy$ which has $G$-dual $h_R(\xi)=1$.  Therefore we have $\tilde{G}=\Xi^3$, $\tilde{F}=(Z^2XY-X^2Y^2)-\Xi^4$, and $H=\Xi^4$.  Then 
	$$\tilde{T}=\frac{\F[x,y,z,\xi]}{\Ann(\Xi^3)}=\frac{\F[x,y,z]}{(x,y,z,\xi^4)}$$
	and 
	$$\tilde{A}=\frac{\F[x,y,z,\xi]}{\Ann(Z^2XY-X^2Y^2-\Xi^4)}=\frac{\F[x,y,z,\xi]}{(x^3,y^3,x^2z, y^2z,x(z^2+xy),y(z^2+xy),z^3, \xi x, \xi y, \xi z, \xi^4-z^2xy)}$$
	and 
	$$\tilde{B}=\frac{\F[x,y,z,\xi]}{\Ann(-\Xi^4)}=\frac{\F[x,y,z,\xi]}{(x,y,z,\xi^5)}.$$ 
	Since $\xi\circ \left(Z^2XY-X^2Y^2-\Xi^4\right)=-\Xi^3=\xi\circ \left(-\Xi^4\right)$, and also $\Ann(Z^2XY-X^2Y^2-\Xi^4)+\Ann(\Xi^4)=\Ann(\Xi^3)$, it follows that $A$ is a connected sum of $\tilde{A}$ and $\tilde{B}$ along $\tilde{T}$, as guaranteed by Theorem \ref{thm:bdcs}. 	
\end{example}

We conclude this section with an example that shows that unlike the blow up operation, the blow down operation may not always preserve the standard grading. 
\begin{example}
	\label{ex:nonst}
	Define the AG algebra
	$$A=\frac{\F[x,y,u]}{(x^2,u^2,xy, xu-yu,xu-y^3)}=\frac{\F[x,y,u]}{\Ann\left(XU+YU+Y^3\right)}$$
	with the non-standard grading $\deg(x)=\deg(y)=1$ and $\deg(u)=2$.  We can blow up along the projection 
	$$\pi\colon A\rightarrow T=\frac{\F[x]}{\Ann(X)}=\frac{F[x]}{(x^2)}$$
	with Thom class is $\tau=u-y^2$ and kernel $K=(y,u)$.  Then setting $f_R(\xi)=\xi^2-(u-y^2)$ (which has $G=X$-dual polynomial $h_R(\xi)=\xi^k$), the cohomological blow up of $A$ along $\pi$ is  
	$$\tilde{A}=\frac{\F[x,y,u,\xi]}{(x^2,u^2,xy,xu-yu,xu-y^3,\xi y,\xi u,\xi^2-(u-y^2))}=\frac{\F[x,y,u,\xi]}{\Ann(\Xi^2X-XU-YU-Y^3)}.$$
	Since $u\equiv \xi^2+y^2$ in $\tilde{A}$, we can eliminate $u$, and get 
	$$\tilde{A}=\frac{\F[x,y,\xi]}{(x^2,(\xi^2+y^2)^2,xy,(x-y)(\xi^2+y^2),x(\xi^2+y^2)-y^3,\xi y,\xi(\xi^2+y^2))}\cong \frac{\F[x,y,\xi]}{(x^2,\xi^3,xy,x\xi^2-y^3,y\xi)},$$
	which has a standard grading.
	\par Incidentally, $\tilde{A}$ is also the cohomological blow-up algebra of a \emph{standard graded} AG algebra $A'$ along a different $T'$, namely:
	$$A'=\frac{\F[x,\xi]}{(x^2,\xi^3)}\rightarrow T'=\F.$$
	Here the kernel is $K'=(x,\xi)$ and the Thom class is the socle generator $\tau=x\xi^2$, and taking $y$ as the ``blow up variable'' with $f_{R}(y)=y^3-x\xi^2$, we find that the blow up 
	$$\tilde{A'}=\frac{A[y]}{(yK',\overline{f(y)})}=\frac{\F[x,\xi,y]}{(x^2,\xi^3,yx,y\xi,y^3-x\xi^2)}\cong \tilde{A}.$$
\end{example}

\section{Minimal Generating Sets and Complete Intersections}
\label{sec:Ideal}
In this paper, a complete intersection (CI) is a quotient of a polynomial ring by an ideal generated by a regular sequence of maximal length.  At a cursory glance, the presentation in Construction \ref{con:hat} may lead one to believe that cohomological blow ups cannot be complete intersections, except in embedding dimension two where all AG algebras are CI, but the following example shows otherwise.
\begin{example}
	\label{ex:cibu}
Define the AG algebras and the surjective map between them 
$$A=\frac{\F[x,y]}{(x^3,y^3)}\overset{\pi}{\longrightarrow} T=\frac{\F[x,y]}{(x^3,y)}, \ \pi(x)=x, \ \pi(y)=0.$$
Then the kernel of $\pi$ is $K=(y)$ and the Thom class is $\tau=y^2$.  Here $d=4$, $k=2$ and $n=2$.  Letting $\xi$ be the blow up variable and $f_A(\xi)=\xi^2-y^2$, define the associated cohomological blow-up algebra as in Construction \ref{con:hat} as
$$\tilde{A}=\frac{A[\xi]}{(\xi\cdot K,f_A(\xi))}=\frac{\F[x,y,\xi]}{(x^3,y^3,\xi y,\xi^2-y^2)}$$
In this presentation the generator $y^3$ is redundant, and we see that $\tilde{A}$ is indeed a complete intersection of Hilbert function $H(\tilde{A})=(1,3,4,3,1)=H(A)+H(T)[1]$.
\end{example}
In this section we will show that Example \ref{ex:cibu} is prototypical of the class of BUG which are CI.  First we introduce yet another description of the cohomological blow-up in terms of its defining ideal.

As usual, let $R=\F[x_1,\ldots,x_r]$ be a graded polynomial ring with homogeneous maximal ideal $\mathfrak{m}=(x_1,\ldots,x_r)$.  Recall from Section \ref{sec:Thom} a homogeneous ideal $I\subseteq R$ is $\mathfrak{m}$-primary and irreducible if and only if the quotient $R/I$ is a graded AG algebra.  We abuse notation slightly and call the socle degree of such an ideal the socle degree of the corresponding quotient.
\begin{construction}
	\label{con:ideal}
	Fix an $\mathfrak{m}$-primary irreducible homogeneous ideal $I\subset R$ of socle degree $d$, and fix a homogeneous polynomial $\tau\in R$ of degree $n$ where $2\leq n<d-1$ and such that $I\subsetneq (I\colon\tau)\subsetneq R$.  Then the ideal $(I\colon\tau)$ is also homogeneous, $\mathfrak{m}$-primary, and irreducible of socle degree $k=d-n$; see Lemma \ref{lem:transition}.
	
	Let $\xi$ be an indeterminate, and fix a homogeneous monic polynomial $f_R(\xi)=\xi^n+r_1\xi^{n-1}+\cdots +r_n$ where $r_i\in R_i$.  
	
	Define the ideal $\hat{I}\subset R[\xi]=\hat{R}$ by 
	\begin{equation}
	\label{eq:hatII}
	\hat{I}=I\cdot R[\xi]+\xi\cdot(I\colon \tau)\cdot R[\xi]+f_R(\xi)\cdot R[\xi].
	\end{equation}
\end{construction}
We now describe some properties of the ideal \eqref{eq:hatII} of Construction \ref{con:ideal}.
	Since $I$ and $(I\colon\tau)$ are both $\mathfrak{m}$-primary and irreducible it follows that $A=R/I$ and $T=R/(I\colon\tau)$ are AG algebras, and since $I\subseteq (I\colon\tau)$, the identity map on $R$ passes to a surjective map of quotient algebras $\pi\colon A\rightarrow T$.  Moreover it follows from Lemma \ref{lem:Thomdual} that one can choose orientations on $A$ and $T$ such that $\overline{\tau}\in A_n$ is the Thom class of $\pi$.  Note that the kernel of $\pi$ is $K=(I\colon\tau)/I\subset R/I=A$.  It follows that if $f_A(\xi)=\overline{f_R(\xi)}\in A[\xi]=R[\xi]/I$, then $\hat{A}$ from Construction \ref{con:hat} satisfies
	\begin{equation}
	\label{eq:hatideal}
	\hat{A}=\frac{A[\xi]}{(\xi\cdot K,\overline{f_R(\xi)})}=\frac{R[\xi]}{I\cdot R[\xi]+\xi\cdot (I\colon\tau)\cdot R[\xi]+(f_R(\xi))\cdot R[\xi]}=\frac{\hat{R}}{\hat{I}}.
	\end{equation} 
	It follows therefore from Lemma \ref{thm:hatAGor} that the ideal $\hat{I}$ from Construction \ref{con:ideal} is $\hat{\mathfrak{m}}=(x_1,\ldots,x_r,\xi)$-primary irreducible if and only if the constant coefficient $r_n\in R_n$ of $f_R(\xi)$ satisfies $r_n-\lambda\cdot \tau\in I$ for some $\lambda\in\F^\times$.  In this case we shall replace the ``hat'' with ``tilde'' and call $\tilde{I}=\hat{I}$ the \emph{cohomological blow up ideal of $I$ and $\tau$}, as it is the defining ideal of the cohomological blow-up algebra $\tilde{A}$ of $A$ along $\pi$ with parameters $(\overline{r_1},\ldots,\overline{r_{n-1}},\lambda)$. 	
	
	It is also clear from Construction \ref{con:ideal}, and Lemma \ref{lem:Thomhat}, that the colon ideal $(\hat{I}\colon \xi)\subset\hat{R}$ satisfies
	\begin{equation}
	\label{eq:colonhat}
	(\hat{I}\colon \xi)=(I\colon\tau)+(f_T(\xi)=\overline{f_R(\xi)})
	\end{equation}
	and we have 
	\begin{equation}
	\label{eq:hatTIdeal}
	\tilde{T}=\frac{T[\xi]}{\left(f_T(\xi)\right)}=\frac{R[\xi]}{(I\colon\tau)\cdot R[\xi]+f_R(\xi)\cdot R[\xi]}=\frac{\hat{R}}{\left(\hat{I}\colon\xi\right)},
	\end{equation} 
	which is the the algebra from Construction \ref{con:hat} Equation \eqref{eq:hatT}.
	
	\subsection{Minimal Generating Sets}
	
		Next we would like to know how the minimal generators of the cohomological blow up ideal $\tilde{I}$ compare to those of $I$ and of $(I\colon \tau)$. We also determine the relations among these minimal generators.
We start by providing a lemma  that helps explain some of these relations.

		\begin{lemma}
		\label{lem:mzero}
		Let $B(\xi)=b_p\xi^p+\cdots+b_0$ and $C(\xi)=c_q\xi^q++\cdots+c_0$ be any homogeneous polynomials in $R[\xi]$ with homogeneous coefficients $b_i,c_i\in R$.  Let $J\subset R$ is any homogeneous ideal in $R$ and assume that $C(\xi)$ is monic, i.e. $c_q=1$.  If the product of $B$ and $C$ is in the ideal in $R[\xi]$ generated by $J$, i.e. 
		$$B(\xi)\cdot C(\xi)\in J\cdot R[\xi],$$
		then every coefficient of $B(\xi)$ must lie in $J$, i.e. $b_i\in J\cdot R$ for all $0\leq i\leq p$.
	\end{lemma}
		\begin{proof}
		The key observation here is that if $D(\xi)=d_r\xi^r+\cdots+d_0$ is any polynomial in the ideal $J\cdot R[\xi]$, then its coefficients $d_i\in J\cdot R$.  Indeed, if $D(\xi)\in J\cdot R[\xi]$, plug in $\xi=0$ to see that $D(0)=d_0\in J\cdot R$.  Then $D(\xi)-d_0=D_1(\xi)=\xi\left(d_r\xi^{r-1}+\cdots+d_{1}\right)\in J\cdot R[\xi]$ and hence $d_r\xi^{r-1}+\cdots+d_{1}\in \left(J\cdot R[\xi]:\xi\right)=J\cdot R[\xi]$, since $\xi$ is a non-zero divisor of $J\cdot R[\xi]$.  Then plug in $\xi=0$ to see that $d_1\in J\cdot R$, and so on. 
		
		Therefore if the product  
		$$B(\xi)\cdot C(\xi)=\sum_{i=0}^{p+q}\left(b_0c_i+b_1c_{i-1}+\cdots+b_{i}c_0\right)\xi^{i}$$
		is in the ideal $J\cdot R[\xi]$, then each of its coefficients must be in $J\cdot R$, i.e.
		$$b_0c_i+\cdots+ b_ic_0\in J\cdot R \ \ \text{for each} \ 0\leq i\leq p+q.$$
		Taking first $i=p+q$, we find that $b_p\cdot c_q\in J\cdot R$ and since $c_q=1$ it follows that $b_p\in J\cdot R$.  Inductively assume that $b_{p-j+1},\ldots,b_p\in J\cdot R$.  Then taking $i=p+q-j$ we find that 
		$$b_pc_{q-j}+b_{p-1}c_{q-j+1}+\cdots+b_{p-j+1}c_{q-1}+b_{p-j}c_q\in J\cdot R$$
		from which it follows that $b_{p-j}c_{q}=b_{p-j}\in J\cdot R$ as well.  Therefore by induction, all coefficients $b_j$ are in $J\cdot R$.
	\end{proof}
	
	We can now provide a short exact sequence that determines the relations among the obvious (not necessarily minimal) set of generators of the cohomological blow up ideal $\tilde{I}$.
	
	\begin{proposition}
	\label{prop:ses}
	Let $\tilde{I}=I \cdot R[\xi]+\xi\cdot (I:\tau) \cdot R[\xi]+(f_R(\xi)) $ be the ideal described in Construction \ref{con:ideal}, let $I'=I\cdot R[\xi], K'=(I:\tau)\cdot R[\xi]$ and let $g=(f_R(\xi)-\lambda\tau)/\xi$. There is a short exact sequence of graded $R[\xi]$-modules
	\begin{equation}
	\label{eq:ses}
	0\to I'(-1) \oplus K'(-n) \xrightarrow{\begin{bmatrix} \xi & \lambda\tau \\ -1 & g \\ 0 & -1\end{bmatrix}} I'\oplus K'(-1)\oplus R[\xi]  \xrightarrow{\begin{bmatrix} 1 & \xi & f_R \end{bmatrix}} \hat{I}\to 0.
	\end{equation} 
	\end{proposition}

\begin{proof}
The definition of $\tilde{I}$ yields surjectivity of the rightmost nonzero map, and the injectivity of  the leftmost nonzero map is clear from its description (note there is a unit entry in each column). That \eqref{eq:ses}  is a complex is seen by  matrix multiplication and utilization of $f=g\xi+\lambda\tau$.

It remains to prove exactness in the middle of \eqref{eq:ses}. 
For this, consider a triple $(\alpha(\xi), \beta(\xi), \gamma(\xi)) \in I'\oplus K'(-1)\oplus R[\xi] $ so that 
\begin{equation}
\label{eq:syz}
\alpha(\xi)+\beta(\xi)\xi+\gamma(\xi)f_R(\xi)=0.
\end{equation}
Since $I'\subseteq K'$ we have $\alpha(\xi)+\beta(\xi)\xi\in K'$, so that also $\gamma(\xi)f_R(\xi)\in K'$ and by Lemma \ref{lem:mzero} it follows that $\gamma(\xi)\in K'$ since $f_R(\xi)$ is monic. Adding the relation
\[
\lambda\tau\gamma(\xi)+g(\xi)\gamma(\xi)-f_R(\xi)\gamma(\xi)=0,
\]
to \eqref{eq:syz} yields $(\alpha(\xi)-\lambda\tau\gamma(\xi))+(\beta(\xi)-g(\xi)\gamma(\xi))\xi=0$. Since $\gamma(\xi)\in K'=(I:\tau)\cdot R[\xi]$, we have $\lambda\tau\gamma(\xi)\in I\cdot R[\xi]=I'$ and thus $\beta'(\xi):=(\beta(\xi)-g(\xi)\gamma(\xi) \in I'$. We have thus obtained the identity
\[
\begin{bmatrix}
\alpha(\xi)\\ \beta(\xi)\\ \gamma(\xi)
\end{bmatrix}
=
-\beta'(\xi)
\begin{bmatrix}
\xi \\ -1\\ 0
\end{bmatrix}
-\gamma(\xi)
\begin{bmatrix}
\lambda\tau \\ g\\ -1
\end{bmatrix},
\]
where $\beta'(\xi)\in I'$ and $\gamma(\xi)\in K'$, establishing the desired exactness.
\end{proof}

Based on the presentation in Proposition \ref{prop:ses}, one can infer a minimal generating set for $\tilde{I}$. Setting $\hat{\fm}$ to be the homogeneous maximal ideal of $R[\xi]$, tensoring the short exact sequence \eqref{eq:ses} with $R[\xi]/\hat{\fm}$, and observing that there are isomorphisms $I'/\hat{\fm}I'\cong I/\fm I$ and $K'/\hat{\fm}K'\cong (I:\tau)/\fm (I:\tau)$ we obtain a new exact sequence of $\F$-vector spaces
\begin{equation}
\label{eq:ses2}
\frac{I}{\fm I}(-1) \oplus \frac{(I:\tau)}{\fm(I:\tau)} (-n) \xrightarrow[\psi_1]{\begin{bmatrix} 0 & \lambda\phi_2 \\ -\phi_1 & 0 \\ 0 & 0\end{bmatrix}} \frac{I}{\fm I}\oplus \frac{(I:\tau)}{\fm(I:\tau)}(-1)\oplus \F  \xrightarrow[\psi_2]{\begin{bmatrix} 1 & \xi & f_R \end{bmatrix}} \frac{\tilde{I}}{\hat{\fm}\tilde{I}}\to 0.
\end{equation}
The zero entries in the first matrix are due to the containments $\tau I'\subseteq \hat{\fm}I'$, $g I'\subseteq \hat{\fm}K'$, and $\xi I'\subseteq \hat{\fm}I'$.
We discuss the remaining maps $\phi_1$ and $\phi_2$ -- the former is induced by inclusion and the other by multiplication by $\tau$.  These maps fit into the sequence
	\begin{equation}
	\label{eq:cpx1}
	\xymatrix{(I\colon\tau)/\mathfrak{m}(I\colon\tau)\ar[r]^-{\times\tau}_-{\phi_2} & I/\mathfrak{m}I\ar[r]_-{\phi_1} & (I\colon\tau)/\mathfrak{m}(I\colon\tau)\ar[r]_-{\phi_2}^-{\times \tau} & I/\mathfrak{m}I.}
	\end{equation}
	Note this sequence of maps \eqref{eq:cpx1} forms a complex, i.e. $\phi_1\circ\phi_2=0$ and $\phi_2\circ\phi_1=0$, hence there are homology groups 
	\begin{equation}
	\label{eq:homology}
	H=\frac{\ker(\phi_1)}{\operatorname{im}(\phi_2)} \ \ \text{and} \ \ H'=\frac{\ker(\phi_2)}{\operatorname{im}(\phi_1)}.
	\end{equation}
	
	We will see that these homology groups $H$ and $H'$ measure the difference between minimal generating sets of $\tilde{I}$ and those of $I$ or $(I\colon\tau)$.  In fact, we will see that these homology groups are obstructions to $\tilde{I}$ being generated by a regular sequence.

	Using \eqref{eq:ses2}, we can express a minimal generating set of $\tilde{I}$ in terms of the minimal generating sets of $I$ and $(I:\tau)$ and the homology groups $H$ and $H'$.
	
	\begin{theorem}
		\label{thm:mingens}
Let $\phi_1, \phi_2$ be defined as in \eqref{eq:homology} and consider the vectors space decompositions
\begin{align}
\label{eq:It}
	\frac{I}{\mathfrak{m}I}= & U\oplus\underbrace{\operatorname{im}(\phi_2)\oplus H}_{\ker(\phi_1)} & 
	\frac{(I\colon \tau)}{\mathfrak{m}(I\colon\tau)}= & W\oplus \underbrace{\operatorname{im}(\phi_1)\oplus H'}_{\ker(\phi_2)}
	\end{align} 
	where $U$ and $W$ are some (non canonical) complements for $\ker(\phi_1)$ and $\ker(\phi_2)$ respectively.
		Then the vector space spanned by the minimal generators for $\hat{I}$ decomposes as 
		\begin{equation}
		\label{eq:hatmg}
		\frac{\tilde{I}}{\hat{\mathfrak{m}}\tilde{I}}\cong U\oplus H \oplus \xi W\oplus \xi H'\oplus \langle f_R(\xi)\rangle.\end{equation} 
	\end{theorem}	
	\begin{proof}
	This follows almost immediately from exactness of Sequence \eqref{eq:ses2}:  Since $\ker(\psi_2)=\operatorname{im}(\psi_1)=\operatorname{im}(\phi_2)\oplus\operatorname{im}(\phi_1)\subset I/\mathfrak{m}I\oplus (I:\tau)/\mathfrak{m}(I:\tau)\oplus\F$ is a direct summand, and since $\psi_2$ is surjective, it passes to an isomorphism 
	$$\frac{\hat{I}}{\hat{\mathfrak{m}}\hat{I}}\cong \frac{\frac{I}{\fm I}\oplus \frac{(I:\tau)}{\fm(I:\tau)}(-1)\oplus \F}{\ker(\psi_2)} \cong U\oplus \oplus H\oplus W\oplus H'\oplus \F$$
	and the result follows.
\end{proof}
	
	For a homogeneous ideal $J\subset R$ we denote by $\mu(J)=\dim_\F\left(J/\mathfrak{m}J\right)$ the number of minimal generators of $J$. The following corollary is an immediate consequence of Theorem \ref{thm:mingens}, and we omit the proof.
	\begin{corollary}
		\label{cor:mingenform}
		With $I$, $\tau$, $(I\colon\tau)$ and $\hat{I}$ as above we have 
		$$\mu(\tilde{I})=\mu(I)+\dim(H')+1=\mu(I\colon\tau)+\dim(H)+1.$$
	\end{corollary}
	Note that in our set up of Construction \ref{con:ideal}, $\mu(\tilde{I})\geq r+1=\dim(\hat{R})$ with equality if and only if $\tilde{I}$ is generated by a regular sequence.  In particular, Corollary \ref{cor:mingenform} shows that the vanishing of homology groups $H$ and $H'$ from \eqref{eq:homology} is necessary for the cohomological blow up ideal to be generated by a regular sequence, and hence for the BUG quotient $\tilde{A}=\hat{R}/\tilde{I}$ to be a CI.  However, the following example shows that the vanishing of $H$ and $H'$ is not quite sufficient for that purpose.  
	
\begin{example}
		\label{ex:zeroH}
		Let $I=\left(x^4, y^4, zx^2, zy^2, z^4-x^2y^2\right)$, and $\tau=z^2-xy$ so that $(I\colon\tau)=\left(x^4,y^4,zx^2,zy^2,z^2+xy\right)$.  Then taking $\hat{R}=R[\xi]=\F[x,y,z][\xi]$ and $f_R(\xi)=\xi^2-\tau$ the cohomological blow up ideal is
		$$\tilde{I}=\left(x^4,y^4,zx^2,z y^2,\xi(z^2+xy),\xi^4-(z^2-xy)\right),$$
		Hence the cohomological blow up $\tilde{A}=\hat{R}/\tilde{I}$ is not a CI.  Note that in this case the homology groups $H$ and $H'$ are both zero.
	\end{example}

	Note in Example \ref{ex:zeroH} that the colon ideal $\left(\tilde{I}\colon \xi\right)=\tilde{I}+(z^2+xy)$ is principal over $\tilde{I}$.  This is related to exact pairs of zero divisors, which allow a complete characterization of BUGs that are CI.  
	
\subsection{Complete Intersection Blow-up Algebras and Exact Zero Divisors}
Below we give a necessary and sufficient condition for a cohomological blow up algebra to be a complete intersection based on exact zero divisors, a notion introduced by I.~Henriques and L.~\c Sega \cite{HS}, which is defined as follows:

\begin{definition}
\label{def:exactzd}
A pair of non-unit elements $a,b$ of a ring $A$ is an {\em exact pair of zero divisors} if $(0 : _A a)= b\cdot A$ and $(0 : _A b) = a\cdot A$.
\end{definition}

\begin{example}
\label{ex:exact}
\begin{enumerate}
	\item In Example \ref{ex:cibu} $a=\overline{y^2}$ and $b=\overline{y}$ form an exact pair of zero divisors on 
	$$A=\frac{\F[x,y]}{(x^3,y^3)}.$$

	\item In Example \ref{ex:zeroH}, $a=\overline{z^2-xy}$ and $b=\overline{z^2+xy}$ form an exact pair of zero divisors on 
	$$A=\frac{\F[x,y,z]}{\left(x^4, y^4, zx^2, zy^2, z^4-x^2y^2\right)}.$$  
\end{enumerate}
\end{example}

If $A$ is Artinian then it suffices to check only one of the conditions in Definition \ref{def:exactzd}. 
\begin{lemma}
	\label{lem:one}
If $A$ is a graded Artinian $\F$-algebra then homogeneous elements $a,b\in A$ of positive degree form an exact pair of zero divisors of $A$ provided that $(0:_A a)=b\cdot A$.
\end{lemma}
\begin{proof}
From  $(0:_A a)=b\cdot A$ we deduce there is an $A$-module isomorphism $A/(b)=A/(0:_A a)\cong a\cdot A$ by means of the diagram
$$\xymatrix{0\ar[r] & A/(0:_Aa)\ar@{-->}[d]_-{\cong}\ar[r]^-{\times a} &  A \ar[d]^-{=} \ar[r] &  A/(a) \ar[d]^-{=} \ar[r] &  0\\
0\ar[r] & a\cdot A \ar[r] & A \ar[r] & A/(a) \ar[r] & 0\\}
$$
Similarly, there is an isomorphism $A/(0:_A b)\cong b\cdot A$.
Moreover, the hypothesis yields $ab=0$ and thus $a\cdot A\subseteq (0:_A b)$.
To see that this containment is in fact an equality we compute
\[
\dim_\F(A) = \dim_\F(A/(b)) + \dim_\F(b\cdot A) = \dim_\F(a\cdot A) + \dim_\F (A)- \dim_\F (0:_A b)
\]
whence $\dim_\F(a)=\dim_\F (0:_A b)$, and hence $a\cdot A=(0:_Ab)$ as desired.
\end{proof}

The next lemma relates exact zero divisors to the homology group $H$ in \eqref{eq:homology}.
\begin{lemma}
	\label{lem:exactH}
	Let $A=R/I$ and $T=R/(I\colon\tau)$ be AG algebras, and suppose that $\overline{\tau}\in A_n$ is part of a pair of exact zero divisors on $A$.  Then the homology groups $H$ and $H'$  from \eqref{eq:homology} vanish. 
\end{lemma} 
\begin{proof}
	Assume that there exists a homogeneous element (of positive degree) $\sigma\in R$ for which $(I\colon\tau)=I+(\sigma)$ and $(I\colon\sigma)=I+(\tau)$, so that $\overline{\tau}$ and $\overline{\sigma}$ are an exact pair of zero divisors in $A$.  
	
	Suppose that $s\in I$ is any minimal generator of $I$ for which $s+\mathfrak{m}I\in \ker(\phi_1)$.  Then $s\in\mathfrak{m}(I\colon\tau)=\mathfrak{m}(I+(\sigma))$, and hence there exists elements $g_1,\ldots,g_t\in I$ and $a_1,\ldots,a_t,b\in \mathfrak{m}$ such that 
	$$s=\sum_{i=1}^ta_ig_i+b\cdot\sigma.$$
	It follows that $b\in(I\colon\sigma)=I+(\tau)$ and hence $b=q+r\cdot \tau$ for some $q\in I$ and $r\in R$.  Therefore we see that 
	$$s-r\cdot\tau\cdot\sigma\in\mathfrak{m}I.$$
	Then the equivalence class of $s$ in $H$ is 
	$$\left[s\right]\equiv \left[r\tau\sigma\right]\equiv 0$$
	which shows that $H=0$, as claimed.  
	
	Moreover, by \cite[Proposition 1.9]{KSV} we have $\tau\cdot\sigma\in I\setminus \mathfrak{m}I$, which yields that $(I:\tau)/\fm(I:\tau)=\langle \sigma \rangle$. Consequently, the second equation in \eqref{eq:It} yields $W=\langle \sigma\rangle$, $\operatorname{im}(\phi_1)=0$, and $H'=0$, also as claimed. 
\end{proof}

The following is a characterization of complete intersection cohomological blow-up algebras in terms of exact pairs of zero divisors.

\begin{theorem}
\label{thm:CI}
Fix oriented AG algebras $A=R/I$ and $T=R/(I:\tau)$ of socle degrees $d>k$ and let $\pi\colon A\rightarrow T$ be the natural surjective algebra map between them with Thom class $\overline{\tau}\in A_n$ and kernel $K\subset A$.  Let $f_R(\xi)=\xi^n+r_1\xi^{n-1}+\cdots+\lambda\cdot \tau\in R[\xi]=\hat{R}$ for some homogeneous elements $r_i\in R$, and let $\tilde{A}=\hat{R}/\tilde{I}$ be the associated cohomological blowup of $A$ along $\pi$ with parameters $(\overline{r}_1,\ldots,\overline{r}_{n-1},\lambda)$.  Then the following are equivalent.
\begin{enumerate}
\item $\tilde{A}$ is a complete intersection. 
\item $A$ is a complete intersection and $\overline{\tau}\in A_n$ is part of an exact pair of zero divisor on $A$.
\item $T$ is a complete intersection and $\overline{\tau}\in A_n$ is part of an exact pair of zero divisor on $A$.
\end{enumerate}
\end{theorem}
\begin{proof}
	Assume that (1) holds.  Then $\tilde{I}$ must be generated by a $\hat{R}$-sequence, and hence it follows from Corollary \ref{cor:mingenform} that both $I$ and $(I\colon\tau)$ must also be generated by $R$-sequences and hence that $A$ and $T$ must be complete intersections, and that the homology groups $H$ and $H'$ must vanish.  Thus according to Equations \eqref{eq:It} and \eqref{eq:hatmg} we have 
	\begin{align*}
	\frac{I}{\mathfrak{m}I}= U\oplus \phi_2( W), &\ \ \ \ \ \  \frac{(I\colon\tau)}{\mathfrak{m}(I\colon\tau)}= \phi_1(U)\oplus W, & \text{and} \ \  
	\frac{\tilde{I}}{\hat{\mathfrak{m}}\tilde{I}}= & U\oplus \xi\cdot W\oplus\langle f_R(\xi)\rangle.
	\end{align*}
	It follows that $\dim_\F(W)=1$, since no two elements of a minimal generating set of $I$ (or $\tilde{I}$) can have a common divisor.  Let $\sigma\in (I\colon\tau)$ be a minimal generator for which $\sigma+\mathfrak{m}(I\colon\tau)\in W$.  Then we can write 
	$$I=(u_1,\ldots,u_{r-1},\tau\cdot\sigma) \ \text{and} \ (I\colon\tau)=(u_1,\ldots,u_{r-1},\sigma)=I+(\sigma)$$
	which implies by Lemma \ref{lem:one} that $a=\overline{\tau}$ and $b=\overline{\sigma}$ form an exact pair of zero divisors for $A$.  This shows that (1.) implies (2.) and (3.).
	
	Next assume that (2.) holds: $A$ is a complete intersection and $\overline{\tau}\in A_n$ is an exact zero divisor for $A$.  It follows from Lemma \ref{lem:exactH} that $H=H'=0$, and hence by Corollary \ref{cor:mingenform} it follows that $\tilde{I}$ is generated by a $\hat{R}$-sequence and hence that the cohomological blow up $\tilde{A}$ is a complete intersection, which is (1.).
	
	Finally, assume that (3.) holds:  $T$ is a complete intersection and $\overline{\tau}\in A_n$ is an exact zero divisor for $A$.  Again, it follows from Lemma \ref{lem:exactH} that $H=0$, and hence it follows from Corollary \ref{cor:mingenform} that the cohomological blow up $\tilde{A}$ is a complete intersection, and thus (1.) holds.
\end{proof}

\begin{remark}
It follows from the description of the minimal generators of $\tilde{I}$ in equation \eqref{eq:hatmg}  together with Theorem \ref{thm:CI} and Lemma \ref{lem:exactH} that when $\tilde{A}$ is a complete intersection and thus $\overline{\sigma}, \overline{\tau}$ is an exact pair of zero divisors on $A$,  a minimal generating set for $\tilde{I}$ can be described as
\[
\frac{\tilde{I}}{\hat{\mathfrak{m}}\tilde{I}}\cong U \oplus \langle \xi \sigma \rangle \oplus \langle f_R(\xi)\rangle.
\]
In particular, since $\tilde{I}$ is generated by a regular sequence and $ \xi \sigma$ is a minimal generator for $\tilde{I}$ it follows from  \cite[Proposition 1.9]{KSV} that  $\xi$ is an exact zero divisor of $\tilde{A}$.  The following example shows that this condition is not quite sufficient to identify a complete intersection as a cohomological blow-up algebra.
\end{remark}

\begin{example}
	\label{ex:Tony}
	Consider the complete intersection  
	$$\tilde{A}=\frac{\F[x,y]}{(x^4+y^4,x^2y^2)}.$$
	Then $\xi=\overline{x}\in\tilde{A}_1$ is an exact zero divisor for $\tilde{A}$, and its Hilbert function is $H(\tilde{A})=(1, 2, 3, 4, 3, 2, 1)$.  This particular CI algebra (or any isomorphic to it) cannot be a cohomological blow-up algebra by a result of \cite{IMMSW} concerning codimension two.
\end{example}
\subsection{Application:  Watanabe's Bold Conjecture}\label{boldsec}
The following ``rather bold'' conjecture was put forth by the fifth author \cite{WatanabeTalk} after noticing that many complete intersections arising as invariant rings could be realized as subrings of complete intersections cut out by quadrics and of the same socle degree.
\begin{conjecture}[Watanabe's Bold Conjecture]\label{boldconj}
	For any standard graded Artinian complete intersection $A$ of socle degree $d$, there is another standard graded Artinian complete intersection $B$ of the same socle degree $d$ cut out by quadrics and an injective algebra map
	from $A$ into $B$, in symbols:
	$$\phi\colon A=\frac{\F[x_1,\ldots,x_r]}{(f_1,\ldots,f_r)}\hookrightarrow \frac{\F[X_1,\ldots,X_N]}{(F_1,\ldots,F_N)}=B, \ \ \deg(F_i)=2, \ \forall \ 1\leq i\leq N.$$
\end{conjecture}
In 2016, the third author proved \cite{McDBC} Watanabe's Bold Conjecture in the special case where $A$ is cut out by polynomials which factor into a product of linear forms\footnote{See also \cite[Theorem 1]{McDSmith} for another case in which Watanabe's Bold Conjecture holds.}.  Here we give another, much shorter, proof of this result (in fact something slightly stronger) using cohomological blow ups.  First some notation.  Let $A=R/I$ is a complete intersection where $R=\F[x_1,\ldots,x_r]$ has the standard grading and $I$ is minimally generated by some regular sequence $(f_1,\ldots,f_r)$.  If the minimal generators $f_1,\ldots,f_n$ can be chosen such that each $f_i$ is a product of linear and/or quadratic forms, i.e. $f_i=L_1\cdots L_k$ where $\deg(L_j)=1 \ \text{or} \ 2$, then we shall say $A$ is of class $\mathcal{W}$.  If $d_i=\deg(f_i)>1$ for all $i$, then the degree sequence is $(d_1,\ldots,d_n)$ and we define its \emph{defect} to be 
\[ \operatorname{def}(A)=d_1+\cdots+d_n-2n=\operatorname{soc.deg.}(A)-n.\]
Let $\mathcal{W}(m)$ denote the subclass consisting of complete intersections of defect $m$, so that $\mathcal{W}=\bigsqcup_{m=0}^\infty \mathcal{W}(m)$.  Complete intersections of class $\mathcal{W}(0)$ are called \emph{quadratic complete intersections}.  

\begin{theorem}
	\label{thm:WBC}
	For every complete intersection $A$ of class $\mathcal{W}$ there is a complete intersection $B$ of class $\mathcal{W}(0)$ of the same socle degree as $A$ and an injective map of algebras $\phi\colon A\hookrightarrow B$.
\end{theorem}
\begin{proof}	
	We use induction on $m$ to show that every complete intersection $A\in\mathcal{W}(m)$ embeds into a complete intersection $B\in\mathcal{W}(0)$ of the same socle degree.  For the base case $m=0$ there is nothing to show.  For the inductive step, assume that $m>0$ and that the statement holds for every $0\leq m'<m$.  Fix $A\in \mathcal{W}(m)$ with presentation as in \eqref{eq:A} with minimal generators $f_1,\ldots,f_n$ that factor into products of linear and/or quadratic forms.  Since $\operatorname{def}(A)=m>0$, we may assume without loss of generality that $\deg(f_n)\geq 3$.  Write $f_n=\tau\cdot g$ where $\deg(\tau)=2$ which is possible since we are assuming that $f_n$ factors as a product of linear and quadratic forms.  Then let $\xi$ be an indeterminate, and define the monic quadratic polynomial $f_R(\xi)=\xi^2-\tau\in R[\xi]$.  Then the cohomological blow up ideal of $I$ with respect to $\tau$ and $f_R(\xi)$ is by \eqref{eq:hatII}
	$$\tilde{I}=(f_1,\ldots,f_{n-1},\xi g,f_R(\xi))\subset R[\xi]$$
	and the quotient $\tilde{A}=R[\xi]/\tilde{I}$ is the cohomological blow-up of $A$ along the map 
	$$\pi\colon A=\frac{\F[x_1,\ldots,x_n]}{(f_1,\ldots,f_{n-1},\tau g)}\rightarrow \frac{\F[x_1,\ldots,x_n]}{(f_1,\ldots,f_{n-1},g)}=T.$$
	Note that the degree sequence of $\tilde{A}$ is $(d_1,\ldots,d_{n-1},d_n-1,2)$ and its defect is $d_1+\cdots+d_{n}-1+2-2(n+1)=(d_1+\cdots+d_n)-2n-1=m-1$.  Finally it is clear that the minimal generators of $\tilde{I}$ are either quadratic (like $f_R(\xi)$) or factor into a product of linear and/or quadratic forms (since the minimal generators of $I$ do).  Therefore $\tilde{A}\in\mathcal{W}(m-1)$.  For the injection, note that the blow down map will do the trick:
	$$\beta\colon A=\frac{\F[x_1,\ldots,x_n]}{(f_1,\ldots,f_{n-1},\tau g)}\hookrightarrow\frac{\F[x_1,\ldots,x_n][\xi]}{(f_1,\ldots,f_{n-1},\xi g, f(\xi))}=\tilde{A}.$$ 
	To complete the proof note that by induction, there is an embedding of $\tilde{A}\in\mathcal{W}(m-1)$ into some complete intersection $B\in\mathcal{W}(0)$ of the same socle degree as $\tilde{A}$ (and hence also $A$), say
	$$\iota\colon \tilde{A}\hookrightarrow B.$$
	Composing $\beta$ and $\iota$ then gives an embedding $\iota\circ\beta\colon A\hookrightarrow B$, as desired.    
\end{proof}

\section{Restrictions on Hilbert Functions}\label{compsec}
In this section we show that standard graded cohomological blow-up algebras cannot have arbitrary Hilbert functions. In fact, in the parameter space of AG algebras of fixed embedding dimension $\geq 3$ and socle degree $4$ or $\geq 6$, cohomological blow-up algebras are quite rare. To justify this assertion we recall the notion of compressed AG algebra \cite{Iarrobino}.  We use the notation of Section \ref{sec:MD}: $R=\F[x_1,\ldots,x_r]$ is a standard graded polynomial ring, $Q=\F[X_1,\ldots,X_r]$ its dual divided power algebra.  For a $d$-form $H$, its associated AG algebra $C=R/\Ann(H)$ has embedding dimension $r$ if $\Ann(H)$ contains no linear forms.

\begin{definition}
	A standard graded AG algebra $C$ of embedding dimension $r$ and socle degree $d$ is {\em compressed} provided
	\[
	\dim_k C_i=
	\begin{cases}
	\dim_\F R_{i}=\binom{r-1+i}{r-1} & \text{if } i\leq \left \lfloor \frac{d}{2}\right \rfloor \\
	\dim_\F C_{d-i} & \text{if } i > \left \lfloor \frac{d}{2}\right \rfloor
	\end{cases}.
	\]
\end{definition}

We can parametrize graded Artinian Gorenstein algebras of fixed socle degree $d$ and codimension $r$ by their Macaulay dual generators in $Q_d$, essentially by elements of the projective space $\P(Q_d)$. The compressed algebras form a dense Zariski open set in this parameter space \cite[Proposition~3.12]{IK}.\footnote{This depends on $R_j/(I^2)_j$ being the tangent space to $\
	Gor(T)$ at $A=R/I$ and gives perhaps the shortest proof. Original references are \cite[Theorem 3.31]{EI} and \cite[Theorem 5.1]{Green}.} We now show that for sufficiently large parameters the cohomological blow-up algebras do not belong to this set.

\begin{theorem}
	\label{thm:compressed}
	Compressed algebras of embedding dimension $r\geq 3$ whose socle degree $d$ satisfies $d=4$ or $d\geq 6$ are not cohomological blow up algebras. 
\end{theorem}

\begin{proof}
	Let $C=R/\Ann(H)$ be a compressed algebra of socle degree $d$ and embedding dimension $r$ and assume that $C$ is the blow up along a surjective morphism $\pi: A\to T$ of some AG algebras of socle degrees $d$ and $k$ respectively. 
	Then since $C$ is compressed $\tilde{I}=\Ann(H)$ is a homogeneous ideal with initial degree 
	\[
	\min\left \{i\mid\tilde{ I}_i\neq 0 \right \}=\left \lfloor \frac{d}{2}\right \rfloor +1 .
	\]
	By \cite[Proposition 3.2]{Boij} the minimal generators of $\tilde{I}$ have degrees $\left \lfloor \frac{d}{2}\right \rfloor +1$ and possibly $\left \lfloor \frac{d}{2}\right \rfloor +2$. Recall that the polynomial $f_R(\xi)$ of degree $n=d-k$ is a minimal generator of $\tilde{I}$ by Theorem \ref{thm:mingens} and thus 
	\begin{equation}
	\label{eq:n}
	n=d-k\in \left\{ \left \lfloor \frac{d}{2}\right \rfloor +1, \left \lfloor \frac{d}{2}\right \rfloor +2\right\},
	\quad \text{ thus } \quad 
	k\in \left\{ \left \lceil \frac{d}{2}\right \rceil -1, \left \lceil \frac{d}{2}\right \rceil -2\right\}.
	\end{equation}
	
	By the definition of blow up, when $n>1$, the embedding dimension of $A$ is $r-1$ (one less than the embedding dimension of $C$) and by the surjectivity of $\pi$ embedding dimension of $T$ is at most $r-1$. This yields the following upper bounds on the Hilbert functions of $A$ and $T$:
	\begin{eqnarray}
	\label{eq:ineqA}
	\dim_\F A_i &\leq& \min \left\{\binom{r-2+i}{r-2},  \dim_\F A_{d-i} \right\} \\
	\label{eq:ineqT}
	\dim_\F T_i &\leq& \min \left\{\binom{r-2+i}{r-2},  \dim_\F T_{k-i} \right\}.
	\end{eqnarray}
	By Theorem \ref{thm:BUMD1} we have
	\begin{equation}
	\label{eq:H(C)}
	H(C)=H(A)+H(T)[1]+ \cdots+H(T)[n-1].
	\end{equation}
	Evaluating the above identity in degrees $i \leq \left \lfloor \frac{k}{2}\right \rfloor +1$ combined with the above inequalities yields
	\[
	\binom{r+i-1}{r-1}=H(A)[i]+\sum_{j=1}^{i} H(T)[i-j]\leq \sum_{j=0}^{i} \binom{r+j-2}{r-2}=\binom{r+i-1}{r-1},
	\]
	where the last equality is a well-known combinatorial identity. This implies that equality must hold both in \eqref{eq:ineqA} and in \eqref{eq:ineqT} for $i \leq \left \lfloor \frac{k}{2}\right \rfloor +1$. In degree $i = \left \lfloor \frac{k}{2}\right \rfloor +2$ since $r\geq 3$ the  inequality \eqref{eq:ineqT} yields 
	\[
	H(T)[i-1]= H(T)[k-i+1]=H(T)\left[ \left \lceil \frac{k}{2}\right \rceil -1\right]\leq  \binom{r-3+ \left \lceil \frac{k}{2}\right \rceil}{r-2}<\binom{r+i-2}{r-2}
	\]
	Evaluating equation \eqref{eq:H(C)} in degree $i = \left \lfloor \frac{k}{2}\right \rfloor +2$ and combining the result with the  inequality \eqref{eq:ineqA} and the inequality displayed above gives
	\[
	H(C)[i]=H(A)[i]+\sum_{j=1}^{i} H(T)[i-j]<  \sum_{j=0}^{i} \binom{r+j-2}{r-2}=\binom{r+i-1}{r-1}.
	\]
	This contradicts our assumption that $C$ is compressed, provided that  
	\[ \left \lfloor \frac{k}{2}\right \rfloor +2\leq \left \lfloor \frac{d}{2}\right \rfloor, \text{ for } k\in \left\{ \left \lceil \frac{d}{2}\right \rceil -1, \left \lceil \frac{d}{2}\right \rceil -2\right\}.\] 
	The above inequality is satisfied if and only if $d=4$ or $d\geq 6$.
\end{proof}

The following example shows that there exists compressed algebras of socle degree $5$ which are cohomological blow-ups.
\begin{example}
		\label{ex:HF}
		Let 
		$$\pi\colon A=\frac{\F[x,y]}{(x^4,y^3)}\rightarrow T=\frac{\F[x,y]}{(x^2-xy,y^2)}$$
		be the map defined by $\pi(x)=x$ and $\pi(y)=y$ and distinguished socle generators $\sigma_A=x^3y$ and $\sigma_T=x^2$.  Then $\tau=xy$ and $K=(x^2-xy,y^2)$, and hence taking parameters $t_1=t_2=0$ and $\lambda=1$ we get the cohomological blow-up 
		$$\tilde{A}=\frac{\F[x,y,\xi]}{(x^4,y^3,\xi(x^2-xy),\xi y^2,\xi^2-xy)}$$
		which has Hilbert function $H(\tilde{A})=(1,3,6,6,3,1)$, and hence $\tilde{A}$ is a compressed BUG of socle degree $5$.
\end{example}

\section{The Lefschetz Property}
\label{Lefsec}
 The strong Lefschetz property  for graded AG algebras is an algebraic version of a property of cohomology rings of smooth complex projective varieties stemming from the Hard Lefschetz theorem in algebraic geometry.  

\begin{definition}\label{LPdef}
A graded Artinian $\F$-algebra $A=\bigoplus_{i=0}^dA_i$ is said to satisfy the {\em strong Lefschetz property} (SLP) if there is a linear form $\ell\in A_1$ for which the multiplication maps $\times\ell^j\colon A_i\rightarrow A_{i+j}$ have full rank $\rank(\times \ell^j)=\min\{\dim_\F A_i, \dim_\F A_{i+j}\}$ for each degree $i$ and each exponent $j$. A linear form $\ell$ with this property is called a strong Lefschetz element for $A$.

If the multiplication maps $\times\ell\colon A_i\rightarrow A_{i+1}$ have full rank for each degree $i$ then $A$ is said to satisfy the {\em weak Lefschetz property} (WLP). 
\end{definition}

More generally, given any graded Artinian algebra $A$, and any linear form $\ell\in A_1$ we can define its \emph{Jordan type} $P_{\ell}$ to be the partition corresponding to the block decomposition of the Jordan canonical form for the nilpotent linear operator $\times \ell\colon A\rightarrow A$.  It is well known that for a standard-graded AG algebra with unimodal
Hilbert function then $\ell$ is strong Lefschetz if and only if the Jordan type $P_\ell$ is equal to $H^\vee$, the conjugate of the Hilbert function regarded as a partition (switch rows and columns in the Ferrers graph) \cite[Proposition 2.10]{IMM0} and $\ell$ is weak Lefschetz if $P_\ell$ has number of parts equal to the Sperner number of $H$ \cite[Proposition 3.5]{H-W}.

In this section we study the strong Lefschetz Property for cohomological blow-up algebras.
To attain this goal, we observe the behavior of these rings in families. Our strategy is to consider every cohomological blow-up algebra as the general fiber in a certain flat family. Interestingly, all fibers in these familes will be AG algebras with the exception of the special fiber. Our proof of the SLP for the cohomological blow up then employs the following well known fact:  in a flat family if SLP holds on the special fiber, then it must also hold on a sufficiently general fiber by semicontinuity of Jordan type, e.g. \cite[Corollary 2.44]{IMM0}.

We continue with the notation established in previous sections, namely $\pi:A\to T$ is a surjective map of oriented AG algebras with Thom class $\tau$ and kernel $K$, $f_A(\xi)\in A[\xi]$ is a homogeneous monic polynomial yielding the cohomological blow-up algebra
\[
\tilde{A}=\frac{A[\xi]}{(\xi \cdot K, f_A(\xi))}=\frac{A[\xi]}{I},
\]
as in Construction \ref{con:hat}.

Consider a weighted monomial order $<$ on $A[\xi]$ obtained by assigning weight $1$ to the variable $\xi$ and weight $0$ to each element of $A$. Then the weight  of a monomial $\mu$ is $w(\mu)=\max\{n:\,\  \xi^n\mid \mu\}$ and the weight of a polynomial $g=\sum c_i\mu_i \in A[\xi]$ with $c_i\in F^\times$ and $\mu_i$ monomials is $w(g)=\max\{w(\mu_i)\}$. Two monomials $\mu, \mu'$ satisfy $\mu<\mu'$ if and only if $w(\mu)<w(\mu')$. The initial form of a polynomial $g=\sum c_i\mu_i \in A[\xi]$ is ${\rm in}_<(g)=\sum_{w(\mu_j)=w(g)} c_j\mu_j$ and the initial ideal of an ideal $J\subseteq A[\xi]$ is 
\[
{\rm in}_<(J)=({\rm in}_<(g) \mid g\in J).
\] 

While it is not usually the case that the generators of the initial ideal of an ideal $J$ are the initial forms of the generators of $J$, this is nevertheless the case for $I$ since, as we show below,  the generators of $I$ form a Gr\"obner basis with respect to $<$. We obtain the following description for the initial ideal of $I$.

\begin{lemma}
\label{lem:initialideal}
The set $\{\xi \cdot K, f_A(\xi)\}$ is a Gr\"obner basis for $I$ with respect to  $<$, i.e. 
the initial ideal of $I$ is
$ {\rm in}_<(I)=(\xi\cdot K, \xi^n)$.
\end{lemma}
\begin{proof}
To show that the set $\{\xi\cdot K, f_A(\xi)\}$ is a Gr\"obner basis for $I$ one utilizes Buchsberger's criterion; see \cite[Theorem 15.8]{Eisenbud}.
Since any element in $\xi\cdot K$ is equal to its initial form, the S-polynomial of any two such elements is $0$. It remains to compute the S-polynomial of $u\xi$ (with $u\in K$) and $f_A(\xi)$ which is
\[
S(f_A(\xi),u\xi)=u(\xi^n+a_1\xi^{n-1}+\cdots+a_{n-1}\xi+\lambda\tau)-u\xi\cdot \xi^{n-1}=ua_1\xi^{n-1}+\cdots+ua_{n-1}\xi
\]
since $K\cdot \lambda\tau=0$ by Lemma \ref{lem:Thomdual}. Since $S(f_A(\xi),u\xi)\in \left(\xi \cdot K \right)$, this polynomial reduces to $0$ modulo the the set $\{\xi\cdot K, f_A(\xi)\}$,  concluding the proof that this is a Gr\"obner basis. It follows that
$ {\rm in}_<(I)=({\rm in}_<(\xi\cdot K), {\rm in}_<(f_A(\xi)))=(\xi\cdot K, \xi^n).$
\end{proof}

The relationship between $\tilde{A}=A[\xi]/I$ and its initial algebra ${\rm in}_<(\tilde{A}):=A[\xi]/{\rm in}_<(I)$ is established by means of the following well-known construction; see \cite[p. 343]{Eisenbud}. Given a parameter $t$ one considers an ideal $\mathcal{I}$ of $A[\xi,t]$ given by 
\[
\mathcal{I}=\left (\xi\cdot K, \xi^n+a_1t\xi^{n-1}+\cdots+a_{n-1}\xi t^{n-1} +\lambda\tau t^n \right).
\]
Note that setting $t=1$ in $\tilde{I}$ recovers $I$ while setting $t=0$ gives $ {\rm in}_<(I)$. We recall some key properties of this construction cf.~\cite[Theorem 15.17]{Eisenbud}.

\begin{theorem}
\label{thm:flatfam}
The following ring $\mathcal{A}$ is free and hence flat as a $\F[z]$-algebra
\begin{equation}
	\label{eq:1par}
	\mathcal{A}=\frac{A[\xi,z]}{\mathcal{I}}=\frac{A[\xi,z]}{(\xi\cdot K, \xi^n+a_1\xi^{n-1}z+\cdots+a_{n-1}\xi z^{n-1} +a_nz^n)}.
	\end{equation}
	Thus $\mathcal{A}$ can be viewed as a flat family of algebras with fibers $\mathcal{A}_c=\mathcal{A}\otimes_{\F[z]}F[z]/(z-c)$ given by
$$\mathcal{A}_c\cong \begin{cases} A[\xi]/I=A[\xi]/(\xi\cdot K,f_A(\xi))=\tilde{A} & \text{if} \ c\in \F^\times\\ 
A[\xi]/{\rm in}_<(I)=A[\xi]/(\xi\cdot K,\xi^n)={\rm in}_<(\tilde{A}) & \text{if} \ c=0\\
\end{cases}$$
\end{theorem}
In the literature, flat families in which the general fibers are isomorphic are sometimes called \emph{jump deformations}.  We remark that Theorem \ref{thm:flatfam} remains true if we replace $\tilde{A}$ with any $\hat{A}$ from Construction \ref{con:hat}, but recall that $\hat{A}$ is Gorenstein if and only if $\hat{A}=\tilde{A}$ by Theorem \ref{thm:hatAGor}. 

Since $\tilde{A}$ is Gorenstein, Theorem \ref{thm:flatfam} implies that the general fibers of the family $\mathcal{A}$ are Gorenstein. However, Theorem \ref{thm:hatAGor} implies that the special fiber $\mathcal{A}_0$ is not Gorenstein, but is boundary-Gorenstein (Remark \ref{rem:fakeG}).  Moreover, Corollary \ref{cor:splitseq} implies that all the fibers have the same Hilbert function given by 
\[ H(\mathcal{A}_c)=H(A)+H(T)[1]+\cdots+H(T)[n-1] \qquad \text{ for all } c\in \F.\]
	
We now consider the strong Lefschetz property for the special fiber. The ring ${\rm in}_<(\tilde{A})$ can thus be recognized as a fibered product 
\[
{\rm in}_<(\tilde{A})=A\times_T B \qquad \text{ with} \qquad B=\frac{T[\xi]}{(\xi^n)}
\]
 with respect to the projections $\pi_A=\pi:A\to T$ and $\pi_B:B\to T=B/(\xi)$, where $\pi_B$ is the canonical projection.
Our proof of Lemma \ref{lem:specialfiber} below closely resembles the proof of \cite[Theorem 5.12]{IMS}, where the SLP is established for fibered products of certain AG algebras of the same socle degree. This is not the case here as $\rm{soc. deg.}(A)=\rm{soc. deg.}(B)+1$. 

\begin{lemma}
\label{lem:specialfiber}
Let $\pi:A\to T$ be a surjective homomorphism of graded AG $\F$-algebras of socle degrees $d>k$ respectively, such that both $A$ and $T$ have SLP. Assume that $\F$ has characteristic zero, or characteristic $p>d>k$. Then the initial algebra of $\tilde{A}$, $A[\xi]/{\rm in}_<(I)$ has SLP as well.
\end{lemma}
\begin{proof}
Recall from Lemma \ref{lem:initialideal} that 
\[
{\rm in}_<(\tilde{A})=\frac{A[\xi]}{{\rm in}_<(I)}=\frac{A[\xi]}{(\xi\cdot K, \xi^n)}.
\]
This algebra admits a decomposition
\[
{\rm in}_<(\tilde{A})=A\oplus J, \text{ where } J=T\xi\oplus \cdots \oplus T\xi^{n-1}=\xi\cdot \frac{T[\xi]}{(\xi^{n-2})}.
\]
The important point here is that in the algebra ${\rm in}_<(\tilde{A})$, the vector space $J$ is actually an ideal, and thus multiplication by a linear form $\ell=\ell_A+\xi$ with $\ell_A\in A_1$ is represented with respect to the above decomposition by a block matrix
\begin{equation}
\label{eq:blockmatrix}
\begin{pmatrix}
\ell_A & 0 \\
* & \ell|_J
\end{pmatrix}
\end{equation}
where $\ell_A$ stands for the map given by multiplication by $\ell_A$ on $A$ and $ \ell|_J$ denotes the restriction of multiplication by $\ell$ to $J$. Since $A$ and $T$ have SLP, one can pick $\ell_A$ so that both $\ell_A$ and its image $\ell_T=\overline{\ell_A}$  in $T$ are strong Lefschetz elements. Then by  \cite[Theorem 6.1]{HW} if characteristic $\F$ is zero, or \cite[Theorem 2.6]{IMM2} in general $\ell=\ell_T+\xi$ is also a strong Lefschetz element on the free extension $T[\xi]/(\xi^{n-2})$ of $T$,  and hence on $J$.  Since the Hilbert functions of $A$ and $J$ are both symmetric around $d/2$ (due to both being AG) and unimodal (due to both having SLP) we conclude that $\ell_A^j:A_i\to A_{i+j}$ and $\ell|_J^j:J_i\to J_{i+j}$ are simultaneously injective or simultaneously surjective. It follows from \eqref{eq:blockmatrix} that $\ell$ is a strong Lefschetz element on ${\rm in}_<(\tilde{A})$.
	\end{proof}
		
We are now ready to prove that cohomological blow-up algebras have SLP.

\begin{theorem}
\label{thm:SLPblowup}
Let $\F$ be an infinite field and let $\pi:A\to T$ be a surjective homomorphism of  graded AG $\F$-algebras of socle degrees $d>k$ respectively such that both $A$ and $T$ have SLP.  Assume that characteristic $\F$ is zero or characteristic $F$ is $ p>d$. Then every cohomological blow-up algebra of $A$ along $T$ satisfies SLP.
\end{theorem}

\begin{proof}
	We have seen that $\mathcal{A}$ is a flat family.  By Lemma \ref{lem:specialfiber}, the special fiber $\mathcal{A}_0$ has SLP.  Therefore by semicontinuity of Jordan type \cite[Corollary 2.44]{IMM0}, and since $\F$ is infinite, we deduce that there is some $c\neq 0$ for which the general fiber $\mathcal{A}_c$ has SLP.  By Theorem \ref{thm:flatfam} it follows that the cohomological blow $\tilde{A}=\mathcal{A}_{1}$ has SLP as well. \end{proof}

The following example shows that our assumptions on the characteristic of the field in Theorem \ref{thm:SLPblowup} are necessary.
\begin{example}
	\label{ex:failingSLP}
	Let $d$ and $k$ be integers satisfying ${2k<d}$, and, following the
	usual notation, write ${n=d-k}$. Let $\F$ be a field of characteristic
	$p$, and suppose there is an integer ${m\ge1}$ such that  ${n<p^m\le
		d-2}$. Consider the AG algebras
	\[
	A=\frac{\F[x,y]}{\operatorname{Ann}(X^d+Y^d)}=\frac{\F[x,y]}{\left(xy,x^d-y^d\right)}
	\qquad\text{and}\qquad
	T=\frac{\F[x,y]}{\operatorname{Ann}(X^k)}=\frac{\F[x,y]}{\left(y,x^{k+1}\right)}
	\cong\frac{\F[x]}{\left(x^{k+1}\right)}.
	\]
	The Thom class of the natural surjection ${\pi:A\to T}$ is ${\tau=x^n}$,
	and choosing ${f_R(\xi)=\xi^n-x^n}$ and ${h_R(\xi)=1}$ as in Construction
	\ref{con:hatMD}, we get a cohomological blow-up algebra
	\[
	\tilde{A}=\frac{\F[x,y,\xi]}{\operatorname{Ann}(X^d+Y^d+X^k\Xi^n)}=\frac{\F[x,y,\xi]}{\left(xy,x^d-y^d,y\xi,x^{k+1}\xi,\xi^n-x^n\right)}.
	\]
	We can easily check that ${\ell_A=x+y}$ and ${\ell_T=x}$ are Lefschetz
	elements in $A$ and $T$, respectively. We also know that the Hilbert
	function of $\tilde{A}$ satisfies
	${H\bigl(\tilde{A}\bigr)_{1}=H\bigl(\tilde{A}\bigr)_{d-1}=3}$. This means
	that if $\tilde{A}$ satisfies the SLP, multiplication by ${\ell^{d-2}}$,
	for a general ${\ell\in \tilde{A}_1}$ must have rank $3$. However, writing
	${\ell=ax+by+c\xi}$, we have
	${\ell^{p^m}=a^{p^m}x^{p^m}+b^{p^m}y^{p^m}+c^{p^m}\xi^{p^m}=a^{p^m}x^{p^m}+b^{p^m}y^{p^m}}$,
	becuse ${\xi^{p^m}}$ is zero in $\tilde{A}$. Therefore ${\ell^{p^m}\xi=0}$
	(note that ${p^m>k}$), meaning that multiplication by $\ell^{p^m}$ has
	rank less than $3$, hence so does multiplication by ${\ell^{d-2}}$.
	Therefore $\tilde{A}$ does not satisfy the SLP.
\end{example}

Theorem \ref{thm:SLPblowup} once again singles out cohomological blow-up algebras among connected sums. Connected sums $A\#_\F B$ over a field $\F$ of two strong or weak Lefschetz algebras $A,B$ are strong or weak Lefschetz, respectively; however taking connected sum $A\#_TB$  over an arbitrary AG algebra $T$ may not in general preserve SLP or WLP \cite[Section 5.2]{IMS}.

The following example shows that the converse of Theorem \ref{thm:SLPblowup} is not true: if the cohomological blowup $\hat{A}$ has SLP it does not follow that $A$ has SLP. In other words, while the process of blowing up preserves SLP, the process of blowing down does not preserve SLP, nor even WLP.
\begin{example}
	\label{ex:Rodrigo}
The following example, originally due to U. Perazzo, but re-examined more recently by R. Gondim and F. Russo \cite{GR}, is an AG algebra with unimodal Hilbert function which does not have SLP or WLP:
	$$A=\frac{\F[x,y,z,u,v]}{\operatorname{Ann}(XU^2+YUV+ZV^2)}=\frac{\F[x,y,z,u,v]}{\left(x^2,xy,y^2,xz,yz,z^2,u^3,u^2v,uv^2,v^3,xv,zu,xu-yv,zv-yu\right)}.$$
	Taking the quotient $T$ of $A$ given by the Thom class $\tau=u^2$ yields
	$$T=\frac{\F[x,y,z,u,v]}{\operatorname{Ann}(X)}=\frac{\F[x,y,z,u,v]}{\left(x^2,y,z,u,v\right)}\cong\frac{\F[x]}{(x^2)}.$$
	Fix a parameter $\lambda\in\F$ and define  polynomials $f_T(\xi)\in T[\xi]$ and $f_A(\xi)\in A[\xi]$ by
	$$f_T(\xi)=\xi^2-\lambda x\xi \quad \text{ and } f_A(\xi)=\xi^2-\lambda x\xi+u^2.$$
	Denoting the ideal of relations of $A$ by $I$ we obtain the cohomological blowup 
	$$\tilde{A}=\frac{\F[x,y,z,u,v, \xi]}{I+\xi(y,z,u,v)+(f_A(\xi))}, $$
	which has Hilbert function $H(\tilde{A})=H(A)+H(T)[1]=(1,6,6,1)$.
	Fix $\F$-bases 
	$$\tilde{A}_1=\operatorname{span}_{\F}\left\{x,y,z,u,v, \xi\right\}, \ \ \text{and} \ \ \tilde{A}_2=\operatorname{span}_{\F}\left\{u^2,uv,v^2,yv,yu,-x\xi\right\}$$
	and let $\ell\in\hat{A}_1$ be a general linear form 
	$$\ell=ax+by+cz+du+ev+f\xi.$$
	Then the matrix for the Lefschetz map $\times\ell\colon \hat{A}_1\rightarrow\hat{A}_2$ and its determinant are given by
	$$M = \left(\begin{array}{cccccc}
	0 & 0 & 0 & d & 0 & -f\\
	0 & 0 & 0 & e & d & 0\\
	0 & 0 & 0 & 0 & e & 0\\
	d & e & 0 & a & b & 0\\
	0 & d & e & b & c & 0\\
	-f & 0 & 0 & 0 & 0 & -(a+\lambda f)\\
	\end{array}\right)\Rightarrow \det(M)=f^2e^4.$$ 
Thus $\ell$ is a strong Lefschetz element  for $\tilde{A}$ if and only if $e\cdot f\neq 0$.  In particular $\tilde{A}$ satisfies SLP and also WLP.
\end{example}

Surprisingly, the analogous result to Theorem \ref{thm:SLPblowup} does not hold for the weak Lefschetz property.  The obstruction to establishing such a result via an analogue of Lemma \ref{lem:specialfiber} is that  the tensor product of two weak Lefschetz algebras need not be weak Lefschetz. Examples of non-AG quadratic algebras which demonstrate this are given in \cite[\S 4.1]{MNS}. Here we point out an AG example. The AG algebra $T$ in Example \ref{8.7ex} below is an example of R. Gondim, quoted as \cite[Example 3.4]{IMM0}. Using the Clebsch-Gordan Theorem \cite[Theorem 3.29]{H-W}, one can see that if $\cha \F=0$ or $\cha \F\geq 7$, then the tensor product $T\otimes B$, where $B= \F[\xi]/(\xi^2)$, has Jordan type (for generic linear form) the tensor product of $P_T=(5,3,3,3,2,2)$ and $P_B=(2)$, which is $(6,4^4,3^2,2^3,1^2)$ with 12 parts; on the other hand $H(T\otimes B)=(1,6,11,11,6,1)$, hence $H(T\otimes B)^\vee=(6,4^5,2^5)$ has so $T\otimes B$ is not weak Lefschetz.

Building upon this, we give an example illustrating that blowing up does not preserve WLP.
	
	\begin{example}\label{8.7ex}
	Consider the following algebra 
	$$A=\frac{\F[x,y,z,u,v]}{\operatorname{Ann}(XU^6+YU^4V^2+ZU^5V)}=\frac{\F[x,y,z,u,v]}{\left(yz,xz,xy,vy-uz,vx,ux-vz,u^5y,u^5v^2,u^6v,u^7,v^3,x^2,y^2,z^2\right)}$$
with $H(A)=(1, 5, 6, 6, 6, 6, 5, 1)$	and its quotient  corresponding to the Thom class $\tau=u^3$
		$$T=\frac{\F[x,y,z,u,v]}{\operatorname{Ann}(XU^3+YUV^2+ZU^2V)}=\frac{\F[x,y,z,u,v]}{\left(z^2,yz,xz,y^2,xy,vy-uz,x^2,vx,ux-vz,u^2y,v^3,u^2v^2,u^3v,u^4 \right)}$$
		with $H(T)=(1, 5, 6, 5, 1)$. Both $A$ and $T$ satisfy WLP, but not SLP: for a generic $\ell$ the Jordan types are $P_{A,\ell}=(8,6,6,6,5,5)$, and $P_{T,\ell}=(5,3,3,3,2,2)$.
	Denoting the ideal of relations of $A$ by $I$ and the ideal of relations of $T$ by $K$, consider the cohomological blow-up algebra
	$$\tilde{A}=\frac{\F[x,y,z,u,v, \xi]}{I+\xi\cdot K+(\xi^3-u^3)}, $$
	which has Hilbert function $H(\tilde{A})=H(A)+H(T)[1]+H(T)[2]=(1, 6, 12, 17, 17, 12, 6, 1)$ and dual Macaulay generator $\tilde{F}= \Xi^3G+\Xi^6X+F$, here $\lambda=-1$ in the Construction \ref{con:hatMD}.
		Fix the following bases for the 17-dimensional vector spaces
\begin{eqnarray*}
		\tilde{A}_3 &=& \operatorname{span}_{\F}\left\{x\xi^2, yu^2, yuv, yu\xi, yv^2, yv\xi, y\xi^2, zv\xi, z\xi^2, u^2v, u^2\xi, uv^2, uv\xi, u\xi^2, v^2\xi, v\xi^2, \xi^3\right\}\\
\tilde{A}_4 &=& \operatorname{span}_{\F}\left\{x\xi3, yuv\xi, yu\xi^2, yv2\xi, yv\xi^2, y\xi^3, zv\xi^2, z\xi^3, u^2v^2, u^2v\xi, u^2\xi^2, uv^2\xi, uv\xi^2, u\xi^3, v2\xi^2, v\xi^3, \xi^4\right\}
\end{eqnarray*}
	Taking a general linear form $\ell\in\tilde{A}_1$ as 
	$\ell=ax+by+cz+du+ev+f\xi$
	we compute the matrix for the Lefschetz map $\times\ell\colon \tilde{A}_3\rightarrow\hat{A}_4$ and its determinant, which shows $\tilde{A}$ does not satisfy WLP
	$$
	\small
	M =
	{\left({\begin{array}{ccccccccccccccccc}
      f&0&e&0&d&0&0&0&0&c&0&b&0&0&0&0&a\\
      0&0&f&e&0&d&0&0&0&0&c&0&b&0&0&0&0\\
      0&0&0&f&0&0&d&0&0&0&0&0&0&b&0&0&0\\
      0&0&0&0&f&e&0&d&0&0&a&0&c&0&b&0&0\\
      0&0&0&0&0&f&e&0&d&0&0&0&0&c&0&b&0\\
      0&d&0&0&0&0&f&0&0&0&0&0&0&0&0&0&b\\
      d&0&0&0&0&0&0&f&e&0&0&0&0&a&0&c&0\\
      0&e&d&0&0&0&0&0&f&b&0&0&0&0&0&0&c\\
      0&0&0&0&0&0&0&0&0&e&0&d&0&0&0&0&0\\
      0&0&0&0&0&0&0&0&0&f&e&0&d&0&0&0&0\\
      0&0&0&0&0&0&0&0&0&0&f&0&0&d&0&0&0\\
      0&0&0&0&0&0&0&0&0&0&0&f&e&0&d&0&0\\
      0&0&0&0&0&0&0&0&0&0&0&0&f&e&0&d&0\\
      0&0&0&0&0&0&0&0&0&0&0&0&0&f&0&0&d\\
      0&0&0&0&0&0&0&0&0&0&0&0&0&0&f&e&0\\
      0&0&0&0&0&0&0&0&0&d&0&0&0&0&0&f&e\\
      0&0&0&0&0&0&0&0&0&0&d&0&0&0&0&0&f\\
      \end{array}}\right)}
	\Rightarrow \det(M)=0.$$
	A Macaulay2 \cite{M2} calculation gives the generic Jordan type $P_{\tilde{A}}=(8,6^5,4^6,2^4,1^2)$ with $18$ parts, whereas the conjugate partition of the Hilbert function is
	$H(\tilde{A})^\vee=(8,6^5,4^6,2^5)$ with $17$ parts, which implies that $\tilde{A}$ has neither SLP nor WLP. 
	\end{example}
	
In Example \ref{8.7ex} the Thom class of the map $A\to T$ has degree 3. This is the minimal possible value for such an example based on the following result.

\begin{theorem}
\label{prop:codim2}
Let $\F$ be an infinite field and let $\pi:A\to T$ be a surjective homomorphism of  graded AG $\F$-algebras such that the difference between the socle degrees of $A$ and $T$ is at most 2 and $A$ and $T$ both satisfy WLP. Then every cohomological blow-up algebra of $A$ along $\pi$ satisfies WLP.
\end{theorem}
\begin{proof}
The hypothsis on the socle degreed of $A$ and $T$ translates into $n\leq 2$.
The proofs of Lemma \ref{lem:specialfiber} in the spacial cases $J=0$ (for $n=1$) or $J=T$ (for $n=2$) and Theorem \ref{thm:SLPblowup} go through upon replacing SLP by WLP throughout.
\end{proof}

\section{Geometric View and Examples}\label{geomsec}
As in the Introduction, our motivation for studying the cohomological blow-up algebras stems from the blow up construction in algebraic geometry.  That the cohomology ring of the blow up of a compact complex manifold along a closed complex submanifold satisfies the conditions of Theorem \ref{lem:buchar} can be pieced together from results in the book of P. Griffiths and J. Harris \cite[Chapter 4, Section 6]{GH}; see also the paper of D. McDuff \cite[Proposition 2.4]{McDuff} for the statements in the symplectic category.  The cohomology of the blow up was also studied by S. Gitler who has obtained a presentation, as in Equation \eqref{eq:cohomologybu}, of the cohomology of the blow up of a complex manifold along a complex submanifold in the case of a surjective restriction map as in Equation \eqref{eq:cohomologybu} \cite[Theorem 3.11]{Gitler}.  An analogous presentation for the Chow ring of the blow up of an algebraic variety along a regularly embedded subvariety with a surjective restriction map also appears in the paper of S. Keel \cite[Appendix, Theorem 1]{Keel}.

Blowing up and blowing down are fundamental building blocks in birational geometry.  For example, in the theory of algebraic surfaces, a classical result states that every birational map between algebraic surfaces (smooth complex projective variety of dimension two) admits a strong factorization, meaning it factors as a sequence of blow ups followed by a sequence of blow downs \cite[Corollary II.12]{Beauville}.  More recently D. Abramovich et al. have proved the weak factorization conjecture, a higher dimensional analogue which states that every birational map between complete non-singular algebraic varieties over an algebraically closed field of characteristic zero factors as a product of blow ups and blow downs (in no specific order) \cite[Theorem 0.1.1]{AKMW}.  In the theory of (smooth, projective) toric varieties, one can be quite explicit with these factorizations by working with the associated (simplicial, polytopal) fan, where blowing up along a toric subvariety corresponds to subdividing a cone of the fan \cite[Proposition 3.3.15]{CLS}.  As we have seen in Section \ref{Lefsec} concerning the strong Lefschetz Property, and will show in a sequel \cite{MMSW} concerning the Hodge-Riemann bilinear relations (HRR), factorizations in terms of blow ups and blow downs can be useful in establishing SLP and HRR. 

P. McMullen \cite{McMullen} in his proof of SLP and HRR for the polytope algebra, an AG algebra which he shows is isomorphic to a certain Artinian reduction of the Stanley-Reisner ring of the corresponding simplicial polytopal fan, has given an explicit formula for a (weak) factorization of the birational map $\P^n\supset(\C^*)^n\rightarrow X$ where $X$ is a smooth projective toric variety of dimension $n$; see also V. A. Timorin \cite{Timorin} for an exposition from the point of view of Macaulay duality.  K. Karu \cite{Karu} used similar arguments in his proof of SLP and HRR for non-simplicial polytopal fans; see also \cite{BBFK,BL}.  More recently F. Ardila, G. Denham, and J. Huh \cite{ADH} have exploited blow up factorizations to prove SLP and HRR for the Chow ring of the Bergman fan associated to a matroid. Especially important in their work is the special case of blow ups which correspond to edge subdivisions in these Bergman fans, and in fact they factor any blow up as a sequence of successive edge subdivison blow ups and blow downs. 
Geometrically, these edge subdivisions correspond to cohomological blow ups with $n=2$. Theorem \ref{prop:codim2} indicates that this $n=2$ scenario is particularly favorable in terms of ascent of the weak Lefschetz property to the cohomological blow-up algebra but it does not necessarily guarantee the descent of WLP or SLP to the blow-down, see Example \ref{ex:Rodrigo}.

Algebraically, one might say that two oriented graded AG algebras $A$ and $A'$ are \emph{birationally equivalent} if there is a sequence of oriented graded AG algebras $(A_0,A_1,\ldots,A_m)$ of some fixed socle degree $d$, where $A_0=A$ and $A_m=A'$ and for each $i$, $A_i$ is either a cohomological blow up or a cohomological blow down of $A_{i-1}$.  We give an example below of several birationally equivalent AG algebras using fans corresponding to smooth projective toric surfaces; in fact we derive a (strong) factorization of the birational map $\P^2\dashrightarrow \P^1\times\P^1$.  First we recall a few fundamental facts on toric varieties and their associated fans; for further details we refer the reader to the book \cite{CLS}.  In these geometric examples, as in the Introduction, we take cohomology with coefficients over the rationals $\F=\Q$.

A presentation for the cohomology ring of a complete simplicial toric variety can be obtained from the associated fan in the following explicit manner.
 Let $\Sigma$ be a complete simplicial fan and let $X$ be the corresponding toric variety. Let $\rho_1, \ldots, \rho_r$ be the rays of $\Sigma$, each $\rho_i$ having minimal generator $u_i$. Introduce a variable $x_i$ for each $\rho_i$. In the ring $\mathbb{Q}[x_1,\ldots, x_r]$, let $I$ be the square-free monomial ideal
\[
 I = (x_{i_1}\cdots x_{i_s} \mid i_1<\cdots <i_s \text{ and } \rho_{i_1},\ldots, \rho_{i_s} \text{ are not part of the same cone of }\Sigma).
 \]
We call $I$ the Stanley-Reisner ideal of $\Sigma$. Let $J$ be the ideal generated by the linear forms
\[
J=\left (\sum_{i=1}^r\langle m, u_i\rangle x_i\mid m\in \sum_{i=1}^r u_i\mathbb{Z} \right).
\]
By  \cite[Theorem 12.4.1]{CLS} the singular cohomology ring of $X$  can be presented as  
\begin{equation}
\label{eq:tvcohom}
H^{{2\bullet}}(X,\mathbb{\Q}) \cong \frac{\mathbb{\Q}[x_1,\ldots, x_r]}{I +J}.
\end{equation}

The following example shows that the oriented graded AG algebras $A_1=\Q[x]/(x^3)$, $A_2=\Q[y,z]/(y^2,z^2-yz)$, $A_3=\Q[r,s,t]/(rt,st,r^2,s^2,t^2+rs)=A_4$, and $A_5=\Q[u,v]/(u^2,v^2)$ are all birationally equivalent to one another.  This corresponds to the well known fact from algebraic geometry that if $\mathbb{P}^2$ is blown up at two points and the proper transform of the line joining the two points is blown down, the resulting surface is isomorphic to $\mathbb{P}^1\times \mathbb{P}^1$. We work through the details of this example from the perspective of cohomological blow ups.

\begin{example}\label{toricex}

The algebraic varieties featured in this example correspond to the following fans
\smallskip

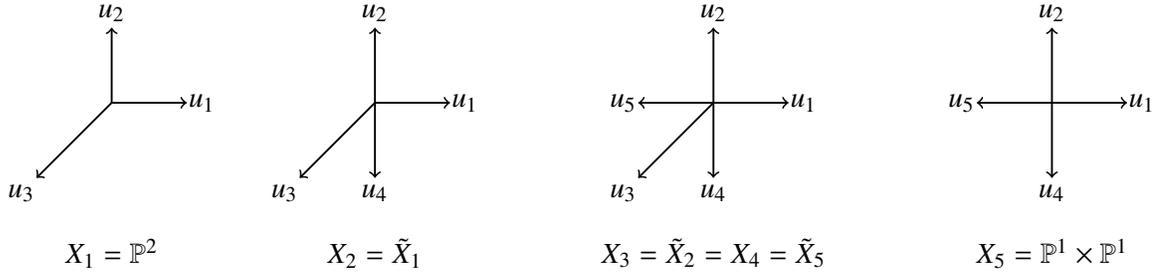
\begin{tikzpicture}
\begin{scope}[shift = {(-0.5,0)}]
\draw[line width=0.25mm, -to] (0,0) -- (1,0);
\draw[line width=0.25mm, -to] (0,0) -- (0,1);  
\draw[line width=0.25mm, -to] (0,0) -- (-1,-1);  
\node at (1.2,0) {$u_1$};
\node at (0,1.2) {$u_2$};
\node at (-1.2,-1.2) {$u_3$};
\node at (0,-2) {$X_1=\mathbb{P}^2$};
\end{scope}

\begin{scope}[shift = {(3,0)}]
\draw[line width=0.25mm, -to] (0,0) -- (1,0);
\draw[line width=0.25mm, -to] (0,0) -- (0,1);  
\draw[line width=0.25mm, -to] (0,0) -- (-1,-1);  
\draw[line width=0.25mm, -to] (0,0) -- (0,-1);  
\node at (1.2,0) {$u_1$};
\node at (0,1.2) {$u_2$};
\node at (-1.2,-1.2) {$u_3$};
\node at (0,-1.2) {$u_4$};
\node at (0,-2) {$X_2=\tilde{X}_1$};
\end{scope}

\begin{scope}[shift = {(7.5,0)}]
\draw[line width=0.25mm, -to] (0,0) -- (1,0);
\draw[line width=0.25mm, -to] (0,0) -- (0,1);  
\draw[line width=0.25mm, -to] (0,0) -- (-1,-1);  
\draw[line width=0.25mm, -to] (0,0) -- (-1,0);
\draw[line width=0.25mm, -to] (0,0) -- (0,-1);  
\node at (1.2,0) {$u_1$};
\node at (0,1.2) {$u_2$};
\node at (-1.2,-1.2) {$u_3$};
\node at (-1.2,0) {$u_5$};
\node at (0,-1.2) {$u_4$};
\node at (0,-2) {$X_3=\tilde{X}_2=X_4=\tilde{X}_5$};
\end{scope}

\begin{scope}[shift = {(12,0)}]
\draw[line width=0.25mm, -to] (0,0) -- (1,0);
\draw[line width=0.25mm, -to] (0,0) -- (0,1);  
\draw[line width=0.25mm, -to] (0,0) -- (-1,0);
\draw[line width=0.25mm, -to] (0,0) -- (0,-1);  
\node at (1.2,0) {$u_1$};
\node at (0,1.2) {$u_2$};
\node at (-1.2,0) {$u_5$};
\node at (0,-1.2) {$u_4$};
\node at (0,-2) {$X_5=\mathbb{P}^1\times \mathbb{P}^1$};
\end{scope}
\end{tikzpicture}

As a toric variety, $X_1=\mathbb{P}^2$ is defined by the fan $\Sigma$ with ray generators $u_1 =e_1, u_1=e_2, u_3=-e_1-e_2$, where $e_1=(1,0)$ and $e_2=(0,1)$ are the standard basis vectors; see \cite[Example 12.4.2]{CLS}. According to the formula above its cohomology algebra is  
\[
A_1=H^{2\bullet}(X_1) \cong \frac{\mathbb{Q}[x_1,x_2,x_3 ]}{(x_1x_2x_3, x_1-x_3, x_2-x_3)}\cong  \frac{\mathbb{Q}[x]}{(x^3)}.
\]

The blow up $X_2=\tilde{X}_1$ of $X_1=\mathbb{P}^2$ at a point  is obtained by subdividing a cone of $\Sigma$ by adding the ray generated by $u_4=-e_2$. Then 
\begin{equation}
\label{eq:X1}
A_2=H^{2\bullet}(X_2)  \cong  \frac{\mathbb{Q}[x_1,x_2,x_3,x_4 ]}{(x_1x_3, x_2x_4, x_1-x_3, x_2-x_3-x_4)}=\frac{\Q[x_1,x_2]}{(x^2_1,x_2^2-x_2x_1)}.
\end{equation}

In terms of cohomological blow ups, take $T_1=\mathbb{Q}$ and let $\pi\colon A_1\rightarrow T_1$ be the natural projection, with Thom class $\tau_1=x^2$ and kernel $K_1=(x)$. The normal bundle of a point in $\P^2$ has total Chern class $c=1$, and hence $f_{T_1}(\xi)=\xi^2$ and $f_{A_1}(\xi)=\xi^2+x^2$, and hence the cohomological blowup of $A$ along $\pi_1$ is 
\begin{equation}
\label{eq:X1'}
\tilde{A}_1=\frac{\mathbb{Q}[x, \xi]}{(x^3, x\xi, \xi^2+ x^2)}= \frac{\mathbb{Q}[x, \xi]}{(x\xi, \xi^2+x^2)}\cong \frac{\Q[x_1,x_2]}{(x_1^2,x_2^2-x_1x_2)}=A_2
\end{equation}
where the last isomorphism sends $x\mapsto x_2$ and $\xi\mapsto x_1-x_2$.

Next define $X_3=\tilde{X}_2$ as the blow up of $X_2$ obtained by adding the ray generated by $u_5=-e_1$.  Then its cohomology algebra is given by  
\begin{align*}
A_3=H^{2\bullet}(X_3)  = & \frac{\mathbb{Q}[x_1,x_2,x_3,x_4, x_5 ]}{(x_1x_3, x_1x_5, x_2x_3, x_2x_4, x_4x_5, x_1-x_3-x_5, x_2-x_3-x_4)}\\
= & \frac{\Q[x_1,x_2,x_3]}{(x_1x_3, x_1^2,x_2x_3,x_2^2,x_3^2+x_1x_2)}.
\end{align*}
Again let $T_2=\Q$ and $\pi_2\colon A_2\rightarrow T_2$ be the canonical projection, with Thom class $\tau_2=x_1x_2$ and kernel $K=(x_1,x_2)$.  Then we have $f_{A_2}(\xi)=\xi^2+x_1x_2$ and the cohomological blow up of $A_2$ along $\pi_2$ is 
$$\tilde{A}_2=\frac{\Q[x_1,x_2,\xi]}{(x_1^2,x_2^2-x_2x_1,\xi x_1, \xi x_2, \xi^2+x_1x_2)}\cong \frac{\Q[x_1,x_2,x_3]}{(x_1x_3, x_1^2,x_2x_3,x_2^2,x_3^2+x_1x_2)}=A_3$$
where the last isomorphism sends $x_1\mapsto x_1$, $x_2\mapsto x_2-x_3$ and $\xi\mapsto x_3$. 

Finally let $X_5=\P^1\times \P^1$, defined by the fan $\Sigma'$ with ray generators $u_1=e_1$, $u_2=e_2$, $u_4=-e_2$, and $u_5=-e_1$, and let $X_4=\tilde{X}_5$ be the blow up at a point obtained by adding the ray to $\Sigma'$ $u_3=-e_1-e_2$.  Since the fans for $X_4$ and $X_3$ are identical, it follows that the toric varieties coincide as well.  Hence their cohomology algebras are 
$$A_5=H^{2\bullet}(\P^1\times\P^1)=\frac{\mathbb{Q}[x_1,x_2,x_4,x_5 ]}{(x_1x_5, x_2x_4, x_1-x_5, x_2-x_4)}=\frac{\Q[x_1,x_2]}{(x_1^2,x_2^2)}$$
and 
$$A_4=A_3=H^{2\bullet}(X_4)=H^{2\bullet}(X_3)=\frac{\Q[x_1,x_2,x_3]}{(x_1x_3, x_1^2,x_2x_3,x_2^2,x_3^2+x_1x_2)}\cong \frac{\Q[x_1,x_2,\xi]}{(x_1^2,x_2^2,x_1\xi, x_2\xi, \xi^2+x_1x_2)}=\tilde{A}_5$$
where the last isomorphism is the obvious $x_1\mapsto x_1$, $x_2\mapsto x_2$, and $x_3\mapsto \xi$.

It follows that the oriented graded AG algebras $A_1$, $A_2$, $A_3=A_4$, and $A_5$ are all birationally equivalent, corresponding to the (strong) factorization of the birational map
$$\xymatrixrowsep{.1pc}\xymatrix{& & X_3=X_4\ar[ddr]\ar[dl]\\
	& X_2 \ar[dl] & \\
	X_1=\P^2\ar@{-->}[rrr] &&& \P^1\times\P^1=X_5\\}.$$
\end{example}

\begin{remark} 
	\label{rem:toricsurj}
	The cohomology algebra of the blow up of a smooth toric variety $X$ along a smooth torus invariant subvariety $Y\subset X$ will always agree with the Constructions \ref{con:hat}, \ref{con:hatMD}, or \ref{con:ideal} of this paper since in that case the restriction map $\pi^*\colon H^{2\bullet}(X)\rightarrow H^{2\bullet}(Y)$ is always surjective.  Indeed in that case, the associated fan of $Y$ corresponds to a subfan of $X$, and the surjectivity follows from the combinatorial presentation of cohomology algebras as in \eqref{eq:tvcohom}.  The following examples shows what can happen in cases where that restriction map is not surjective.
\end{remark}

\begin{example}
	\label{ex:Segre1}
		Let $\pi\colon Y=\P^1\times\P^2\hookrightarrow \P^5=X$ be the Segre embedding.  Then we have a short exact sequence of vector bundles on $Y$:
		$$\xymatrix{0\ar[r] & \mathcal{T}_Y\ar[r] & \pi^*\mathcal{T}_X\ar[r] & \mathcal{N}_{Y/X}\ar[r] & 0}$$
		where $\mathcal{T}_Y$ is the tangent bundle of $Y$, $\pi^*\mathcal{T}_X$ is the restriction of the tangent bundle of $X$ to $\pi(Y)$ and $\mathcal{N}_{Y/X}$ is the normal bundle to $\pi(Y)\subset X$.  If we identity the cohomology algebras as the oriented graded AG algebras  
		$$A=H^{2\bullet}(X)\cong \frac{\Q[x]}{(x^6)}, \ \ T=H^{2\bullet}(\P^1\times\P^2)\cong \frac{\Q[y,z]}{(y^2,z^3)}$$
		where $x$, $y$, and $z$ are the classes of a hyperplane in $H^{2}(\P^5)$, and the factors $H^2(\P^1)$, and $H^2(\P^2)$ of $H^{2\bullet}(\P^1\times\P^2)\cong H^{2\bullet}(\P^1)\otimes_{\Q}H^{2\bullet}(\P^2)$, then the induced map $\pi^*\colon A\rightarrow T$ satisfies $\pi^*(x)=y+z$.  Note that $\pi^*$ is not surjective here.  From the Euler sequence we compute the total Chern classes $c(\pi^*\mathcal{T}_X)=(1+\pi^*(x))^6$ and $c(\mathcal{T}_Y)=(1+y)^2\cdot (1+z)^3$.  It follows from the Whitney product formula that the total Chern class for the normal bundle is 
		\begin{eqnarray*}
		c(\mathcal{N}_{Y/X})=\frac{c(\pi^*\mathcal{T}_X)}{c(\mathcal{T}_Y)}= & \frac{\left(1+y+z\right)^6}{(1+y)^2\cdot(1+z)^3} & =\left(1+y+z\right)^6\cdot \left(1-y\right)^2\cdot \left(1-z+z^2\right)^3\\
		& & =(6yz+3z^2)+(4y+3z)+1
	\end{eqnarray*}
		and hence the Chern classes are 
		$$\begin{cases}
		c_1(\mathcal{N}_{Y/X})= & 4y+3z\\
		c_2(\mathcal{N}_{Y/X})= 
		& 6yz+3z^2.\\
		\end{cases}$$		
		
		Hence if we blow up $X$ along $Y$ then, according to Equation \eqref{eq:HYtild}, the cohomology algebra of the exceptional divisor $\tilde{Y}$ is given by 
		$$\tilde{T}= \frac{\Q[y,z,\xi]}{(y^2,z^3,\xi^2-(4y+3z)\xi+(6yz+3z^2))}\cong H^{2\bullet}(\tilde{Y}).$$
		Moreover using the conditions of Theorem \ref{lem:buchar}, we can derive a presentation of the cohomology algebra of the blow up manifold $\tilde{X}$:
		\begin{equation}
		\label{eq:segre}
		\tilde{A}=\frac{\Q[x,\xi]}{\left(\xi^3-6x\xi^2+12x^2\xi-8x^3, 3\xi^4-9x\xi^3+6x^2\xi^2+4x^3\xi\right)}\cong H^{2\bullet}(\tilde{X}).
		\end{equation}
		where the restriction map $\tilde{\pi}\colon \tilde{A}\rightarrow \tilde{T}$ defined by $\tilde{\pi}(x)=y+z$ and $\tilde{\pi}(\xi)=\xi$ has Thom class $\tilde{\tau}=-\xi$, the blow up map $\beta\colon A\rightarrow \tilde{A}$ defined by $\beta(x)=x$ is injective with $\beta(a_{soc})=\tilde{a}_{soc}=x^5$, and the Hilbert function satisfies $H(\tilde{A})=H(A)+H(T)[1]=(1,2,3,3,2,1)$.  Furthermore, a Macaulay2 \cite{M2} calculation computes the Macaulay dual generator of $\tilde{A}$ as
		$$\tilde{F}=X^5-3X^3\Xi^2-10X^2\Xi^3-24X\Xi^4-48\Xi^5.$$
		Note that $\tilde{A}$ in \eqref{eq:segre} does not fit the model described by our Construction \ref{con:hat}; in particular the defining ideal of $\tilde{A}$ does not contain any monic polynomial of degree $n=2$. \end{example}

\begin{remark}
	\label{rem:nostdgrad}
	The algebra $\tilde{A}$ computed in Equation \eqref{eq:segre} might be termed a cohomological blow up along the \emph{non-surjective} map $\pi$.  In that case $A$ and $T$ are both standard graded, and $\tilde{A}$ is too, but this need not hold in general.  For example, if we blow up $X=\P^8$ along the Segre embedding of $Y=\P^2\times\P^2$, then again we have a non-surjective restriction map 
	$$\pi\colon A=H^{2\bullet}(X)=\Q[x]/(x^9)\rightarrow \Q[y,z]/(y^3,z^3)=H^{2\bullet}(Y)=T$$
	but the cohomology of the blow up of $X$ along $Y$, $\tilde{A}\cong H^{2\bullet}(\tilde{X})$ has Hilbert function 
	$$H(\tilde{A})=H(A)+H(T)[1]+H(T)[2]+H(T)[3]=(1, 2, 4, 7, 8, 7, 4, 2, 1),$$ 
	which implies that $\tilde{A}$ cannot be standard graded. 
\end{remark}
Motivated by these examples, we pose some problems for further research.
\begin{problem}
	\label{prob:birat}
	Generalize Example \ref{toricex} and find other algebras which are birationally equivalent to $A_0=\F[x]/(x^{d+1})$.  Can one classify them? 
\end{problem}

\begin{problem}
	\label{prob:nonsurjbu}
	Generalize Example \ref{ex:Segre1} and find a construction, similar in spirit to Construction \ref{con:hat}, for a cohomological blow up of an AG algebra $A$ along \emph{any} (i.e. possibly non-surjective) restriction map $\pi\colon A\rightarrow T$.  Does it have similar properties as the cohomological blow up along a surjective map, i.e. flat family, strong Lefschetz, connected sum, minimal generators?
\end{problem}

\begin{appendices}
\section{A Guide To Our Examples}
Given below is a list of the examples in this paper together with a brief description of the idea that example is attempting to illustrate.
\begin{enumerate}
	\item Example \ref{ex:notGor} shows that without further qualifications, Construction \ref{con:hat} can produce non-Gorenstein, Gorenstein, or boundary-Gorenstein algebras, i.e. non-Gorenstein algebras in the closure of the Gorenstein locus of the Hilbert scheme of that Hilbert function.

	\item Example \ref{ex:lambda} shows that over non-algebraically closed fields, distinct parameter values $\lambda$ can produce non-isomorphic cohomological blow up algebras.
	
	\item Example \ref{ex:NotPBI} shows that $\tilde{T}=R[\xi]/\Ann \left(h_R(\xi)\circ\left(\Xi^{d-1}\cdot G\right)\right)$ is not necessarily a free extension over $T=R/\Ann(G)$ for any choice of $h_R(\xi)\in R[\xi]$.  
	
	\item Example \ref{ex:chris} provides an algebra $\hat{A}_{MD}$ from Construction \ref{con:hatMD} which is not a cohomological blow up. 
	
	\item Example \ref{ex:blowup1} shows a cohomological blow up as a connected sum. 
	
	\item Example \ref{ex:13631} shows a cohomological blow down of Hilbert function $H(A)=(1,3,6,3,1)$ as a connected sum.  
	
	\item Example \ref{ex:nonst} shows that the cohomological blow up $\hat{A}$ may be standard graded even if $A$ is not.

	\item Example \ref{ex:cibu} gives a cohomological blow up that is a complete intersection.
	
	\item Example \ref{ex:zeroH} gives ideals $I$ and $(I:\tau)$ with homology groups $H$ and $H'$ equal to zero, but where the cohomogical blow up ideal $\hat{I}$ is not generated by a regular sequence. 
	
	\item Example \ref{ex:exact} gives examples of exact pairs of zero divisors. 
	
	\item Example \ref{ex:Tony} gives a CI with exact zero divisors which is not a BUG.  
	
	\item  Example \ref{ex:HF} shows a compressed AG algebra of socle degree $5$ and embedding dimension $3$ can be a cohomological blow up of a standard graded AG algebra.
	
	\item Example \ref{ex:failingSLP} shows if $A$ and $T$ have SLP over $\F_p$, then $\hat{A}$ may fail SLP.
	
	\item Example \ref{ex:Rodrigo} gives algebras in which $\hat{A}$ and $\hat{T}$ both have SLP, but the cohomological blow down $A$ does not have SLP.  
	
	\item Example \ref{8.7ex} shows that if $A$ and $T$ both have WLP (but fail SLP), then $\hat{A}$ may fail WLP.
	
	\item Example \ref{toricex} gives a geometric example of a strong factorization of a birational map between toric varieties which yields several birationally equivalent AG algebras, i.e. a sequence of AG algebras in which each one is either a cohomological blow up or blow down of the previous.
	
	\item Example \ref{ex:Segre1} computes a presentation of the blow up of $\P^5$ along the Segre embedding $\P^1\times\P^2\hookrightarrow \P^5$ in which case the restriction map on cohomology is not surjective. 

\end{enumerate}
\end{appendices}

\begin{ack} The authors are grateful to the series of annual Lefschetz Properties In Algebra and Combinatorics conferences and workshops, some of which we each participated in, beginning with G\"{o}ttingen (2015), where the fifth author proposed the Bold Conjecture \ref{boldconj}, and followed by
meetings at Mittag Leffler (2017),  Levico (2018), CIRM Luminy (2019), and Oberwolfach (2020).

The second author was 
partially supported by CIMA -- Centro de Investiga\c{c}\~{a}o em 
Matem\'{a}tica e Aplica\c{c}\~{o}es, Universidade de \'{E}vora, project 
UIDB/04674/2020 (Funda\c{c}\~{a}o para a Ci\^{e}ncia e Tecnologia).
The fourth author is supported by NSF grant DMS--2101225.

\end{ack}

\end{document}